\input ./style/arxiv-general.cfg
\documentclass[aop,MSNbibl,seceqn,dvips]{arximspdf}
\makeatletter
   \@ifpackageloaded{graphicx}{}{\usepackage{graphicx}}
\makeatother
\usepackage{mathbh}

\doi{10.1214/14-AOP939} 
\volume{43}
\issue{5}
\pubyear{2015}
\firstpage{2405}
\lastpage{2457}
\docsubty{FLA}

\makeatletter
\def\xrightarrow{\mathop{\longrightarrow}\limits^}
\newcommand{\eqref}[1]{(\ref{#1})}
\def\mathbbm{\mathbh}
\newcommand{\R}{\mathbb{R}}
\newcommand{\E}{\mathbb{E}}
\newcommand{\Pb}{\mathbb{P}}
\newcommand{\Z}{\mathbb{Z}}
\newcommand{\N}{\mathbb{N}}
\newcommand{\1}{\mathbbm{1}}

\def\epsk{{\tilde\varepsilon_k}}

\def\laweq{\stackrel{\mathrm{law}}=}
\def\FIN{\mathrm{FIN}}
\def\FK{\mathrm{FK}}
\def\GFIN{\mathrm{SSBM}}
\def\IIC{\mathrm{IIC}}
\def\IPC{\mathrm{IPC}}
\def\Max{\mathrm{max}}
\def\comb{\mathrm{comb}}

\def\Leb{{\operatorname{Leb}}}
\def\supp{\mathop{\operatorname{supp}}}
\def\Id{{\mathrm{Id}}}

\newtheorem{theorem}{Theorem}[section]
\newtheorem{lemma}[theorem]{Lemma}
\newtheorem{proposition}[theorem]{Proposition}

\newtheorem{theoremm}{Theorem}[section]
\newtheorem{propositionn}[theoremm]{Proposition}

\newproclaim{definition}[theorem]{Definition}
\newproclaim{observation}[theorem]{Observation}
\newproclaim{example}[theorem]{Example}
\newproclaim{assumption}{Assumption}

\newproclaim{remark}[theorem]{Remark}

\makeatother

\begin{document}
\begin{frontmatter}

\title{Randomly trapped random walks}
\runtitle{Randomly trapped random walks}

\begin{aug}
\author[A]{\fnms{G\'erard} \snm{Ben Arous}\ead[label=e1]{benarous@cims.nyu.edu}},
\author[B]{\fnms{Manuel} \snm{Cabezas}\thanksref{T1}\corref{}\ead[label=e2]{mncabeza@mat.puc.cl}},
\author[C]{\fnms{Ji\v{r}\'{\i}}~\snm{\v{C}ern\'{y}}\ead[label=e3]{jiri.cerny@univie.ac.at}}
\and
\author[A]{\fnms{Roman}~\snm{Royfman}\ead[label=e4]{roman.royfman@gmail.com}}
\runauthor{Ben Arous, Cabezas, \v Cern\'y, Royfman}
\thankstext{T1}{Supported by Iniciativa Cient\'{i}fica Milenio NC120062 and Grant CNPq-PDJ 150897/2012-0.}
\affiliation{New York University, Instituto de Matem\'atica Pura e
Aplicada, University of Vienna and New York University}
\address[A]{G. Ben Arous\\
R. Royfman\\
Courant Institute\\
\quad of Mathematical Sciences\\
New York University\\
251 Mercer Street\\
New York, New York 10012\\
USA\\
\printead{e1}\\
\phantom{E-mail:\ }\printead*{e4}}%
\address[B]{M. Cabezas\\
Instituto de Matem\'atica Pura e Aplicada\\
Estrada Dona Castorina 110\\
Rio de Janeiro\\
Brazil\\
\printead{e2}}
\address[C]{J. \v{C}ern\'{y}\\
Faculty of Mathematics\\
University of Vienna\\
Oskar-Morgenstern-Platz 1\\
1090 Vienna\\
Austria\\
\printead{e3}}
%
\end{aug}

\received{\smonth{3} \syear{2013}}
\revised{\smonth{2} \syear{2014}}


\begin{abstract}
We introduce a general model of trapping for random walks on graphs. We
give the possible scaling limits of these \textit{Randomly Trapped
Random Walks} on $\Z$. These scaling limits include the well-known
fractional kinetics process, the Fontes--Isopi--Newman singular diffusion
as well as a new broad class we call \textit{spatially subordinated
Brownian motions}. We give sufficient conditions for convergence and
illustrate these on two important examples.
\end{abstract}

%
%
%
%
\end{frontmatter}

\section{Introduction}\label{sec1} 

We present here a general class of trapping mechanisms for random walks.
This class includes the usual ``effective'' models of trapping, from the
Continuous Time Random Walks (CTRW) (see \cite{montrollweiss}), to the
Bouchaud Trap Models (BTM) (see
\cite{weakergodicitybreaking,phys,bouchaud-1995,BovierFaggionato05} and
\cite{bcnotes}). It is in fact much wider. This higher level of
generality is needed for the study of random walks on classical random
structures, where the trapping is not introduced {ab initio} as in
the CTRW or the BTM, but is created by the complexity of the underlying
geometry.
We introduce the class of models for general graphs, but
restrict the study in this paper to the case of the line $\Z$. We
obtain a rather complete understanding of the asymptotic behavior of
these trapped walks on $\Z$. We give first a description of all possible
scaling limits, and then proceed to give wide sufficient conditions for
convergence to each of the possible scaling limits. We illustrate this by
two simple examples, one effective and the other geometric, where we
exhibit a rich transition picture between those different asymptotic
regimes and scaling limits.

The behavior of these models in higher dimension or other graphs is open.
It seems clear that, when the underlying graph is transient, the
asymptotic behavior should be much simpler. One might even risk the
conjecture that, when the underlying graph is transient, the Brownian and
the fractional kinetics scaling limits obtained both for the CTRW or the
BTM, should be prevalent in general.

Consider a graph $G=(V,E)$, where $V$ denotes the set of vertices, and $E$
the set of edges. A general ``trapping landscape'' on the graph $G$ will be
given by a collection $\bolds{\pi}=(\pi_x)_{x \in V}$ of probability
measures on $(0,\infty)$. Consider now the continuous-time random process
$X:=(X_t)_{ t\geq0}$ defined on $V$ as follows: $X_t$ stays at a vertex
say $x \in V$, for a random duration sampled from the distribution $\pi_x$
and then moves on to one of the neighbors of $x$, chosen uniformly at
random. If the process $X$ visits $x$ again at a later time, the random
duration of this next visit at $x$ is sampled again and independently,
from the distribution $\pi_x$. We will call the process $X$ the trapped
random walk (TRW) defined by the trapping landscape
$\bolds{\pi}=(\pi_x)_{x \in V}$.

This structure contains the important and very well-studied class of
continuous time random walks (CTRW) as the simple particular case where
the trapping landscape is constant, that is, $\pi_x$ is independent of
$x \in V$. So, in particular the possible scaling limits, on the graph
$\Z^d$, include the Brownian Motion (BM) and the Fractional Kinetics (FK)
models (see \cite{mula}).

We will study in fact a much richer class of models, by considering the
case of random trapping landscapes, that is, the situation where the
landscape $(\pi_x)_{x \in V}$ is given as an i.i.d. sample of a
distribution on the space of probability measures on $(0,\infty)$. The
random collection $(\pi_x)_{x \in V}$ is now a random environment. We
have thus one extra layer of randomness and call the random process $X$
defined as above, for every fixed (or quenched) realization of
environment, a Randomly Trapped Random Walk (RTRW).

This richer class contains the Bouchaud Trap Model. This is the case
where the probability measures $\pi_x$ are chosen as exponential
distributions with mean $\tau(x) $, and the $\tau(x)$'s are chosen as
i.i.d. random variables in $(0,\infty)$. The scaling limits of this model
in dimension 2 and above include the Brownian motion and the fractional
kinetics models (see \cite{BCM06,BC07,Mourrat11,BC08}), and in dimension
one, the Fontes--Isopi--Newman (or FIN) singular diffusion (see
\cite{fin02,JaraLandimTexeira11} and also \cite{Cerny06}).

The new class of RTRWs also contains completely new examples which have
motivated this general study. These examples are of random walks in
random media, where the trapping mechanism is not imposed a priori, but
is a consequence of the geometric characteristics of the medium. For
instance, one of our main motivations is given by the random walk on an
incipient critical Galton--Watson tree (introduced by Kesten in
\cite{kesten}, see also \cite{BarlowKumagai2006}). This incipient
critical tree can be seen as made of a one-dimensional backbone, and of
very long dead-ends (in fact finite critical trees) attached to this
backbone. The trapping landscape is here of geometric origin: the
projection of the random walk along the backbone is trapped by the very
long sojourns in the dead-ends. We are also interested in the similar
problem of the random walk on the invasion percolation cluster on a
regular tree (see \cite{ipc}). These two examples will not be treated
here, but in a forthcoming work.

In this paper, we build the foundation by studying the general question
of understanding the scaling limits of our general class of RTRWs in
dimension one. We call this general class of limit processes Randomly
Trapped Brownian Motions (RTBMs). These processes are all obtained
through random time-changes of Brownian motion. The needed class of time
changes is rich and complex. The class of RTBMs contains naturally the
scaling limits of the examples mentioned above, that is, the Brownian motion,
the FK dynamics, and the FIN-diffusion. But it also contains very
interesting new processes, which we call Spatially Subordinated Brownian
Motions (SSBM). The class of geometric models mentioned above (the random
walk on the incipient critical tree and the invasion percolation
cluster) have scaling limits that belong to these new classes of
models, hence the necessity of the general study done here.

In order to begin the discussion about the asymptotic behavior of the
process $X$ that we have defined above, we remark that its structure is a
priori quite simple. It is given by a random time-change of the standard
discrete-time random walk, say $Y=(Y_n)_{n \geq0}$, on the graph $G$.
Indeed, we first define $S(n)$, the ``clock process,'' that is, the sum
of the
random trapping durations along the first $n$ steps of the random walk
$(Y_n)_{n \geq0} $. More precisely, consider an random array of
independent positive numbers $(s^k_x)_{k \geq1, x \in V}$ where for
every fixed vertex $x \in V$ the numbers $(s^k_x)_{k \geq1}$ are an
i.i.d. sample with common distribution $\pi_x$. Also, define $L(x,n)$
to be
the local time of the random walk $Y$, that is, the number of visits of the
site $x$ before (and including) time $n$.
%
\begin{equation}
\label{e:prertrwstart} L(x,n)= \sum_{k=0}^{n}
\1_{\{Y_k=x\}}.
\end{equation}
The clock process is simply defined as
%
\begin{equation}
S(n)= \sum_{k=0}^{n-1} s^{L(Y_k,k)}_{Y_k}
= \sum_{x\in G} \sum_{k=1}^{L(x,n-1)}
s_x^k.
\end{equation}
Then, clearly the process $X$ is the time change of the simple random
walk $Y$ by this random additive functional, that is,
%
\begin{equation}
\label{e:prertrw} X_t = Y_n \qquad\mbox{if } S(n) \leq t <
S(n+1).
\end{equation}

It is thus perfectly natural, at least when $G=\mathbb Z$, to expect
that the
possible scaling limits will be random time-changes of Brownian motion.
But it might not be obvious that the asymptotic behavior of the
time-change can be as rich as we find it to be. In the case of the FK
processes, it is clear that this time change is a stable subordinator, and
is independent of the underlying Brownian motion. In the case of the FIN
diffusion, this time-change is not independent of the underlying Brownian
motion and is very singular since it retains the randomness of the
spatial information contained in the traps.

In the general situation, the time-change will be even more complex. We
show in our first result (Theorem~\ref{t:scalinglimitclassification})
that the asymptotic behavior is in general a mixture of an FK type
situation, and of the new class of processes, the Spatially Subordinated
Brownian Motions (SSBM). These processes are again defined by a time
change of Brownian motion, where the time change retains some of the
randomness of the spatial information about deep traps, in a much more
intricate fashion than in the FIN case.

In order to illustrate this new class of processes, we also show in this
article that very simple models give rise to them, much simpler indeed
that the two geometric models mentioned above. We start with the simplest
of such models, which we call the model with ``transparent traps'':
Consider the
Bouchaud trap model with the following twist: at site $x \in\Z$ the
process $X$ can, with positive probability, ignore the trap. This model
exhibits different regimes where the scaling limits can be very
different. They include the Brownian motion, the FK dynamics, the FIN
diffusion and in a critical regime a new example of our wide class of
SSBMs. This model is interesting since, although very simple, it
contains this rich array of limiting behaviors and this new transition.
In fact it contains, in a very simple way, the main mechanism: the
possibility to ignore somewhat the deep traps.

As a next step, and building on this intuition, we give finally a
complete study of a simple geometric example, much closer to the cases of
the random walk on the incipient critical tree and invasion percolation
cluster. We study the random walk on comb models. This model is also
rich. If one add a drift toward the teeth of the comb, then various
regimes mentioned above are also present in this model.

\section{Statement of results}
\label{s:results}
In this section, we provide precise statements of our results. We begin by
describing the processes that will later appear as possible scaling
limits of RTRW's on $\Z$. We will define first the Fractional Kinetics
processes, then introduce our new class of spatially subordinated Brownian
motions, and then specialize this definition to introduce the
Fontes--Isopi--Newman (or FIN) diffusion.

%
\begin{definition}[(Fractional kinetics)]
\label{d:fk}
Let $(B_t)_{t\geq0}$ be a standard one-dimensional Brownian motion and
let $(V^\alpha_t)_{t\geq0}$ be an $\alpha$-stable subordinator [for
some $\alpha\in(0,1)$] independent of $B$. Let
$\psi^\alpha_t:=\inf\{s\geq0\dvtx V_s^\alpha>t\}$. The fractional kinetics
process of index $\alpha$, $Z^\alpha$, is defined as
\[
Z^\alpha_t:=B_{\psi^\alpha_t}.
\]
\end{definition}

Next we define \textit{Spatially Subordinated Brownian Motions}
(SSBMs). Let $\mathfrak{F}^{\ast}$ be the set of Laplace exponents of
subordinators (i.e., of nondecreasing L\'evy processes), that is the
set of continuous functions
$f\dvtx \mathbb R_+\to\mathbb R_+$ that can be expressed as
%
\begin{equation}
\label{e:fdPi} f(\lambda)=f_{\mathtt d, \Pi}(\lambda):= \mathtt d\lambda+\int
_{\mathbb{R}_+}\bigl(1-e^{-\lambda t}\bigr)\Pi(dt)
\end{equation}
for a $\mathtt d\ge0$ and a measure $\Pi$ satisfying
$\int_{(0,\infty)}(1\wedge t)\Pi(dt)< \infty$. We endow $\mathfrak
F^\ast$
with topology of pointwise convergence and the corresponding Borel
$\sigma$-algebra.

Let $\mathbb{F}$ be a $\sigma$-finite measure on $\mathfrak{F}^{\ast}$
and let $(x_i,f_i)_{i\in\N}$ be a Poisson point process on
$\mathbb R\times\mathfrak F^\ast$ with intensity $dx \otimes\mathbb F$.
Let $(S_t^i)_{t\ge0}$, $i\in\N$, be a family of processes, such that,
conditioned on a realization of $(x_i,f_i)_{i\in\N}$, $(S^i)_{i\in\N
}$ is
distributed as an independent sequence of subordinators, where the
Laplace exponent of $S^i$ is given by $f_i$. We will assume that the
measure $\mathbb{F}$ satisfies the following assumption:
%
\begin{equation}
\label{e:assumptiononF} \sum_{i:x_i\in[0,1]}S_1^i<
\infty\qquad \mbox{almost surely.}
\end{equation}

Let $B$ be a one-dimensional standard Brownian motion started at the
origin, independent of the $(S^i)_{i\in\N}$, and $\ell(x,t)$ be its local
time. Define
%
\begin{equation}
\phi_t:=\sum_{i\in\N}S^i_{\ell(x_i,t)}
\end{equation}
and $\psi_t:=\inf\{s\geq0\dvtx \phi_s>t\}$.

\begin{definition}[(Spatially subordinated Brownian motion)]
\label{d:GFIN}
The process $B^{\mathbb{F}}$ defined as
\[
B^{\mathbb{F}}_t:=B_{\psi_t}
\]
is called an $\mathbb{F}$-spatially subordinated Brownian motion.
\end{definition}

\begin{remark}
Assumption \eqref{e:assumptiononF} ensures that $\phi_t$ is finite
for all $t\geq0$ and hence the $\mathbb{F}$-SSBM is well defined.
\end{remark}

The FIN diffusion is a particular case of a SSBM. It is in fact a
Markovian SSBM, which has been introduced as the scaling limit of the BTM
on $\Z$ in \cite{fin02}; see also \cite{BC05}. For every $v>0$, consider
the atomic measure $\delta_{f_v}$ concentrated on the linear function
$f_v(\lambda)= v\lambda$. For $\gamma\in(0,1)$, consider the
measure $\mathbb F$ on $\mathfrak F^\ast$ defined by
%
\begin{equation}
\mathbb F^{\gamma} = \int_0^\infty\gamma
v^{-1-\gamma} \delta_{f_v} \,dv.
\end{equation}

\begin{definition}[(Fontes--Isopi--Newman diffusion)]
\label{d:FIN}
For $\gamma\in(0,1)$, the $\mathbb{F}^{\gamma}$-SSBM is the FIN-diffusion
of index $\gamma$ ($\FIN_\gamma$).
\end{definition}

To see that this definition agrees with the usual one, it is sufficient
to observe that the L\'evy process $S_t$ corresponding to the Laplace
exponent $f_v$ satisfies $S_t=t v$, and thus
$\phi_t$ can be written as $\sum_{i} v_i \ell(x_i,t)$ for a Poisson
process $(x_i,v_i)$ on
$\mathbb R\times(0,\infty)$ with intensity $dx\, \gamma v^{-1-\gamma} \,dv$.

Finally, we will define processes which are constructed as mixtures of the
SSBM's and the FK-processes. Let $\mathbb{F}$ be a $\sigma$-finite
measure on $\mathfrak{F}^{\ast}$ satisfying \eqref{e:assumptiononF} and
$(x_i,f_i)_{i\ge0}$, $(S^i)_{i\in\N}$ be as in Definition~\ref{d:GFIN}.
Let $(V^{\gamma})_{t\geq0}$ be an $\gamma$-stable subordinator [for some
$\gamma\in(0,1)$] independent of the processes $(S^i)_{i\in\N}$,
and $B$ be a Brownian motion independent of the
$(S^i)_{i\in\N}$ and $V^\gamma$. Let $\ell(x,t)$ be the local time
of $B$.
Define
%
\begin{equation}
\phi_t:=\sum_{i\in\N}S^i_{\ell(x_i,t)}+V^{\gamma}_t
\end{equation}
and $\psi_t:=\inf\{s\geq0\dvtx \phi_s>t\}$.

\begin{definition}[(FK-SSBM mixture)]
\label{d:FKSSBM}
The process $(B_{\psi_t})_{t\geq0}$ is called an FK-SSBM mixture.
\end{definition}

\begin{remark}
Note that the SSBM and the FK-processes are both particular cases of
FK-SSBM mixtures. The SSBM is obtained by taking $V^\gamma\equiv0 $
(i.e., the ``trivial'' $\gamma$-stable subordinator), and the FK process is
recovered by taking $\mathbb F$ to be a zero measure.
\end{remark}

\begin{remark}
We make here a small digression and describe the results of the
companion paper \cite{BenCab14} about the Random Walk on the Incipent
Infinite Cluster and the Random Walk on the Invasion Percolation Cluster.
As shown by Kesten in \cite{kesten}, the IIC on a regular tree is
composed of a single infinite path, called the \emph{backbone} from
which there emerge finite branches. These random branches are
independent and distributed as critical branching trees.
Let $W^{\IIC}$ be the projection on the backbone of a simple random
walk on the IIC. Since the backbone is one-dimensional, $W^{\IIC}$ can
be seen as a random walk on $\N$ with random jump times, where each
branch of the IIC represents a trap.

In the companion paper \cite{BenCab14}, we study the scaling limit of
$W^{\IIC}$ and show how it can be obtained using the tools developed
here. More precisely, it is easy to see that only the largest traps
will be relevant in the large time behavior of $W^{\IIC}$. On the
other hand, in \cite{AldousCRT3}, Aldous showed that the scaling limit
of critical trees conditioned on being large is the \emph{Continuum
Random tree} (CRT).
Moreover, as shown by Croydon in \cite{rwrt}, the random walk on those
large, critical trees scales to the Brownian motion on the CRT.
The time that $W^{\IIC}$ spends on a large trap is given by the
inverse local time at the root of the Brownian motion on the CRT.
More specifically, we prove that the scaling limit of $W^{\IIC}$ is a
spatially subordinated Brownian motion $B^\mathbb F$, where $\mathbb F$
is related to the law of the Laplace exponent of the inverse local time
at the root of the Brownian motion on the CRT.

The case of the Invasion Percolation Cluster (IPC) is similar. The IPC
also can be seen as a one-dimensional backbone adorned with finite
branches. In this case, the branches are distributed as subcritical
percolation trees, where the percolation parameter converges to the
critical value as we advance along the backbone; see~\cite{ipc}.
Let $W^{\IPC}$ denote the projection on the backbone of a random walk
in the IPC. In \cite{BenCab14}, we study the scaling limit of $W^{\IPC
}$. Moreover, we will see that it is not the same as that of $W^{\IIC
}$, although being very similar. This scaling limit is not exactly an
SSBM but a very slight modification of one.

To get the results described above, we will make use of general
convergence criteria deduced in the present article.
\end{remark}

\subsection{Classification theorem}
The first result we present is a classification theorem which
characterizes the set of limiting processes of RTRWs with an
i.i.d. trapping landscape.

Consider $P\in M_1(M_1((0,\infty)))$ [i.e., $P$ is a probability
measure on the space of probability measures on $(0,\infty)$]. Let
$\bolds\pi $ be the corresponding i.i.d. trapping landscape, that
is an i.i.d. sequence $\bolds\pi= (\pi_z)_{z\in\mathbb Z}$,
$\pi_z\in M_1((0,\infty))$ with marginal $P$ defined on a probability
space $(\Omega, \mathcal F, \mathbb P)$. Given a realization of
$\bolds\pi$, let $(s_x^i)_{x\in\mathbb Z, i\ge1}$ be an
independent collection of random variables such that $s_x^i$ has
distribution $\pi_x$, and let $X$ be the RTRW whose random trapping
landscape is $\bolds\pi$, defined as in
\eqref{e:prertrwstart}--\eqref{e:prertrw}. We write $P^{\bolds\pi}$
for the law of $X$ given $\bolds\pi$. The distribution of $X$ is
then the
semidirect product $\mathbb P\times P^{\bolds\pi}$.

\begin{theorem}
\label{t:scalinglimitclassification}
Assume that there is a nondecreasing function $\rho$ such that the
processes
%
\begin{equation}
X^\varepsilon_t = \varepsilon X_{\rho(\varepsilon)^{-1}t},\qquad t\ge0,
\end{equation}
converge as $\varepsilon\to0$ in
$(\mathbb P\times P^{\bolds\pi})$-distribution on the space
$D(\mathbb R_+)$ of
cadlag functions endowed with Skorokhod topology to a process $U$
satisfying the nontriviality assumption
%
\begin{equation}
\limsup_{t\to\infty}|U_t|=\infty\qquad \mbox{almost
surely.}
\end{equation}
Then one
of the two following possibilities occurs:

\begin{longlist}[(ii)]
\item[(i)] $\rho(\varepsilon)= \varepsilon^2 L(\varepsilon)$ for a
function $L$ slowly varying at $0$. Then there exists $c>0$ such that
$U_t=(B_{c^{-1}t})_{t\geq0}$ where $B$ is a standard Brownian motion.

\item[(ii)] $\rho(\varepsilon) = \varepsilon^\alpha L(\varepsilon
)$ for
$\alpha>2$ and a function $L$ slowly varying at $0$. Then
$U$ is a
FK-SSBM mixture $(B_{\psi_t})_{t\geq0}$. Moreover, index $\gamma$ of
the $\gamma$-stable subordinator associated to $B_{\psi_t}$ equals
$2/\alpha$ and the intensity measure $\mathbb F$ satisfies the
scaling relation
%
\begin{equation}
\label{e:Fscaling} a \mathbb F(A)= \mathbb F\bigl(\sigma_a^\alpha
A\bigr)\qquad \mbox{for every }A\in\mathcal B\bigl(\mathfrak F^\ast\bigr),
a>0,
\end{equation}
where
$\sigma_a^\alpha\dvtx  \mathfrak F^\ast\to\mathfrak F^\ast$ is
defined by
%
\begin{equation}
\label{e:sigma} \sigma_a^\alpha(f) (\lambda) = a f
\bigl(a^{-\alpha} \lambda\bigr).
\end{equation}
\end{longlist}
\end{theorem}

\begin{remark}
\label{r:scalingremark}
The map $\sigma_a^\alpha$ maps the Laplace exponent of a L\'evy process
$V$ to the Laplace exponent of the L\'evy process
$a^{-\alpha} V(a \cdot)$.
\end{remark}


\subsection{Convergence theorems}
We now present sufficient conditions for the convergence to the processes
described above. Let $X$ be, as above, a RTRW with i.i.d. trapping
landscape whose marginal is $P\in M_1(M_1((0,\infty)))$.

\subsubsection{Convergence to Brownian motion}
We start by presenting general criteria for the convergence to the
Brownian motion. For any probability measure $\nu\in M_1((0,\infty
))$, we
define $m(\nu)$ to be its mean,
%
\begin{equation}
\label{e:defmean} m(\nu)= \int_{\mathbb R_+} x \nu(dx).
\end{equation}

\begin{theorem}
\label{t:BMconvergence}
Assume that
%
\begin{equation}
M:=\int m(\pi) P(d\pi)\in(0,\infty).
\end{equation}
Then $\mathbb P$-a.s., as $\varepsilon\to0$, the
rescaled RTRW
$(\varepsilon X_{M^{-1}\varepsilon^{-2}t})_{t\ge0}$ converges
to a standard Brownian motion, in $P^{\bolds\pi}$-distribution
on the space
$D(\R_+)$.
\end{theorem}

\begin{remark}
Observe that Theorem~\ref{t:BMconvergence} is a quenched result: the
convergence holds for $\mathbb P$-a.e. realization of the trapping
landscape $\bolds\pi$.
\end{remark}

\subsubsection{Convergence to the Fractional Kinetics process}

We now deal with the convergence to the FK process. Let, as usual, $X$
be a RTRW with i.i.d. trapping landscape $\bolds\pi$ whose
marginal is $P$. We write
%
\begin{equation}
\hat\pi(\lambda):= \int_0^\infty e^{-\lambda t}
\pi(dt)
\end{equation}
for the Laplace transform of a probability measure over
$(0,\infty)$, and set
%
\begin{equation}
\label{e:FKGamma} \Gamma(\varepsilon):=\mathbb E\bigl[1-\hat\pi_0 (
\varepsilon)\bigr].
\end{equation}
It is easy to see that $\Gamma $ is strictly increasing on $\mathbb R_+$,
taking values in $[0,\Gamma_\Max)$ for some $0< \Gamma_\Max\leq1$.
Therefore, the inverse $\Gamma^{-1}$ is well defined on this interval.
For $\varepsilon$ small enough, we can thus introduce the inverse time
scale $q_{\FK}$ by
%
\begin{equation}
\label{e:qFK} q_{\FK}(\varepsilon)=\Gamma^{-1}\bigl(
\varepsilon^{2}\bigr).
\end{equation}

\begin{theorem}
\label{t:FKconv}
Assume that
%
\begin{equation}
\label{e:qFKcond} q_\FK(\varepsilon) = \varepsilon^{\alpha} L(
\varepsilon)
\end{equation}
for some $\alpha>2$ and a slowly varying function $L$. In addition
assume that
%
\begin{equation}
\label{e:FKsecondmoment} \lim_{\varepsilon\to0} \varepsilon^{-3} \mathbb
E \bigl[ \bigl(1-\hat{\pi}_0\bigl(q_{\FK}(\varepsilon)
\bigr) \bigr)^{2} \bigr] =0.
\end{equation}
Then, as $\varepsilon\to0$, the rescaled RTRW
$(\varepsilon X_{q_{\FK}(\varepsilon)^{-1}t})_{t\ge0}$ converges in
$P^{\bolds\pi}$-distri\-bution on $D(\mathbb R_+)$ to the FK
process with parameter $\gamma=2/\alpha$, in $\Pb$-probability.

In addition, if $\varepsilon^{-3}$ in \eqref{e:FKsecondmoment} is
replaced by $\varepsilon^{-4-\delta}$, $\delta>0$, then the
convergence in $P^{\bolds\pi}$-distribution holds $\Pb$-a.s.
\end{theorem}

\begin{remark}
Due to \eqref{e:qFK}, \eqref{e:qFKcond} is equivalent to
%
\begin{equation}
\label{e:qFKcondeq} \Gamma(\varepsilon)=\varepsilon^{2/\alpha }\tilde L(
\varepsilon) =\varepsilon^\gamma\tilde L(\varepsilon),
\end{equation}
for some slowly varying function $\tilde L$.
\end{remark}

\subsubsection{Convergence to spatially subordinated Brownian motions}
Here, we present sufficient conditions for the convergence to the SSBM
processes introduced in Definition~\ref{d:GFIN}. We assume that $X$ is
a RTRW with an i.i.d. random trapping landscape
$\bolds\pi=(\pi_z)_{z\in\Z}$ with marginal
$P\in M_1(M_1((0,\infty)))$.

We recall that $m(\nu)$ denotes the mean of the probability
distribution $\nu$; see~\eqref{e:defmean}.
Our first assumption is that the distribution of $m(\pi_0)$ has heavy tails.

\begin{assumption*}[(HT)]
\label{a:PP}
There exists $\gamma\in(0,1)$ and a nonvanishing slowly varying
function at infinity $L\dvtx \R_+\to\R_+$ such that
%
\begin{equation}
P\bigl[\pi\in M_1\bigl((0,\infty)\bigr)\dvtx m(\pi)>u
\bigr]=u^{-\gamma}L(u).
\end{equation}
\end{assumption*}

\begin{remark}
\label{rk:assumptionpp}
We define
$V \in D(\mathbb R)$ by
%
\begin{equation}
\label{e:VV} V_x= \cases{ \displaystyle\sum_{i=1}^{\lfloor x \rfloor}
m(\pi_i), &\quad $x\ge1$,\vspace *{2pt}
\cr
0,&\quad $x\in[0,1)$,\vspace*{2pt}
\cr
\displaystyle\sum_{i=\lfloor x \rfloor+1}^0 m(\pi_i),
&\quad $x<0$.}
\end{equation}
Under Assumption (HT), there exists a function $d(\varepsilon)$
satisfying\break
$d(\varepsilon)=\varepsilon^{-1/\gamma}\tilde L(\varepsilon)$ for a
function $\tilde L$ slowly varying at $0$, such that\break
$(d(\varepsilon)^{-1}V_{\varepsilon^{-1}x})_{x\in\R}$ converges in
distribution on $D(\R)$ to a (two-sided) $\gamma$-stable
subordinator with L\'evy measure $\gamma v^{-1-\gamma} \,dv$. In addition,
we may assume that $d$ is strictly decreasing and continuous.
\end{remark}

Next, we prepare the statement of the second assumption. For each
$a\in\R_+$, let $\pi^a$ be a random measure having the law of $\pi_0$
conditioned on $m(\pi_0)=a$. Let
%
\begin{equation}
\label{e:defq} q(\varepsilon):=\varepsilon d(\varepsilon)^{-1},
\end{equation}
where $d$ is as in Remark~\ref{rk:assumptionpp}.

For $\varepsilon>0$, define
$\Psi_{\varepsilon}\dvtx M_1((0,\infty))\to C(\mathbb{R}_+)$ by
%
\begin{equation}
\label{e:defPsi} \Psi_{\varepsilon}(\nu) (\lambda):=\varepsilon^{-1}
\bigl(1-\hat{\nu}\bigl(q(\varepsilon)\lambda\bigr)\bigr),\qquad \nu \in M_1
\bigl((0,\infty)\bigr), \lambda\ge0.
\end{equation}
Observe that $\Psi_{\varepsilon}(\nu)$ is the Laplace exponent of a pure
jump L\'evy process whose jumps have intensity $\varepsilon^{-1}$ and the
size of jumps divided by $q(\varepsilon)$ has distribution $\nu$. In
particular, $\Psi_\varepsilon(\nu)\in\mathfrak F^\ast$ for every
$\nu\in M_1((0,\infty))$. Our second assumption is:

\begin{assumption*}[(L)]
There exists $\mathbb{F}_1\in M_1(\mathfrak F^\ast)$ such that
%
\begin{equation}
\label{e:assumptionD0} \mbox{law of }\Psi_{\varepsilon}\bigl(\pi^{d(\varepsilon)}\bigr)
\xrightarrow{\varepsilon\to0} \mathbb{F}_1.
\end{equation}
In addition, $\mathbb F_1$ is nontrivial, that is,
%
\begin{equation}
\label{e:Fnontrivial} \mathbb F_1\neq\delta_{\mathbf{0}},
\end{equation}
where $\mathbf{0}$ is the identically zero function.
\end{assumption*}

\begin{remark}
\label{r:fboudedness}
Observe that $\Psi_\varepsilon(\pi^{d(\varepsilon)})$ is a Laplace
exponent of a subordinator $S$ such that
$\mathbb E[S_1]=\varepsilon^{-1}d(\varepsilon)q(\varepsilon)=1$. The
measure $\mathbb F_1$ thus gives the full mass to the set
$\mathfrak F\subset\mathfrak F^\ast$ of functions
$f\dvtx \mathbb{R}_+\rightarrow\mathbb{R}$ that can be written as
$f(\lambda)= \mathtt d \lambda + c \int(1-e^{-\lambda t})\Pi(dt)$
for $\mathtt d + c \le1$ and $\Pi$ satisfying
$\int_{\mathbb{R}_+}t\Pi(dt)=1$. In particular, $f(\lambda) \le
\lambda$.
\end{remark}

\begin{theorem}
\label{t:RSBMconv}
Assume that \textup{(HT)} and \textup{(L)} hold. Then, as $\varepsilon\to0$,
$(\varepsilon X_{q(\varepsilon)^{-1}t})_{t\ge0}$ converges on
$D(\mathbb R_+)$ in $\mathbb P\times P^{\bolds\pi}$-distribution
to a SSBM process $(B^{\mathbb F}_t)_{t\ge0}$ introduced in Definition~\ref{d:GFIN}. The intensity measure $\mathbb F$ which determines the
law of the limiting process is given by
%
\begin{equation}
\label{e:GFINF} \mathbb F(d f):= \int_{0}^\infty
\gamma v^{-\gamma-1} \mathbb F_v(df) \,dv,
\end{equation}
where [recall \eqref{e:sigma} for the notation]
%
\begin{equation}
\label{e:Fv} \mathbb F_v:= \mathbb F_1 \circ
\sigma_{v^{\gamma}}^{1+1/\gamma}.
\end{equation}
\end{theorem}

\begin{remark}
Observe that the scaling relation \eqref{e:Fscaling} is satisfied for
$\mathbb F$ in~\eqref{e:GFINF} and $\alpha= 1 +\frac{1} \gamma$.
Indeed, since
$\sigma^{1+1/\gamma}_a \sigma^{1+1/\gamma}_b = \sigma^{1+1/\gamma}_{ab}$,
for any $A\in\mathcal B(\mathfrak F^\ast)$,
%
\begin{eqnarray}
\mathbb F\bigl(\sigma^{1+1/\gamma}_a A\bigr) &=& \int\gamma
v^{-1-\gamma} \mathbb F_v\bigl(\sigma^{1+1/\gamma}_a
A\bigr) \,dv
\nonumber
\\
&=& \int\gamma v^{-1-\gamma} \mathbb F_1\bigl(\sigma
^{1+1/\gamma}_{av^\gamma} A\bigr) \,dv
\\
&=& a \int\gamma u^{-1-\gamma} \mathbb F_1\bigl(\sigma
^{1+1/\gamma}_u A\bigr) \,du = a \mathbb F(A).\nonumber
\end{eqnarray}
\end{remark}

\subsubsection{Convergence to the FIN diffusion}
Next, we present a theorem which gives sufficient conditions for
convergence to the FIN diffusion. Recall that $m(\nu)$ denotes the
expectation $\nu\in M_1((0,\infty))$, and define
$m_2(\nu)=\int_{\R_+} t^2   \nu(dt)$ to be its second moment. As before,
we let $\pi^a$ stand for a random measure having the distribution of
$\pi_0$ given $m(\pi_0)=a$. Define random variable $m_2(a):=m_2(\pi^a)$.

\begin{theorem}
\label{t:FINconv}
Assume that \textup{(HT)} holds and let $d(\varepsilon)$ be as in
Remark~\ref{rk:assumptionpp}. In addition, let
$\varepsilon d(\varepsilon)^{-2}m_2(d(\varepsilon))\xrightarrow
{\varepsilon\to0} 0$
in distribution. Then\break $(\varepsilon^{-1}X_{q(\varepsilon)t})_{t\geq0}$
converges to the $\FIN_{\gamma}$ diffusion in the sense of Theorem~\ref{t:RSBMconv}.
\end{theorem}

We conclude the \hyperref[sec1]{Introduction} with a description of the organization of
the paper. In Section~\ref{s:examples}, we will define two examples of
RTRWs for which we will prove convergence results. First, we will define
the transparent traps model and we will state the theorem which describes
its phase diagram (see Theorem~\ref{t:phasediagramtransparenttraps}).
Then we will define the comb model and we will present
Theorem~\ref{t:phasediagramcombmodel} which deals with its possible
scaling limits.

Sections~\ref{s:delayed} and \ref{s:rtrwrtbm} contain the main
definitions which will be used through the paper. In
Section~\ref{s:delayed}, we give the precise definitions of trapped random
walks and trapped Brownian motions. In Section~\ref{s:rtrwrtbm}, we give
the definitions and examples of randomly trapped random walks and
randomly trapped Brownian motions. In Section~\ref{s:convergence}, we
prove a general result from which one can deduce convergence of trapped
processes from the convergence of their respective \textit{trap measures}.

The bulk of the paper is Section~\ref{s:limits} where we deal with limits
of RTRWs. In Section~\ref{ss:classification}, we prove the
classification of the all possible limits of RTRWs with i.i.d. trapping
landscape stated in Theorem~\ref{t:scalinglimitclassification}. In
Section~\ref{ss:BMconvergence}, we prove Theorem~\ref{t:BMconvergence}
which deals with the convergence to the Brownian motion. The convergence
to the FK process stated in Theorem~\ref{t:FKconv} will be proved in
Section~\ref{ss:FKconvergence}. In Section~\ref{ss:GFINconvergence},
we will prove the convergence to the SSBM stated in
Theorem~\ref{t:RSBMconv}. In Section~\ref{ss:FINconvergence}, we prove
Theorem~\ref{t:FINconv} which states the convergence to the FIN diffusion.

Finally, Section~\ref{s:applications} deals with the proof of the
theorems for the transparent traps model and the comb model. In
Section~\ref{ss:faketrap}, we will prove
Theorem~\ref{t:phasediagramtransparenttraps} and in
Section~\ref{ss:combmodel} we will prove Theorem~\ref{t:phasediagramcombmodel}.

The \hyperref[app]{Appendix} collects, for the reader's convenience, several known results from
the random measure theory that are used through the paper.

\section{Examples}
\label{s:examples}

In this section, we define two examples of RTRWs. We also present the
theorems which describe their phase-diagrams.

\subsection{Transparent traps model}
\label{ss:transparenttraps}
The simplest model which we will treat is the \textit{trap model with
transparent traps}. Let $\alpha,\beta>0$, and let
$(\tau_x)_{x\in\mathbb{Z}}$ be a i.i.d. sequence of positive random
variables which satisfy
%
\begin{equation}
\label{tail} \lim_{u\to\infty}u^{\alpha}\mathbb{P}(
\tau_0>u)=c\in (0,\infty),
\end{equation}
and $\Pb(\tau_x>1)=1$. For each $x\in\mathbb{Z}$, consider the random
probability distribution
$\pi_x:=(1-\tau_x^{-\beta})\delta_1 +\tau_x^{-\beta}\delta_{\tau_x}$.

\begin{definition}[(Trap model with transparent traps)]\label{d:faketrap}
Let $X$ be the RTRW with random trapping landscape $(\pi_x)_{x\in\Z
}$. Then $X$ is the called the trap model with transparent traps.
\end{definition}

\begin{figure}[b]

\includegraphics{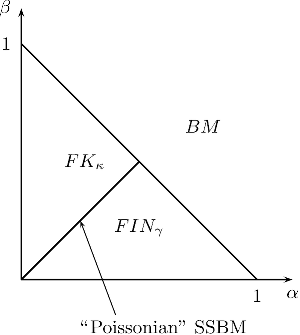}

\caption{Phase diagram for the transparent traps model.}
\label{figure 1}
\end{figure}

The reason for this name is the following. When $X$ reaches
$x\in\mathbb Z$, it is trapped there for time $\tau_x$ with probability
$\tau_x^{-\beta}$, otherwise it does not ``see'' the trap and just stays
at $x$ for a unit of time. The phase-diagram of the transparent traps
model (see Figure \ref{figure 1}) is given by the following theorem.

\begin{theorem}\label{t:phasediagramtransparenttraps}
The trap model with transparent traps has the following scaling
behavior:
\begin{longlist}[(iii)]
\item[(i)]
If $\alpha+\beta>1$, then for
$m:=\mathbb E (m(\pi_0))<\infty$, the process
$\varepsilon X_{m\varepsilon^{-2} t}$ converges
to a standard
Brownian motion in the sense of Theorem~\ref{t:BMconvergence}.

\item[(ii)]
If $\alpha+\beta<1$ and $\alpha>\beta$, then for
$\gamma= \alpha/(1-\beta)$ and
$q(\varepsilon)=\varepsilon^{1+1/\gamma}$, the process
$\varepsilon X_{q(\varepsilon)^{-1}t}$ converges
to $\FIN_{\gamma}$ in the sense of Theorem~\ref{t:RSBMconv}.

\item[(iii)]
If $\alpha+ \beta<1$ and $\alpha<\beta$, then for
$\kappa= \alpha+ \beta$ and
$q(\varepsilon)= \varepsilon^{2/\kappa}$, the process
$\varepsilon X_{q(\varepsilon)^{-1}t}$ converges to a
fractional kinetics process with parameter $\kappa$
in the sense of Theorem~\ref{t:FKconv}.

\item[(iv)]
If $\alpha+ \beta<1$ and $\alpha=\beta$, then for
$q(\varepsilon)=\varepsilon^{1/\alpha}$ the process
$\varepsilon X_{q(\varepsilon)^{-1}t}$ converges,
in the sense of Theorem~\ref{t:RSBMconv}, to a SSBM process, which will
be referred as a ``Poissonian'' SSBM.
\end{longlist}
\end{theorem}

\begin{remark}
In the case $\alpha+ \beta=1$ which is not covered by the theorem, the
scaling limit is a Brownian motion, but a logarithmic correction should
be added to the scaling. We do not consider this case here for the
sake of
brevity.
\end{remark}

\subsection{Comb model}
\label{ss:combpre}
The comb model is a ``geometric''
RTRW on a graph that looks like a comb with
randomly long teeth. More precisely, consider an i.i.d. family $N_z$,
$z\in\mathbb Z$, satisfying
%
\begin{equation}
\label{tailcombb} \Pb(N_0=n)=\mathcal Z^{-1}
n^{-1-\alpha},\qquad n\geq1
\end{equation}
for some $\alpha>0$ and a normalizing constant
$\mathcal Z = \mathcal Z(\alpha)$. Let $G_z$ be the graph with vertices
$\{(z,0),(z,1),\ldots,(z,N_z)\}$ and with nearest-neighbor edges, and let
$G_{\comb}$ be the tree-like graph composed by a backbone $\Z$ with
leaves $(G_z)_{z\in\Z}$; $(z,0)\in G_z$ is identified with
$z\in\mathbb Z$ on the backbone. By projecting the simple random walk on
$G_{\comb}$ to the backbone we obtain a RTRW denoted $X^{\comb}$.

\begin{figure}[b]

\includegraphics{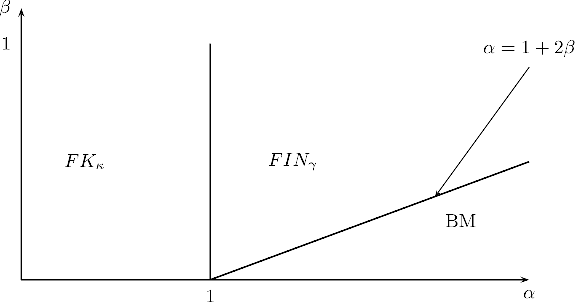}

\caption{Phase diagram for the comb model.}
\label{figure 2}
\end{figure}

We will see later that the behavior of $X^\comb$ is not very rich. When
$\alpha>1$, the teeth are ``short'' and the mean time spent on them has a
finite expectation, thus $X^\comb$ is diffusive and Brownian motion is
its scaling limit. On the other hand, when $\alpha<1$, then the teeth
may be ``long,'' and the expectation of the mean trapping time is infinite.
However, as it is rather unlikely for the random walk on $G_\comb$ to
reach the tip of long teeth, it takes many visits to a tooth to discover
that it is long. This indicates that in this case the FK process is the
limit.

To obtain a richer behavior, we need to increase the chance that the
random walk on $G_\comb$ hits the tips of the teeth. Therefore, we add a
small drift pointing to the tips, as follows. Let $Y^\comb$
be a random walk on $G_\comb$ which on the backbone behaves like the
simple random walk on $G_\comb$,
%
\begin{equation}
\Pb\bigl[Y^\comb_{k+1}= z\pm1| Y^\comb_k
= z\bigr]=\Pb\bigl[Y^\comb_{k+1}=(z,1)|Y^\comb_k=z
\bigr]=\tfrac{1}3,
\end{equation}
and, when on the tooth $G_z$, it performs a
random walk with a drift $g(N_z)\geq0$ pointing away
from the backbone, reflecting at the tip:
for any $z$ and $0<n<N_z$,
%
\begin{eqnarray}
\Pb\bigl[Y^{\comb}(k+1)=(z,n+1)|Y^{\comb}(k)=(z,n)\bigr]&=&
\bigl(1+g(N_z)\bigr)/{2},
\\
\Pb\bigl[Y^{\comb}(k+1)=(z,n-1)|Y^{\comb}(k)=(z,n)\bigr]&=&
\bigl(1-g(N_z)\bigr)/{2},
\\
\qquad
\Pb\bigl[Y^{\comb}(k+1)=(z,N_z-1)|Y^{\comb}(k)=(z,N_z)
\bigr]&=&1.
\end{eqnarray}
We will choose $g$ as
%
\begin{equation}
g(N)=\min\bigl(\beta N^{-1}\log(N),1\bigr)
\end{equation}
for some $\beta\geq0$. The case $\beta=0$ corresponds to the comb model
without drift.

\begin{definition}[(Comb model)]\label{d:combmodel}
We define $X^\comb$ as then the projection of $Y^\comb$ to the
backbone. More precisely, $X^{\comb}_t=z$ iff
$Y^\comb_{\lfloor t \rfloor}\in G_z$.
\end{definition}

The next theorem describes the phase-diagram of $X^\comb$ (see Figure \ref{figure 2}).

\begin{theorem}
The comb model has the following scaling behavior:
\label{t:phasediagramcombmodel}
\begin{longlist}[(iii)]
\item[(i)]
If $\alpha>1 $ and $1+2\beta<\alpha$, then for some
$m\in(0,\infty)$, the process
$\varepsilon X^\comb_{m\varepsilon^{-2} t}$ converges
to a standard
Brownian motion in the sense of Theorem~\ref{t:BMconvergence}.

\item[(ii)]
If $\alpha>1$ and $1+2\beta>\alpha$, then for
$\gamma= \alpha/(1+2\beta)$ there exists a regularly varying
function $q(\varepsilon)$ of index $1+1/\gamma$, such that the process
$\varepsilon X^\comb_{q(\varepsilon)^{-1}t}$ converges
to $\FIN_{\gamma}$ in the sense of Theorem~\ref{t:RSBMconv}.
The same holds true for the line $\alpha=1$, \mbox{$\beta>0$}.

\item[(iii)]
If $\alpha<1$, then for $\kappa=\frac{1+\alpha}{2(1+\beta)}$, there
exists a regularly varying function $q(\varepsilon)$ of index
$2/\kappa$ such that the process
$\varepsilon X^{\comb}_{q(\varepsilon)^{-1}t}$\vspace*{1pt} converges to a
fractional kinetics process with parameter $\kappa$ in the sense of
Theorem~\ref{t:FKconv}.
\end{longlist}
\end{theorem}

\begin{remark}
We expect that on the line $\alpha=1+2\beta$ the scaling limit is
Brownian motion.
\end{remark}

\section{Trapped random walks and trapped Brownian motions}
\label{s:delayed}

\subsection{Trapped random walk} 

In this section, we give the definitions of several classes of processes
which we will use through the paper.

\subsubsection{Time changed random walk}

We first consider ``deterministic'' time change. Let $Z = (Z_k)_{k\ge0}$
be a simple symmetric discrete-time random walk on $\mathbb Z$, $Z_0=0$,
and let $(s_x^i)_{x\in\Z,i\in\N}$ (with $\mathbb N=\{1,2,\ldots\}$)
be a
family of positive numbers. We define \textit{time changed random walk}
as the continuous-time $\mathbb Z$-valued process following the same
trajectory as $Z$, characterized by stating that the duration of the $i$th
visit of $X$ to $x\in\Z$ is $s_x^i$.

Alternatively, the time changed random walk can be defined using the
following procedure, which will be more suitable for generalization into
a continuous setting. Consider an atomic measure on
$\mathbb H:= \R\times\R_+ = \mathbb R \times[0,\infty)$ given by
%
\begin{equation}
\label{e:mudef} \mu:=\sum_{x\in\Z,i\in\N} s_x^i
\delta_{(x,i)}.
\end{equation}
Let
%
\begin{equation}
L(x,t):=\sum_{i=1}^{\lfloor t\rfloor}
\1_{\{Z_i=\lfloor x \rfloor\}},\qquad t\ge0,x\in\mathbb R
\end{equation}
be the local time of $Z$. For a Borel-measurable function
$f\dvtx \mathbb R\to\mathbb R_+$, define the set $U_f\subset\mathbb H$ of
points under the graph of $f$ by
%
\begin{equation}
\label{e:Uf} U_f:=\bigl\{(x,y)\in\mathbb H\dvtx y\leq f(x)\bigr
\}.
\end{equation}
Let $\phi[\mu,Z]\dvtx \mathbb R_+\to\mathbb R_+$ be the function
%
\begin{equation}
\phi[\mu,Z]_t:=\mu(U_{L(\cdot, t)}), \qquad t \ge0,
\end{equation}
and let $\psi[\mu,Z]$ be its right-continuous generalized inverse
%
\begin{equation}
\psi[\mu,Z]_t:=\inf\bigl\{s>0\dvtx \phi[\mu,Z]_s>t
\bigr\},\qquad  t\ge0.
\end{equation}

\begin{definition}
\label{d:tcrw}
The \emph{$\mu$-time changed
random walk} $(Z[\mu]_t)_{t\ge0}$ is the process given by
%
\begin{equation}
Z[\mu]_t:=Z_{\psi[\mu,Z]_t}, \qquad t\ge0.
\end{equation}
\end{definition}

\begin{remark}
(a) If $\mu(\mathbb H)<\infty$, $Z[\mu]$ is not defined for times
$t> \mu(\mathbb H)$ and might not be defined for $t=\mu(\mathbb H)$.

(b) It is easy to see that the functions $\phi[\mu,Z]$ and
$\psi[\mu,Z]$ are nondecreasing and right-continuous. Hence, $Z[\mu]$
has right-continuous trajectories.
\end{remark}

\subsubsection{Trapped random walk}

We want to, of course, to consider random time changes. One natural way how to
introduce randomness is to require that the duration of every visit to
$x\in\Z$ is distributed according to some probability distribution~$\pi_x$,
which may depend on $x$, assuming also that the durations of the visits
are independent, and independent of the direction of the jumps of the
random walk $Z$. We will call such random time change \textit{trapped
random walk} with (deterministic) \textit{trapping landscape}
$\pi= (\pi_x)_{x\in\mathbb Z}$.

More precisely, extending Definition~\ref{d:tcrw}, we may define the
trapped random walk as follows.

\begin{definition}[(Trapped random walk)]
\label{d:Trapped random walk}
Let $\bolds\pi=(\pi_x)_{x\in\mathbb Z}$ be a sequence of
probability measures on $(0,\infty)$, $(s_x^i)_{i\in\N,x\in\Z}$ an
independent family of random variables such for every $x\in\Z$,
$(s_x^i)_{i\in\N}$ is an i.i.d. sequence distributed according to
$\pi_x$. Let $\mu$ be a random measure on $\mathbb H$ defined as in
\eqref{e:mudef}, and let $Z$ be a simple symmetric random walk
independent of $(s_x^i)_{x\in\mathbb Z,i\in\mathbb N}$. The $\mu$-time
changed random walk $Z[\mu]$ is then called \emph{trapped random
walk} (TRW) with \emph{trap measure $\mu$} and \emph{trapping
landscape $\bolds\pi$}.
\end{definition}

We present three examples of TRWs.

\begin{example}[(Montrol--Weiss continuous-time random walk)]
\label{ex:MWCTRW}
Let $\pi_x=\pi_0$ for all $x\in\mathbb{Z}$, and assume that $\pi_0$
satisfies the tail condition
%
\begin{equation}
\lim_{u\to\infty}u^\gamma\pi_0\bigl([u,\infty
)\bigr)=c
\end{equation}
for some
$\gamma\in(0,1)$ and $c\in(0,\infty)$. In this case, the durations of
visits $(s_x^i)_{i\in\N,x\in\Z}$ form an i.i.d. family with marginal
$\pi_0$, and the trapped random walk $Z[\mu]$ is a one-dimensional
continuous-time random walk \`a la Montroll--Weiss (see
\cite{montrollweiss}).
\end{example}

\begin{example}[(Geometric TRW)]
\label{ex:GTRW}
Let $(G_x)_{x\in\Z}$ be a family of rooted finite graphs, and let $G$
be the graph obtained by attaching the graphs $G_x$ to vertices
of~$\mathbb Z$. More precisely, denote by $V(G_x)$ the set of vertices
of~$G_x$, and assume that $(V(G_x))_{x\in\Z}$ are pairwise disjoint. Then
$G$ is the graph whose set of vertices is $V(G):=\bigcup_{x\in\Z}V(G_z)$,
and its set of edges $E(G)$ is determined by: $(y,z)\in E(G)$ iff one
of the following conditions hold:
\begin{itemize}
\item There exists $x\in\Z$ such that $y,z\in V(G_x)$ and $y$
and $z$ are neighbors in $G_x$.
\item There exists $x\in\Z$ such that $y$ is the root of $G_x$
and $z$ is the root of $G_{x+1}$.
\item There exists $x\in\Z$ such that $y$ is the root of $G_x$
and $z$ is the root of $G_{x-1}$.
\end{itemize}
Hence, $G$ is a graph consisting of a copy of $\Z$ (called the
\textit{backbone}) from which emerge \textit{branches} $G_x, x\in\Z$.
We will naturally identify the backbone with $\Z$.

Let $Y:=(Y_k)_{k\geq0}$ be a discrete time, symmetric random walk on $G$
with $Y_0=0$. We can project $Y$ to the backbone to obtain a continuous
time $\mathbb Z$-valued process $W:=(W_t)_{t\geq0}$ given by
$W_t= x\in\Z$ iff $Y_{\lfloor t \rfloor}\in G_x$. We call $W$
\emph{Geometric trapped random walk}. Its waiting times are of course
related to the distribution of the return time to the root for the
simple random walks on the finite graphs $G_x$.
\end{example}

\begin{example}[(Markovian random walk on $\mathbb Z$)]
The trapped random walk is in general not Markovian. However, when for
a family of positive numbers $(m_x)_{x\in\Z}$, $\pi_x$ is the
exponential distribution with mean $m_x$, then the trapped random walk
$Z[\mu]$ with trapping landscape $(\pi_z)_{z\in\Z}$ is Markovian. The
total jump rate at $x$ is $m_x^{-1}$.
\end{example}

\subsection{Trapped Brownian motion} 
We now define continuous counterparts of the previously defined processes.

\subsubsection{Time changed Brownian motion}

\begin{definition}[($\mu$-time changed Brownian motion)]
\label{d:delayedBM}
Let $\mu$ be a deterministic measure on $\mathbb H$, and $B$ be a standard
one-dimensional Brownian motion. Denote by $\ell(x,t)$ a bi-continuous
version of the local time of $B$, and define
%
\begin{eqnarray}
\phi[\mu,B]_t&:=&\mu(U_{\ell(\cdot, t)}),
\nonumber
\\[-8pt]
\\[-8pt]
\nonumber
\psi[\mu,B]_t&:=&\inf\bigl\{s>0\dvtx \phi[\mu,B]_s>t
\bigr\}.
\end{eqnarray}
The \emph{$\mu$-time changed Brownian motion} $(B[\mu]_t)_{t\ge
0}$ is the
process given by
%
\begin{equation}
B[\mu]_t:=B_{\psi[\mu,B]_t},\qquad  t\ge0.
\end{equation}
\end{definition}

\begin{remark}
It is easy to see that the functions $\phi[\mu,B]$, and $\psi[\mu,B]$
are nondecreasing and right-continuous. Hence, $B[\mu]$ has
right-continuous trajectories.
\end{remark}

\subsubsection{Trapped Brownian motion}

Before defining the class of trapped Brownian motions, we recall the
definition of random measure with independent increments
(see Section~10 of \cite{Kal02}).

\begin{definition}
\label{trapmeasure}
A random measure $\mu$ on $\mathbb H$ is called a \emph{measure with
independent increments} iff for every two disjoint sets
$A,B\in{\mathcal B}(\mathbb H)$, the random variables $\mu(A)$ and
$\mu(B)$ are independent.
\end{definition}

For any random measure $\mu$ and $A\in\mathcal B(\mathbb R)$ we define
the $\mu$-trapping process $(\mu\langle A \rangle_t)_{t\geq0}$ by
%
\begin{equation}
\mu\langle A\rangle_t:= \mu\bigl(A\times[0,t]\bigr).
\end{equation}
Note that, if $\mu$ is a measure with independent increments and $A,B$
are disjoint Borel subsets of $\mathbb{R}$, then $\mu\langle A
\rangle$
and $\mu\langle B \rangle$ are independent processes.

\begin{definition}[(L\'evy trap measure)]
A random measure $\mu$ on $\mathbb H$ is called \emph{L\'evy trap
measure} when $\mu\langle A \rangle$ is a L\'evy process for every
bounded $A\in\mathcal{B}(\mathbb{R})$.
\end{definition}

Observe that a L\'evy trap measure does not need to have independent
increments. Its increments are independent in the time direction, but
not necessarily in the space direction.
L\'evy trap measures which, in addition, have independent increments
will be used to define the trapped Brownian motions.

\begin{definition}[(Trapped Brownian motion)]
Let $\mu$ be a random measure on $\mathbb H$ and $B$ be a standard
one-dimensional Brownian motion. Suppose that
(i) $\mu$ is independent from $B$,
(ii) $\mu$ is a measure with independent increments,
(iii) $\mu$ is a L\'evy trap measure.
Then $B[\mu]$ is called \emph{Trapped Brownian Motion} (TBM) with
\emph{trap measure} $\mu$.
\end{definition}

The class of TBMs includes the following processes.

\begin{example}[(Speed-measure changed Brownian motion)]
\label{ex:speedmeasure}
Fix $\rho\in M(\mathbb R)$ (cf. the \hyperref[app]{Appendix} for the notation) and let
$\Leb_+$ be the Lebesgue measure on
$(\mathbb{R}_+, \mathcal B(\mathbb R_+))$. Define
$\mu:=\rho\otimes\Leb_+$. Then $\mu$ is a (deterministic) L\'evy trap
measure. Furthermore, as $\mu$ is deterministic, it is also a measure
with independent increments.

The TBM $B[\mu]$ is simply a time change of Brownian motion with speed
measure $\rho$. Indeed, this time change $B^\rho$ is usually defined as
%
\begin{equation}
\label{timechange} \bigl(B^\rho_t\bigr)_{t\geq0}:=(B_{\psi_{\rho}(t)})_{t\geq0}
\end{equation}
for $ \phi_{\rho}(s):=\int_{\mathbb{R}}\ell(x,s)\rho(dx)$ and
$\psi_{\rho}(t):=\inf\{s>0\dvtx \phi_{\rho}(s)>t \}$.
By Fubini's theorem, it is easy to see that
%
\begin{equation}\qquad\quad
\phi_{\rho}(s) = \int_{\mathbb{R}}\int
_0^{\ell(x,s)} \,dy\, \rho (dx) = (\rho\otimes\Leb_+)
(U_{\ell(\cdot,s)}) = \phi[\rho\otimes\Leb_+,B]_s.
\end{equation}
This implies that $B^\rho$ equals $B[\mu]$.
\end{example}

\begin{example}[(Fractional kinetics process)]
\label{ex:FK}
Let ${\mathcal P}=(x_i,y_i,z_i)_{i\in\mathbb{N}}$ be a Poisson point
process on $\mathbb H\times(0,\infty)$ with
intensity measure
%
\begin{equation}
\varrho=\gamma z^{-1-\gamma}\,dx \,dy \,dz,\qquad  \gamma\in(0,1).
\end{equation}
Define the random measure $\mu_\FK$ on $\mathbb H$ as
%
\begin{equation}
\label{e:muFK} \mu_\FK=\mu_\FK^\gamma:=\sum
_i z_i\delta_{(x_i,y_i)}.
\end{equation}

It is easy to see that for every compact $K\subset\mathbb H$, $\mu
_\FK(K)$ has a
$\gamma$-stable distribution with the scaling parameter proportional
to the Lebesgue measure of $K$. Further, as $\mathcal P$ is a Poisson
point process, we have that $\mu_\FK(K_1)$ and $\mu_\FK(K_2)$ are
independent when $K_1$, $K_2$ are disjoint. Thus, $\mu_\FK$ is a
measure with independent increments, and $\mu_\FK\langle A\rangle$ is
a stable L\'evy process for each bounded $A \in{\mathcal B} (\R)$, and
thus $\mu_\FK$ is a L\'evy trap measure.

The TBM $B[\mu]$ corresponding to this measure is
the FK process introduced in Definition~\ref{d:fk}. To see this, it is
enough to show that the process $(\phi[\mu,B]_t)_{t\in\mathbb
{R}_+}$ is
a $\gamma$-stable subordinator that is independent of $B$.

This can be proved as follows. Fix a realization of the Brownian motion
$B$. Then its local time is also fixed. As $\Leb(U_{\ell(\cdot,t)})
= t$
and
$U_{\ell(\cdot,s)}$, $(U_{\ell(\cdot,t)}\setminus U_{\ell(\cdot,s)})$
are disjoint sets
for every $s<t$, we have that $\phi[\mu,B]_t$ has $\gamma$-stable
distribution with the scaling parameter proportional to $t$, and
$\phi[\mu,B]_t-\phi[\mu,B]_s$ is independent of $\phi[\mu,B]_s$.
Hence, for every realization of $B$, $\phi[\mu,B]$ is a $\gamma$-stable
subordinator, and thus $\phi[\mu,B]$ is a $\gamma$-stable
subordinator independent of $B$.
\end{example}

The last important example goes in the direction of the SSBM.

\begin{example}
\label{ex:quenchedGFIN}
Let $k\in\mathbb{N}\cup\{\infty\}$ and $((S^i_t)_{t\geq0})_{i <
k}$ be
a family of independent subordinators. Let $(x_i)_{i < k}$ be real
numbers. Denoting by $dS^i$ the Lebesgue--Stieltjes measure corresponding
to $S^i$, it is immediate that
%
\begin{equation}
\mu(dx\otimes dy):=\sum_{i < k}\delta_{x_i}(dx)
\otimes dS^i(y)
\end{equation}
is a L\'evy trap measure with independent increments. The TBM $B[\mu]$
is a
process which is always located at some $x_i$.
\end{example}

\section{Randomly trapped random walk and randomly trapped Brownian
motion}
\label{s:rtrwrtbm}

The classes of trapped random walks and trapped Brownian motions are too
small to include some processes that we want to consider, in particular,
Bouchaud's trap model, the FIN diffusion and the projections of the
random walk on IIC, IPC. More precisely, quenched distributions of these
models (given corresponding random environments) are trapped random
walks. If we want to consider averaged distributions, we need to
introduce larger classes, \textit{randomly trapped random walks} and
\textit{randomly trapped Brownian motion}. Their corresponding random
measures will be constructed as mixtures of the respective trap measures.

The mixture of random measures is defined as follows. Let
$(\Omega, \mathcal F, \mathbb P)$ be a probability space, and let for
every $\omega\in\Omega$, $\mu_\omega$ be a random measure on
$\mathbb H$ defined on some other probability space
$(\tilde\Omega, \tilde{\mathcal F}, \tilde{\mathbb P})$. The random
measure $\mu\dvtx \Omega\times\tilde\Omega\to M(\mathbb H)$ given by
%
\begin{equation}
\mu(\omega,\tilde\omega) (A)= \mu_\omega(\tilde\omega) (A), \qquad A\in
\mathcal B(\mathbb H)
\end{equation}
is called mixture of $\mu_\omega$ with respect to $\mathbb P$. For
the reader's convenience, Proposition~\ref{p:existence} ensuring the
existence of the mixtures is included in the \hyperref[app]{Appendix}.

\subsection{Randomly trapped random walk}
\label{ss:rtrw}

\begin{definition}[(Randomly trapped random walk)]
\label{d:rtrw}
Let $(\Omega,\mathcal{F},\Pb)$ be a probability space and
$(\mu_\omega)_{\omega\in\Omega}$ a family of \emph{trap measures} on
a probability space
$(\tilde\Omega, \tilde{\mathcal F}, \tilde{\mathbb P})$ indexed by
$\omega\in\Omega$. Let $\mu$ be the mixture of
$(\mu_{\omega})_{\omega\in\Omega}$ with respect to~$\Pb$, and $Z$ a
simple random walk independent of $\mu$. Then the $\mu$-time changed
random walk $Z[\mu]$ is called \emph{Randomly Trapped Random Walk}
(RTRW) with \emph{trap measure}~$\mu$.
\end{definition}

\begin{definition}[(Random trapping landscape)]
\label{d:trappinglandscape}
Let $Z[\mu]$ be a RTRW where $\mu$ is the mixture of
$(\mu_\omega)_{\omega\in\Omega}$ w.r.t. $\Pb$. Let
$\bolds\pi:=(\pi_x)_{x\in\mathbb{Z}}\dvtx \Omega\to M_1((0,\infty
))^{\mathbb{Z}}$
be defined by stating that, for each $\omega\in\Omega$,
$\bolds\pi(\omega)$ is the trapping landscape of $Z[\mu
_{\omega}]$.
$\bolds\pi(\omega)$ is called the \emph{random trapping
landscape of} $\mu$.

Let $\mathbf P = \mathbb P\circ\bolds\pi^{-1}$ be the
distribution of $\bolds\pi$ on $M_1((0,\infty))^{\mathbb Z}$. If
$\mathbf P$ is a product measure, that is,
$\mathbf P=\bigotimes_{x\in\Z} P^x$ for some
$P_x\in M_1(M_1((0,\infty)))$, $x\in\mathbb Z$, then the coordinates
of the random trapping landscape $(\pi_x)_{x\in\Z}$ are independent. In
this case, we say that the random trapping landscape $\bolds\pi$
is \emph{independent}. If $\mathbf P=\bigotimes_{x\in\Z} P$ for
some $P\in M_1(M_1((0,\infty)))$, then the $(\pi_x)_{x\in\Z}$ are
i.i.d., and we say the random trapping landscape $\bolds\pi$ is
\emph{i.i.d.}
\end{definition}

As usual, we give some examples of RTRWs.

\begin{example}[(Bouchaud trap model)]
The symmetric one-dimensional \emph{Bouchaud trap model} (BTM) is a
symmetric continuous time random walk $X$ on $\Z$ with random jump
rates. More precisely, to each vertex $x\in\Z$ we assign a positive
number $\tau_x$ where $(\tau_x)_{x \in\Z}$ is an i.i.d. sequence of
positive random variables defined on a probability space
$(\Omega,\mathcal F,\Pb)$ such that
%
\begin{equation}
\label{e:btm-heavytails} \lim_{u\rightarrow\infty}u^\gamma\mathbb{P}[
\tau_z\geq u]=c,\qquad \gamma\in(0,1), c\in(0,\infty).
\end{equation}
Each visit of $X$ to $x\in\Z$ lasts an exponentially
distributed time with mean $\tau_x$.

It can be seen easily that the BTM is a RTRW. Its random trapping
landscape is given by
%
\begin{equation}
\bolds\pi(\omega)=(\nu_{\tau_x(\omega)})_{x\in\mathbb Z},
\end{equation}
where $\nu_a$ is the exponential distribution with mean $a$. As $\tau_x$
are i.i.d., the random trapping landscape $\bolds\pi$ is i.i.d.
\end{example}

\begin{example}[(Trap model with transparent traps)]
\label{ex:faketrap}
The trap model with transparent traps defined in Section~\ref{ss:transparenttraps} is a particular case of RTRW.
In Section~\ref{ss:faketrap} we will study the scaling limits of this process.
\end{example}

The following three examples of RTRW are of geometric nature. The first
(and the easiest) one is studied in this paper, the behavior of the
next two examples will be considered a follow-up paper.

\begin{example}[(Comb model)]
\label{ex:combmodel}
The Comb model defined in Section~\ref{ss:combpre} is a RTRW. Its
scaling limits are given in Theorem~\ref{t:phasediagramcombmodel} which
we prove in Section~\ref{ss:combmodel}.
\end{example}

\begin{example}[(Incipient critical Galton--Watson tree)]
Let $T$ be a rooted, regular tree of forward degree $g>1$. Let us perform
critical percolation on $T$ and denote by $\mathcal{C}_n$ the
percolation cluster of the root conditioned on reaching level $n$, that
is conditioned on having a vertex whose graph-distance from the root is
$n$. By letting $n\to\infty$ the trees $\mathcal{C}_n$ converge to the
\textit{Incipient infinite cluster} (IIC) (for details of this
construction, see \cite{kesten}). The IIC is an infinite random tree
and it can be shown that it has a single path to infinity, that is,
there is a single unbounded nearest neighbor path started at the root.
Such path is called the \textit{backbone}. The backbone is obviously
isomorphic (as a graph) to $\N$, hence the IIC can be seen as $\N$
adorned with dangling branches. We denote $\mathcal{L}_k$ the branch
emerging from the $k$th vertex of the backbone. Let
$(Y^{\IIC}_k)_{k\in\N}$ be a simple random walk on the IIC starting
from the root. Let $W^{\IIC}$ be the projection of $Y^{\IIC}$ to the
backbone. More precisely, let $(W^{\IIC}_t)_{t\geq0}$ be a
continuous-time random walk taking values in $\N$ defined by stating
that $W^{\IIC}_t=k$ if and only if
$Y^{\IIC}_{\lfloor t \rfloor}\in\mathcal{L}_k$. Then $W^{\IIC}$ is a
RTRW (disregarding for the moment the fact that
it takes values on $\N$ instead of $\Z$). In this case, the branches
$(\mathcal{L}_k)_{k\in\N}$ play the role of traps.
\end{example}

\begin{example}[(Invasion percolation cluster)]
\label{ex:ipc}
One can also consider, instead of the incipient infinite cluster, the
\textit{invasion percolation cluster} (IPC) on a regular tree. The
construction of the IPC is as follows: Recall that $T$ denotes a
rooted, regular tree of forward degree $g>1$. Let $(w_x)_{x\in T}$ be
an i.i.d. sequence of random variables uniformly distributed over
$(0,1)$. Set $\mathcal{I}^0:=\{\mbox{root of }T\}$ and
%
\begin{equation}
\mathcal{I}^{n+1}:=\mathcal{I}^n\cup \bigl\{x\dvtx d
\bigl(x,\mathcal{I}^n\bigr)=1 \mbox{ and } w_x=\min\bigl
\{w_z\dvtx d\bigl(\mathcal{I} ^n,z\bigr)=1\bigr\} \bigr
\},
\end{equation}
where $d$ is the graph distance in $T$. That is, $\mathcal{I}^{n+1}$ is
obtained from $\mathcal{I}^n$ by adding the vertex on the outer
boundary of $\mathcal{I}^n$ with the smallest ``weight.'' The
\emph{invasion percolation cluster} on $T$ is defined as
$\bigcup_{n\in\N}\mathcal{I}^n$. The IPC will be denoted as
$\mathcal{I}^{\infty}$. It can be shown (see \cite{ipc}) that, as the
IPC, the IIC possesses a single path to infinity. We can define a RTRW
$W^{\IPC}$ in the same way we have defined $W^{\IIC}$.
\end{example}

\subsection{Randomly trapped Brownian motion}

Finally, we define the randomly trapped Brownian motion analogously to RTRW.

\begin{definition}[(Randomly trapped Brownian motion)]
\label{d:rtrwbm}
Let a random measure $\mu$ be the mixture of
$(\mu_{\omega})_{\omega\in\Omega}$ with respect to $\Pb$, where for
each $\omega\in\Omega$, $\mu_{\omega}$ is a trap measure of a TBM.
Furthermore, let us suppose that $\mu$ is independent of the Brownian
motion $B$. Then $B[\mu]$ is called \emph{randomly trapped Brownian
motion} (RTBM) with \emph{trap measure} $\mu$.
\end{definition}

\begin{example}[(FIN diffusion)]
\label{ex:FIN}
Let $\mathcal{P}=(x_i,v_i)_{i\in\mathbb{N}}$ be a Poisson point process
on $\mathbb R\times(0,\infty)$ with intensity measure
$\gamma \,dx\, v^{-1-\gamma} \,dv$, $\gamma\in(0,1)$, defined on a
probability space $(\Omega,\mathcal{F},\mathbb{P})$. For each
$\omega\in\Omega$, let
$\mu_{\omega}:=\sum_{i\in\mathbb{N}}\delta_{x_i(\omega)}\otimes
v_i(\omega)\Leb_+$.
By Proposition~\ref{p:existence}, the mixture of
$(\mu_\omega)_{\omega\in\Omega}$ w.r.t. $\Pb$ exists and thus there
exists the mixture $\mu_\FIN$ of $(\mu_{\omega})_{\omega\in\Omega}$
w.r.t. $\Pb$.

Recalling Example~\ref{ex:speedmeasure}, it is easy to see that
$B[\mu_\omega]$ is a time change of $B$ with speed measure
$\rho(dx) = \sum_{i} v_i(\omega) \delta_{x_i(\omega)}(dx)$. Comparing
this with Definition~\ref{d:FIN}, we see that the RTBM corresponding
to $\mu_\FIN$, $B[\mu_\FIN]$, is a FIN diffusion.
\end{example}

\begin{example}[(Spatially subordinated Brownian motion)]
\label{ex:GFIN}
Recall from \eqref{e:fdPi} that $\mathfrak{F}^{\ast}$ is the
set of Laplace exponents of subordinators. Let $\mathbb{F}$ be a
$\sigma$-finite measure on $\mathfrak{F}^{\ast}$ satisfying the
assumption appearing in \eqref{e:assumptiononF} and let
$(x_i,f_i)_{i\ge0}$ be a Poisson point process on
$\mathbb R\times\mathfrak F^\ast$ defined on a probability space
$(\Omega, \mathcal F, \mathbb P)$ with intensity $dx \otimes\mathbb F$.
Let $(S_t^i)_{t\ge0}$, $i\ge0$, be a family of independent
subordinators, Laplace exponent of $S^i$ being $f_i$, defined on a
probability space
$(\tilde\Omega, \tilde{\mathcal F}, \tilde{\mathbb P})$.

For a given realization of $(x_i,f_i)_{i\ge0}$, we set similarly as
in Example~\ref{ex:quenchedGFIN}
%
\begin{equation}
\mu_{(x_i,f_i)} (dx\otimes dy) = \sum_{i\ge0}
\delta _{x_i}(dx)\otimes dS^i(y).
\end{equation}
It follows that for a fixed realization of $(x_i,f_i)_{i\ge0}$, the
measure $\mu_{(x_i,f_i)} $ is a L\'evy trap measure with
independent increments.

Using Proposition~\ref{p:existence}, we can show that the mixture of
$(\mu_{(x_i(\omega), f_i(\omega))})_{\omega\in\Omega}$
w.r.t.~$\Pb$,
%
\begin{equation}
\label{e:muGFIN} \mu_\GFIN^{\mathbb F}(\omega,\tilde\omega):=
\mu_{(x_i(\omega),f_i(\omega))} (\tilde\omega)
\end{equation}
is a random measure. The corresponding RTBM is the $\mathbb{F}$-spatially
subordinated Brownian motion introduced in Definition~\ref{d:GFIN}.
\end{example}


\section{Convergence of processes}
\label{s:convergence}

We study now the convergence of various classes of processes introduced in
the previous section.

\subsection{Convergence of time changed random walks} 

We start by presenting the basic convergence theorems for $\mu$-time
changed random walks and $\mu$-time changed Brownian motions. These
theorems allow one deduce the convergence of processes (TRWs, TBMs, RTRWs,
RTBMs) from the convergence of their associated random measures. This,
in turn, makes possible to use the well-developed theory of convergence
of random measures; see, for example, \cite{rmeasures}.

First, we need few additional definitions. We say that a random measure
$\mu$ is \emph{dispersed} if
%
\begin{equation}\qquad
\mu\bigl(\bigl\{(x,y)\in\mathbb H\dvtx y=f(x)\bigr\}\bigr)=0\qquad \mbox{almost
surely, for all $f\in C_0(\mathbb R,\mathbb R_+)$}
\end{equation}
(here, $C_0$ stands for the space of continuous functions with compact
support).
We say that a random measure $\mu$ is \emph{infinite} if
$\mu(\mathbb H)=\infty$, almost surely.
We say that $\mu$ is \emph{dense} if its support is $\mathbb H$, almost surely.

We write $D(\mathbb R_+)$, $D(\mathbb R)$ for the sets of real-valued
\textit{cadlag} functions on $\mathbb R_+$, or $\mathbb R$, respectively.
We endow these sets either with the standard Skorokhod $J_1$-topology, or
with the so called $M_1$-topology, and write $D(\mathbb R_+,J_1)$,
$D(\mathbb R_+,M_1)$ when we want to stress the topology used. Also,
$D(\mathbb R_+,U)$ will denote $D(\mathbb R_+)$ endowed with the
uniform topology. For definitions and properties of these topologies, see
\cite{Whi02}, Chapters~12 and 13.

Let $\mu$ be a random measure and $\varepsilon>0$.
We define the scaled random measure $\mathfrak{S}_{\varepsilon}(\mu
)$ by
%
\begin{equation}
\mathfrak{S}_{\varepsilon}(\mu) (A):=\mu\bigl(\varepsilon^{-1}A
\bigr)\qquad \mbox{for each $A\in{\mathcal B}(\mathbb{H})$.}
\end{equation}

Our first theorem states that the convergence of $\mu$-time changed
random walks can be deduced from the convergence of associated measures.
As it does not complicate the proof, we allow $\mu$ to be random.

\begin{theorem}[(Convergence of time changed random walks)]
\label{t:delayedrwconv}
Let $\mu^\varepsilon$, $\varepsilon>0$, be a family of infinite random
measures supported on $\mathbb Z\times\mathbb N$, and let
$Z$ be a simple random walk independent of them.
Assume that there exists a nondecreasing
function $q\dvtx \mathbb{R}_+\to\mathbb{R}_+$ with
$\lim_{\varepsilon\to0}q(\varepsilon)=0$, such that, as
$\varepsilon\to0$,
$q(\varepsilon)\mathfrak{S}_{\varepsilon}(\mu^\varepsilon)$ converges
vaguely in distribution to a dispersed, infinite, dense random measure
$\mu$.
Then the corresponding time changed random walks $Z[\mu^\varepsilon ]$
converge after rescaling to the time changed Brownian motion $B[\mu]$,
%
\begin{equation}
\bigl(\varepsilon Z\bigl[\mu^{\varepsilon}\bigr]_{q(\varepsilon)^{-1}t}
\bigr)_{t\ge0} \xrightarrow{\varepsilon\to0} \bigl(B[\mu]_t
\bigr)_{t\ge0}
\end{equation}
in distribution on $D(\mathbb R_+,J_1)$. Here $B$ is a Brownian motion
independent of $\mu$.
\end{theorem}

The next theorem, which we will not need later in the paper, gives a
similar criteria for convergence of time changed Brownian motions. We present
it as it has intrinsic interest and because its proof is a simplified
version of the proof of Theorem~\ref{t:delayedrwconv}.

\begin{theorem}
\label{t:delayedbmconv}
Let $\mu^\varepsilon$, $\varepsilon>0$, be a family of infinite random
measures on $\mathbb H$, and let
$B$ be a Brownian motion independent of them.
Assume that, as
$\varepsilon\to0$,
$\mu^\varepsilon$ converges
vaguely in distribution to a dispersed, infinite, dense random measure
$\mu$.
Then the corresponding time changed Brownian motions $B[\mu
^\varepsilon ]$
converge to $B[\mu]$,
%
\begin{equation}
\bigl(B\bigl[\mu^{\varepsilon}\bigr]_{t}\bigr)_{t\ge0}
\xrightarrow{\varepsilon\to0} \bigl(B[\mu]_t\bigr)_{t\ge0},
\end{equation}
in distribution on $D(\mathbb R_+,J_1)$.
\end{theorem}

\begin{pf}
As $\mu^{\varepsilon}$ converges vaguely in distribution to $\mu$, in
virtue of the Skorokhod representation theorem, there exist random
measures $(\bar{\mu}^{\varepsilon})_{\varepsilon>0}$ and $\bar{\mu
}$ on
$\mathbb H$ defined on a common probability space
$(\tilde{\Omega},\tilde{{\mathcal F}},\tilde{\mathbb{P}})$, such that
$\bar{\mu}^{\varepsilon}$ is distributed as $\mu^\varepsilon$,
$\bar{\mu}$
is distributed as $\mu$ and $\bar{\mu}^{\varepsilon}$ converges vaguely
to $\bar{\mu}$ as $\varepsilon\to0$, $\tilde{\Pb}$-a.s. Without
loss of generality, we can suppose that on the space
$(\tilde{\Omega},\tilde{{\mathcal F}},\tilde{\mathbb{P}})$ there is
defined a one-dimensional standard Brownian motion
$(B_t)_{t\ge0}$ independent of
$(\bar{\mu}^{\varepsilon})_{\varepsilon>0}$ and~$\bar{\mu}$.

First, we show that
$\phi[\bar{\mu}^{\varepsilon},B] \to\phi[\bar{\mu},B]$ in
$D(\mathbb{R}_+,M_1)$, $\tilde{\Pb}$-a.s. as $\varepsilon\to0$:
Using that $\bar{\mu}$ is a dispersed random measure,
%
\begin{equation}
\tilde\Pb\bigl[\bar{\mu}(\partial U_{\ell(\cdot,t)})=0, \forall0\le t\in \mathbb
Q\bigr] = 1,
\end{equation}
where $\partial A$ denotes the boundary of $A$ in $\mathbb H$.
Since $U_{\ell(\cdot,t)}$ is a bounded set, this implies that for all
$0\le t\in\mathbb Q$
%
\begin{equation}
\phi\bigl[\bar{\mu}^{\varepsilon},B\bigr]_t=\bar{
\mu}^{\varepsilon
}(U_{\ell(\cdot,t)}) \xrightarrow{\varepsilon\to0} \bar{
\mu}(U_{\ell(\cdot,t)})=\phi[\bar{\mu},B]_t, \qquad\tilde{\Pb}\mbox{-a.s.}
\end{equation}
Since, by \cite{Whi02}, Theorems 12.5.1 and 13.6.3, on the set of monotonous
functions the convergence on $D(\mathbb R_+,M_1)$ is equivalent to
pointwise convergence on a dense subset including $0$ and
since $\phi[\bar{\mu}^{\varepsilon},B]$ and $\phi[\bar{\mu},B]$ are
nondecreasing in $t$, we know that
%
\begin{equation}
\label{phiconvergence} \phi\bigl[\bar{\mu}^{\varepsilon},B\bigr]\to\phi[\bar{\mu},B]
\end{equation}
in
$D(\mathbb{R}_+,M_1)$, $\tilde\Pb$-a.s., as claimed.

Since the random measures $\bar{\mu}^{\varepsilon}$ and $\bar{\mu
}$ are
infinite, the functions $\phi[\bar{\mu}^{\varepsilon},B]$ and
$\phi[\bar{\mu},B]$ are unbounded. As, by hypothesis, $\bar{\mu}$ is
dense, then the function $\phi[\bar{\mu},B]$ will be strictly
increasing. Hence, \cite{Whi02}, Corollary~13.6.4, allows us to deduce
uniform convergence of $\psi[\bar{\mu}^\varepsilon,B]$ to
$\psi[\bar{\mu},B]$ from \eqref{phiconvergence}.

Using the continuity of the Brownian paths and
\cite{Whi02}, Theorem~13.2.2, we get that
$B[\bar{\mu}^\varepsilon]_t\to B[\bar{\mu}]_t$ in the $J_1$-topology.
$\bar{\mu}^\varepsilon$ and $\bar{\mu}$ are distributed as the
$\mu^\varepsilon$ and~$\mu$, respectively, the convergence in
distribution of $B[{\mu}^{\varepsilon}]$ to $B[{\mu}]$ follows.
\end{pf}

\begin{pf*}{Proof of Theorem~\ref{t:delayedrwconv}}
As $q(\varepsilon)\mathfrak{S}_{\varepsilon}(\mu^\varepsilon)$
converges vaguely in distribution to $\mu$, we can, in virtue of the
Skorokhod representation theorem, construct random measures
$(\bar{\mu}^{\varepsilon})_{\varepsilon>0}$ and $\bar{\mu}$
defined on a
common probability space
$(\tilde{\Omega},\tilde{{\mathcal F}},\tilde{\mathbb{P}})$, such that
$\bar{\mu}^{\varepsilon}$ is distributed as
$q(\varepsilon)\mathfrak{S}_\varepsilon(\mu^\varepsilon)$, $\bar
{\mu}$
is distributed as $\mu$, and $\bar{\mu}^{\varepsilon}$ converges
vaguely to $\bar{\mu}$ as $\varepsilon\to0$, $\tilde{\Pb}$-a.s. Without
loss of generality, we can suppose that on the space
$(\tilde{\Omega},\tilde{{\mathcal F}},\tilde{\mathbb{P}})$ there is
defined a one-dimensional standard Brownian motion
$(B_t)_{t\ge0}$ independent of
$(\bar{\mu}^{\varepsilon})_{\varepsilon>0}$ and $\bar{\mu}$.

Set $B^{\varepsilon}_t:=\varepsilon^{-1}B_{\varepsilon^2t}$. For each
$\varepsilon>0$, we define a sequence of stopping times
$(\sigma^{\varepsilon}_k)_{k=0}^{\infty}$ by $\sigma^{\varepsilon}_0:=0$,
%
\begin{equation}
\sigma^{\varepsilon}_k:=\inf{\bigl\{t>\sigma^{\varepsilon}_{k-1}
\dvtx B^{\varepsilon}_t\in\mathbb{Z}\setminus \bigl
\{B^{\varepsilon}_{\sigma^{\varepsilon}_{k-1}}\bigr\}\bigr\}}.
\end{equation}
Then the process $(Z^{\varepsilon}_k)_{k\in\mathbb{N}}$ defined by
$Z^{\varepsilon}_k:=B^{\varepsilon}_{\sigma^{\varepsilon}_k}$ is a
simple symmetric random walk on $\mathbb{Z}$. We define the local time
of $Z^{\varepsilon}$ as
$L^{\varepsilon}(x,s):=\sum_{i=0}^{\lfloor s \rfloor}1_{\{Z^\varepsilon_i=\lfloor x
\rfloor\}}$.
Define
%
\begin{equation}
\bar\phi_s^{\varepsilon}= q(\varepsilon)^{-1}
\mathfrak{S}_{\varepsilon^{-1}}\bigl(\bar{\mu }^{\varepsilon}\bigr)
(U_{L^{\varepsilon}(\cdot,s)}),\qquad s\ge0,\varepsilon>0.
\end{equation}
Note that
$q(\varepsilon)^{-1}\mathfrak{S}_{\varepsilon^{-1}}(\bar{\mu
}^{\varepsilon})$
is distributed as $\mu^\varepsilon$. Hence,
$(\bar\phi^\varepsilon_t)_{t\ge0}$ is distributed as
$(\mu^\varepsilon(U_{L^1(\cdot,t)}))_{t\ge0} = (\phi[\mu
^\varepsilon]_t)_{t\ge0}$.
Hence, denoting
$\bar\psi^\varepsilon_t:=\inf\{s>0\dvtx \bar\phi^{\varepsilon}_s>t\}
$, we see
that for each $\varepsilon>0$, the process
$(Z^{\varepsilon}_{\bar\psi^{\varepsilon}_t})_{t\geq0}$ is distributed
as $(Z[\mu^{\varepsilon}]_{t})_{t\geq0}$.

The proof of Theorem~\ref{t:delayedrwconv} relies on the following two
lemmas.

\begin{lemma}
\label{l:Kt}
For each $t\geq0$, there exists a random compact set $K_t$ such that
$\bigcup_{\varepsilon>0} \supp L^{\varepsilon}(\varepsilon
^{-1}\cdot,\varepsilon^{-2}t)$
is contained in $K_t$.
\end{lemma}

\begin{pf}
By the strong Markov property for the Brownian motion $B$, for each
$\varepsilon>0$,
$(\sigma^{\varepsilon}_k-\sigma^\varepsilon_{k-1})_{k>0}$ is an
i.i.d. sequence with
$\tilde{\mathbb E}[\sigma^{\varepsilon}_{i}-\sigma^{\varepsilon}_{i-1}]=1$.
Thus, by the strong law of large numbers for triangular arrays,
$\tilde{\Pb}$-almost surely, there exists a (random) constant $C$
such that
$\varepsilon^{2}\sigma^{\varepsilon}_{\lfloor\varepsilon
^{-2}t\rfloor}\leq C$
for all $\varepsilon>0$. Thus, for each $\varepsilon>0$, the support
of $L^{\varepsilon}(\varepsilon^{-1}\cdot,\varepsilon^{-2}t)$ is
contained in the support of $\ell(\cdot,C)$. Therefore, it is
sufficient to choose $K_t=\supp(\ell(\cdot,C))$.
\end{pf}

\begin{lemma}
\label{l:convdereloj}
$(q(\varepsilon)\bar\phi^\varepsilon_{\varepsilon^{-2}t})_{t\ge0}
\xrightarrow{\varepsilon\to0} (\phi[\bar{\mu},B]_t)_{t\ge0}$
$\tilde\Pb$-a.s. on $(D(\mathbb{R}_+),M_1)$.
\end{lemma}

\begin{pf}
It is easy to see that
%
\begin{equation}
\label{e:1} q(\varepsilon)\bar\phi_{t\varepsilon^{-2}}^{\varepsilon} =
\mathfrak{S}_{\varepsilon^{-1}} \bigl(\bar{\mu}^{\varepsilon}\bigr)
(U_{L^\varepsilon(\cdot,\varepsilon^{-2}t)}) =\bar{\mu}^{\varepsilon}
(U_{\varepsilon L^\varepsilon(\varepsilon^{-1}\cdot,\varepsilon^{-2}t)}).
\end{equation}
By \cite{borodin}, Chapter IV, Theorem~2.1, for each $t\geq0$,
$\tilde{\mathbb P}$-a.s.,
$\varepsilon L^{\varepsilon}(\varepsilon^{-1}x,\varepsilon^{-2}t)\to
\ell(x,t) $
uniformly in $x$, as $\varepsilon\to0$. Thus for any $\eta>0$ there exists
$\varepsilon_\eta$ such that, if $\varepsilon<\varepsilon_\eta$ we
will have that
$\varepsilon L^\varepsilon(\varepsilon^{-1}\cdot,\varepsilon
^{-2}t)\leq\ell(\cdot,t)+\eta$.
Note that $\ell(\cdot,t)+\eta$ is not compactly supported. Let
$h_{\eta}\dvtx \mathbb H\to\mathbb R_+$ be a continuous function which
for every $t\ge0$ coincides with $\ell(\cdot,t)+\eta$ on $K_t$,
$h_\eta(\cdot,t)\leq\eta$ outside $K_t$, and $h_{\eta}(\cdot,t)$ is
supported on $[\inf K_t-\eta, \sup K_t+\eta]$. Using Lemma~\ref{l:Kt}, we find that
$\varepsilon L^\varepsilon(\varepsilon^{-1}\cdot,\varepsilon
^{-2}t)\leq h_\eta(\cdot,t)$.
Thus,
%
\begin{equation}
\label{e:2} \bar{\mu}^{\varepsilon} (U_{\varepsilon L^\varepsilon(\varepsilon^{-1}\cdot,\varepsilon^{-2}t)}) \leq\bar{
\mu}^\varepsilon(U_{h_\eta(\cdot,t)}).
\end{equation}
As $\bar{\mu}$ is a dispersed random measure, for fixed $t$,
$\bar{\mu}(\partial U_{h_\eta(\cdot,t)}) =\bar{\mu}(\partial
U_{\ell(\cdot,t)})=0$,
$\tilde{\Pb}$-a.s. For any $\delta>0$ and all $\varepsilon$ small
enough (depending on $\delta$), as $\bar{\mu}^{\varepsilon}$
converges vaguely to $\bar{\mu}$,
%
\begin{equation}
\label{e:3} \bar{\mu}^\varepsilon(U_{h_\eta(\cdot,t)}) \leq\bar{
\mu}(U_{h_\eta(\cdot,t)})+\delta/2.
\end{equation}
For each $\delta> 0$, there exists $\eta>0$ such that
$\bar{\mu}(U_{h_\eta(\cdot,t)})\leq\bar{\mu}(U_{\ell(\cdot,t)})+\delta/2$.
Combining this with \eqref{e:1}--\eqref{e:3}, we find that
%
\begin{equation}
\limsup_{\varepsilon\to0} q(\varepsilon)\bar\phi_{t\varepsilon^{-2}}^{\varepsilon}
=\limsup_{\varepsilon\to0} \bar{\mu}^{\varepsilon} (U_{\varepsilon L^\varepsilon(\varepsilon^{-1}\cdot,\varepsilon^{-2}t)})
\leq\phi[\bar{\mu},B]_t.
\end{equation}
A lower bound can be obtained in a similar way. Hence, after taking
union over $0\le t \in\mathbb Q$,
%
\begin{equation}
\tilde\Pb\Bigl[\lim_{\varepsilon\rightarrow0} q(\varepsilon) \bar
\phi_{\varepsilon^{-2}t}^{\varepsilon}= \phi[\bar\mu,B]_t, \forall0\le
t\in\mathbb Q\Bigr] =1.
\end{equation}
Since $\bar\phi^{\varepsilon}_t$ and $\phi[\bar{\mu},B]$ are
nondecreasing in $t$,
$(q(\varepsilon)\bar\phi^{\varepsilon}_{\varepsilon^{-2}t})_{t\ge0}$
converges to $(\phi[\bar{\mu},B]_t)_{t\ge0}$, $\tilde\Pb$-a.s. on
$(D(\mathbb{R}_+),M_1)$, completing the proof of the lemma.
\end{pf}

Theorem~\ref{t:delayedrwconv} then follows from
Lemma~\ref{l:convdereloj} by repeating the arguments of the last paragraph
in the proof of Theorem~\ref{t:delayedbmconv}.
\end{pf*}

\subsection{Convergence of trapped processes} 

The class of time changed random walks is very large, and the associated
convergence criteria rather general. Applying these criteria, however,
requires to check the convergence of the underlying random measures,
which might be complicated in many situations.

As we have seen, the underlying random measures of trapped processes
(TRW, TBM) satisfy additional assumptions. This will make checking their
convergence easier than in the general case.

\begin{proposition}
\label{p:Levymeasureconv}
\textup{(i)} Let $\mu^\varepsilon$, $\mu$ be L\'evy trap measures with independent
increments (i.e., they are trap measures of some TBMs). Then
$\mu^\varepsilon$ converges vaguely in distribution to $\mu$, iff
$\mu^\varepsilon(I\times[0,1])$ converges in distribution to
$\mu(I\times[0,1])$ for every compact interval $I=[a,b]$ such that
$\mu(\{a,b\}\times\mathbb R_+)=0$, $\tilde{\mathbb P}$-a.s.

\textup{(ii)} The same holds true if
$\mu^\varepsilon= \mathfrak S_\varepsilon(\nu^\varepsilon)$ for a
family of trap measures $\nu^\varepsilon$ of some TRWs.
\end{proposition}

\begin{pf}
We will use the well-known criteria for the convergence of random measures
recalled in Proposition~\ref{p:kalconv} in the \hyperref[app]{Appendix}.
When $\mu$ is a L\'evy trap measure with independent increments, the
distribution of $\mu([a,b]\times[c,d])$, $a,b\in\mathbb R$,
$c,d\in\mathbb R_+$, is determined by the distribution of
$\mu([a,b]\times[0,1])$, since by definition $\mu\langle
[a,b]\rangle$
is a L\'evy process. In particular, the assumptions of the proposition
imply the convergence in distribution of $\mu^\varepsilon(A)$ to
$\mu(A)$ for every $A\in\mathcal A$ where $\mathcal A$ is the set of
all rectangles $I\times[c,d]$ with $I$ as in the statement of the
proposition and $d\ge c\ge0$.

As $\mu\langle I\rangle$ is a L\'evy process, we have
$\mathcal A\subset\mathcal T_\mu$ [see \eqref{e:Tmu} for the notation].
Moreover, it is easy to see that
$\mathcal A$ is a DC semiring. The fact that $\mu^\varepsilon$ are
measures with independent increments combined with the well-known
criteria for vague convergence in distributions of random measures (see
Proposition~\ref{p:kalconv} in the \hyperref[app]{Appendix})
then implies claim (i).

The proof of claim (ii) is analogous. It suffices to observe that the
distribution of $\nu^\varepsilon$ is determined by distributions of
$\mu^\varepsilon([a,b]\times[0,1])$, $a,b\in\mathbb R$, as well.
\end{pf}

We apply this proposition in few examples.

\begin{example}[(Stone's theorem)]
Let $\rho_\varepsilon\in M(\mathbb R)$, $\varepsilon>0$, be a family
of measures on $\mathbb{R}$. Assume that, as $\varepsilon\to0$,
$\rho_{\varepsilon}$ converges vaguely to a measure
$\rho\in M(\mathbb R)$ whose support is $\mathbb R$. Set
$\mu_\varepsilon= \rho_\varepsilon\otimes\Leb_+$,
$\mu= \rho\otimes\Leb_+$. We have seen in
Example~\ref{ex:speedmeasure} that $\mu_\varepsilon$ and $\mu$ are
L\'evy trap measures with independent increments, and that
$B[\mu_\varepsilon]$ and $B[\mu]$ are a time changes of Brownian
motion with speed measure $\rho_{\varepsilon}$ and $\rho$,
respectively. Let $a,b$ be such that $\rho(\{a,b\})=0$, and thus
$\mu(\{a,b\}\times\mathbb R_+)=0$. By vague convergence of
$\rho_\varepsilon$ to $\rho$,
$\mu_\varepsilon([a,b]\times[0,1])\to\mu([a,b]\times[0,1])$. Also,
$\mu$ is a dispersed, infinite and dense random measure (because the
support of $\rho$ is $\R$). Therefore, by
Proposition~\ref{p:Levymeasureconv}, $\mu_\varepsilon$ converges
vaguely to $\mu$, and thus, by Theorem~\ref{t:delayedbmconv},
$B[\mu_\varepsilon]$ converges in distribution to $B[\mu]$ in
$D(\mathbb R_+,J_1)$.

This result is well known and was originally obtained by Stone
\cite{stone}. His result states that convergence of speed measures
implies convergence of the corresponding time-changed Brownian motions.
Thus, Theorem~\ref{t:delayedbmconv} can be viewed a generalization of
Stone's result.
\end{example}

\begin{example}
Let $\mu$, $Z[\mu]$ be as in Example~\ref{ex:MWCTRW} (a continuous-time
random walk \`a la Montroll--Weiss). Then, using
Theorem~\ref{t:delayedrwconv} and Proposition~\ref{p:Levymeasureconv},
we can prove that
$(\varepsilon Z[\mu]_{\varepsilon^{-2/\gamma}t})_{t\ge0}$ converges in
distribution to the FK process. (This result
was previously obtained in \cite{mula}.)

Indeed, let $K_\gamma$ be a positive stable law of index $\gamma$. It
is easy to see that $\mu$ is a trap measure corresponding to a TRW.
Example~\ref{ex:FK} implies that FK process is a trapped Brownian
motion whose corresponding trap measure $\mu_\FK$ is L\'evy. Moreover,
from the fact that $\mu_\FK$ is defined via Poisson point process whose
intensity has no atoms, we see that for every $a\in\mathbb R$,
$\mu_\FK(a\times\mathbb R_+)=0$, $\tilde\Pb$-a.s.

To apply Proposition~\ref{p:Levymeasureconv}, we should check that
$\varepsilon^{2/\gamma}\mathfrak{S}_\varepsilon(\mu)([a,b]\times[0,1])$
converges in distribution to $(b-a)^{1/\gamma}K_{\gamma}$.
However,
%
\begin{equation}
\label{sum} \varepsilon^{2/\gamma}\mathfrak{S}_\varepsilon(\mu)
\bigl([a,b]\times[0,1]\bigr) =\varepsilon^{2} \sum
_{x=a\varepsilon^{-1}}^{b\varepsilon^{-1}}\sum_{j=1}^{\varepsilon^{-1}}
s_x^j,
\end{equation}
where, by their definition in Example~\ref{ex:MWCTRW}, the
$(s_x^j)_{x\in\Z,j\in\N}$ are i.i.d. random variables in the
domain of attraction of the $\gamma$-stable law. The classical result
on convergence of i.i.d. random variables (see, e.g., \cite{gnedenko})
yields that \eqref{sum} converges in distribution to
$(b-a)^{1/\gamma}K_\gamma$. On the other hand, it is easy to see that
$\mu_\FK$ is an infinite, dispersed and dense random measure. The
convergence of processes then follows from Theorem~\ref{t:delayedrwconv}.
\end{example}

We finish this section with a lemma that shows that the trap measures of
TBMs are always dispersed, which simplifies checking the assumptions of
Theorem~\ref{t:delayedrwconv}.

\begin{lemma}
\label{l:borde}
Let $ \mu$ be a L\'evy trap measure with independent increments
defined on a
probability space $(\Omega,{\mathcal F},\Pb)$ and
$f\in C_0(\mathbb R, \mathbb R_+)$. Then $\Pb$-a.s.
%
\begin{equation}
\mu\bigl(\bigl\{(x,y)\in\mathbb H\dvtx y=f(x)\bigr\}\bigr)=0,
\end{equation}
that is $\mu$ is a dispersed trap measure.
\end{lemma}

\begin{pf}
Let $I^{n,i}=[(i-1)2^{-n},i2^{-n})$, and set
$m^{n,i}=\inf\{f(x)\dvtx x\in I^{n,i}\}$,
$M^{n,i}=\sup\{f(x)\dvtx x\in I^{n,i}\}$. Let
$B^{n,i}$ be the boxes
%
\begin{equation}
B^{n,i}:=I^{n,i}\times\bigl[m^{n,i}, M^{n,i}
\bigr].
\end{equation}
Then for all $n$ we have
%
\begin{equation}
\label{e:boxes} \bigl\{(x,y)\in\mathbb H\dvtx y=f(x)\bigr\} \subset\bigcup
_i B^{n,i},
\end{equation}
and $B^{n+1,2i-1}\cup B^{n+1,2i} \subset B^{n,i}$,
which implies that $\mu(\bigcup_i B^{n,i})$ is nonincreasing in $n$.

The uniform continuity of $f$ implies that for each $\delta>0$, there
exists $n_\delta$, such that for each $n>n_\delta$ and $i\in\mathbb Z$,
$M^{n,i}-m^{n,i}<\delta$. Since $\mu(B^{n,i})$ is distributed as
$\mu( I^{n,i} \times[0,M^{n,i}-m^{n,i}])$,
$\mu(B^{n,i})$ is stochastically dominated by
$\mu({ I^{n,i}\times[0,\delta]})$.
If $I^{n,i}\cap\supp f=\varnothing$, then
$M^{n,i}=m^{n,i}=0$. Hence, writing $J$ for the $1$-neighborhood
of $\supp f$, in the stochastic domination sense,
%
\begin{equation}
\mu \biggl(\bigcup_i B^{n,i} \biggr)\leq
\mu\bigl(J \times[0,\delta]\bigr).
\end{equation}
Since
${\mu(J \times[0,\delta])\xrightarrow{\delta\to0}0}$, $\Pb$-a.s.,
we see that $\mu(\bigcup_i B^{n,i})\xrightarrow{n\to\infty} 0$
in distribution. Together with the monotonicity of
$\mu(\bigcup_i B^{n,i})$, this implies that the convergence holds
$\Pb$-a.s. The lemma follows using \eqref{e:boxes}.
\end{pf}

\section{Convergence of RTRW to RTBM} 
\label{s:limits}
\label{s:iidconvergence}

In this section, we give the proofs of the convergence theorems stated in
Section~\ref{s:results}. First we prove
Theorem~\ref{t:scalinglimitclassification}, which gives a complete
characterization of the set of processes that appear as the scaling limit
of such RTRWs. We then provide the proofs of Theorems
\ref{t:BMconvergence}, \ref{t:FKconv}, \ref{t:RSBMconv} and~\ref{t:FINconv}. We recall that these theorems formulate criteria
implying the convergence of RTRWs to several limiting processes. Remark,
however, that our goal is not to characterize completely their domains of
attraction. Instead of this, we try to state natural criteria which can be
easily checked in applications.

\subsection{Set of limiting processes}
\label{ss:classification}
This section contains the proof of Theorem~\ref{t:scalinglimitclassification}.
We need a simple lemma first.

\begin{lemma}
\label{l:converse}
Let $X$ be a RTRW with i.i.d. trapping landscape $\bolds\pi$ and
random trap measure $\mu$.
Assume that for some
nondecreasing function $\rho(\varepsilon)$ satisfying
$\lim_{\varepsilon\to0} \rho(\varepsilon)=0$, the processes
$X^\varepsilon_\cdot:= \varepsilon X_{\rho(\varepsilon)^{-1}\cdot}$
converge in distribution on $D(\mathbb R_{+}, J_1)$ to some process $U$
satisfying the nontriviality assumption
$\limsup_{t\to\infty} |U_t| = \infty$ a.s. Then the family of measures
$\mu^\varepsilon:=\rho(\varepsilon)\mathfrak S_{\varepsilon}(\mu
)$ is
relatively compact for the vague convergence in distribution.
\end{lemma}

\begin{pf}
By \cite{Kal02}, Lemma~16.15, a sequence $\mu^\varepsilon$ of random
measures on $\mathbb H$ is relatively compact for the vague convergence
in distribution iff for every compact $A\subset\mathbb H$ the family of
random variables $(\mu^\varepsilon(A))_{\varepsilon>0}$ is tight in
the usual sense.

Assume now, by contradiction, that $(\mu^\varepsilon)$ is not
relatively compact. Then there exists $A\subset\mathbb H$ compact and
$\delta>0$ such that
%
\begin{equation}
\limsup_{\varepsilon\to0}\mathbb P\bigl[\mu^\varepsilon(A)>K\bigr]\ge
\delta \qquad\mbox{for all $K>0$}.
\end{equation}
Let $Z$ be a simple random walk on $\mathbb Z$ independent of $\mu$
such that $X=Z[\mu]$, and let $L(\cdot,\cdot)$ be its local time.
Since, uniformly in $\varepsilon\in(0,1)$,
$\varepsilon U_{L(\cdot,\varepsilon^{-2}t)}\xrightarrow{t\to\infty
}\mathbb H$,
it is possible to choose $t$ and $M$ large such that
%
\begin{equation}
\liminf_{\varepsilon\to0} \mathbb P \Bigl[ A\subset U_{\varepsilon L(\cdot, \varepsilon^{-2}t)},
\sup_{s\le\varepsilon^{-2}t}\bigl |\varepsilon Z(s)\bigr|\le M \Bigr] \ge1/2.
\end{equation}
Since the simple random walk $Z$ and $\mu$ are independent,
this implies, using the identity
$\rho(\varepsilon)\phi[\mu,Z]_{t\varepsilon^{-2}} = \mu
^\varepsilon(U_{\varepsilon L(\cdot, \varepsilon^{-2}t)})$,
%
\begin{equation}
\limsup_{\varepsilon\to0} \mathbb P \Bigl[\rho(\varepsilon)\phi[
\mu,Z]_{t\varepsilon^{-2}} \ge K, \sup_{s\le\varepsilon^{-2}t} \bigl|\varepsilon Z(s)\bigr|\le
M \Bigr] \ge\varepsilon/2,
\end{equation}
and thus
%
\begin{equation}
\limsup_{\varepsilon\to0} \mathbb P \Bigl[\psi[\mu,Z ]_{K\rho(\varepsilon)^{-1}}\le
t\varepsilon^{-2}, \sup_{s\le\varepsilon^{-2}t} \bigl|\varepsilon Z(s)\bigr|\le
M \Bigr] \ge\varepsilon/2.
\end{equation}
As
$X^\varepsilon = \varepsilon Z_{\psi[\mu,Z]_{t\rho(\varepsilon)^{-1}}}$,
and $K$ is arbitrary
%
\begin{equation}
\limsup_{\varepsilon\to0}\mathbb P \Bigl[\sup_{s<\infty
}\bigl|X^\varepsilon
(s)\bigr|\le M \Bigr]\ge \varepsilon/2,
\end{equation}
which contradicts the nontriviality assumption on the limit $U$.
\end{pf}

\begin{pf*}{Proof of Theorem~\ref{t:scalinglimitclassification}}
Let $\mu$ be the random trap measure of the RTRW $X$, and
$\mu^\varepsilon = \rho(\varepsilon)\mathfrak S_\varepsilon(\mu)$.
In view of Lemma~\ref{l:converse} and the assumptions of the theorem,
the family $(\mu^\varepsilon)$ is relatively compact. Therefore, there
is a sequence $\varepsilon_k$ tending to~$0$ as $k\to\infty$
such that $\mu^{\varepsilon_k}$ converges vaguely in distribution.

To show the theorem, we should thus first characterize all possible
limit points of random trap measures of RTRW's with i.i.d. trapping
landscape.

\begin{lemma}
\label{l:clas}
Assume that $\mu^\varepsilon$ converges as $\varepsilon\to0$
vaguely in distribution to a nontrivial random measure $\nu$. Then
one of the two following possibilities occurs:
\begin{longlist}[(ii)]
\item[(i)]
$\rho(\varepsilon)= \varepsilon^2 L(\varepsilon)$ for a
function $L$ slowly varying at $0$, and
$\nu= c \Leb_{\mathbb H}$, $c\in(0,\infty)$.

\item[(ii)] $\rho(\varepsilon) = \varepsilon^\alpha L(\varepsilon
)$ for
$\alpha>2$ and a function $L$ slowly varying at $0$, and $\nu$
can be written as
%
\begin{equation}
\nu = c_1 \mu_\FK^{2/\alpha} + \mu_\GFIN^{\mathbb F},
\end{equation}
where $c_1\in[0,\infty)$, $\mu_\FK^{2/\alpha}$ is the random
measure corresponding to the FK process defined in
Example~\ref{ex:FK}, and $\mu_\GFIN^{\mathbb F}$ is the random
measure of SSBM process given in Example~\ref{ex:GFIN},
$\mu_\FK^{2/\alpha}$ and $\mu_\GFIN^{\mathbb F}$ are mutually
independent. Moreover, the intensity measure $\mathbb F$ determining
the law of $\mu_\GFIN^{\mathbb F}$ satisfies the scaling relation~\eqref{e:Fscaling}.
\end{longlist}
In the both cases the limit measure $\nu$ is dense, infinite and
dispersed.
\end{lemma}

We first use this lemma to complete the proof of
Theorem~\ref{t:scalinglimitclassification}. By Lemmas~\ref
{l:converse} and
\ref{l:clas}, we can find a sequence $\varepsilon_k$ tending to $0$
as $k\to\infty$ such that $\mu^\varepsilon\to\nu$ vaguely in
distribution and $\nu$ is as in (i) or (ii) of Lemma~\ref{l:clas}, and
$\nu$ is dense, dispersed and infinite. Therefore, by
Theorem~\ref{t:delayedrwconv}, the family $X^\varepsilon$ of
processes converges in distribution on $D(\mathbb R_+, J_1)$ along the
subsequence $\varepsilon_k$ to a RTBM $X[\nu]$. As we assume that
the limit $\lim_{\varepsilon\to0}X^\varepsilon=U $ exists, we see
that $U=X[\mu]$.

The theorem then follows from the fact, that if (i) of
Lemma~\ref{l:clas} occurs, then $U=X[\nu]$ is a multiple of Brownian
motion, and thus (i) of the theorem occurs. On the other hand, if (ii) of
Lemma~\ref{l:clas} occurs, then $U=X[\nu]$ is a FK-SSBM mixture with
the claimed properties.
\end{pf*}

It remains to show Lemma~\ref{l:clas}.

\begin{pf*}{Proof of Lemma~\ref{l:clas}}
The proof that $\rho(\varepsilon)$ must be a regularly varying
function is standard: For $a >0$, $A\in\mathcal B(\mathbb H)$ bounded,
observe that
$\mathfrak S_{\varepsilon a }(\mu)(a A) = \mathfrak S_\varepsilon(\mu)(A)$.
Therefore,
%
\begin{eqnarray}
\label{e:scaling} \nu(A) &=&\lim_{\varepsilon\to0} \rho(\varepsilon)
\mathfrak S_\varepsilon(\mu) (A)
\nonumber
\\
&=&\lim_{\varepsilon\to0} \frac{\rho(\varepsilon)}{\rho(a\varepsilon)} \rho(a\varepsilon) \mathfrak
S_{a\varepsilon} (\mu) (aA)
\\
&=& \nu(a A) \lim_{\varepsilon\to0} \frac{\rho(\varepsilon)}{\rho(a\varepsilon)}.\nonumber
\end{eqnarray}
As both $\nu(A)$ and $\nu(aA)$ are nontrivial random variables, this
implies that the limit
$\lim_{\varepsilon\to0} \frac{\rho(\varepsilon)}{\rho
(a\varepsilon)}= c_k$
exists and is nontrivial. The theory of regularly varying functions
then yields
%
\begin{equation}
\label{e:rhoregvar} \rho(\varepsilon) = \varepsilon^{\alpha} L(\varepsilon)
\end{equation}
for $\alpha>0$ and a slowly varying function $L$. Inserting
\eqref{e:rhoregvar} into \eqref{e:scaling} also implies the scaling
invariance of $\nu$,
%
\begin{equation}
\label{e:nuscaling} a^\alpha \nu (A) \laweq\nu (a A), \qquad A\in\mathcal B(\mathbb
H), a>0.
\end{equation}

We now need to show that $\nu$ is as in (i) or (ii). To this end, we
use the theory of ``random measures with symmetries'' developed by
Kallenberg in \cite{Kal90,Kal05}. We recall from \cite{Kal05}, Chapter~9.1,
that random measure $\xi$ on $\mathbb H$ is said separately
exchangeable iff for any measure preserving transformations $f_1$ of
$\mathbb R$ and $f_2$ of $\mathbb R_+$
%
\begin{equation}
\xi\circ(f_1\otimes f_2)^{-1} \laweq\xi.
\end{equation}
Moreover, by \cite{Kal05}, Proposition~9.1, to check separate
exchangeability it is sufficient to restrict $f_1$, $f_2$ to
transpositions of dyadic intervals in $\mathbb R$ or $\mathbb R_+$,
respectively.

We claim that the limiting measure $\nu$ is separately exchangeable.
Indeed, restricting $\varepsilon$ to the sequence $\varepsilon_n= 2^{-n}$,
taking $I_1,I_2\subset\mathbb R$ and $J_1,J_2\subset\mathbb R_+$
disjoint dyadic intervals of the same length and defining $f_1$, $f_2$
to be transposition of $I_1, I_2$, respectively, $J_1,J_2$, it is easy
to see, using the i.i.d. property of the trapping landscape
$\bolds\pi$ and independence of $s^i_z$'s, that for all $n$
large enough.
%
\begin{equation}
\rho(\varepsilon_n)\mathfrak S_{\varepsilon_n }(\mu) \circ
(f_1\otimes f_2)^{-1} \laweq \rho(
\varepsilon_n)\mathfrak S_{\varepsilon_n }(\mu).
\end{equation}
Taking the limit $n\to\infty$ on both sides proves the separate
exchangeability of $\nu$.

The set of all separately exchangeable measures on $\mathbb H$ is
characterized by the following theorem which is a simple modification
of \cite{Kal05}, Theorem~9.23.

\begin{theorem}
\label{t:Kallenberg}
A random measure $\xi$ on $\mathbb H$ is separately exchangeable iff
almost surely
%
\begin{eqnarray}
\label{e:Kaldecomp} \xi &= &\gamma \Leb_{\mathbb H} + \sum
_k l(\alpha, \eta_k) \delta_{\rho_k,\rho'_k} +
\sum_{i,j} f\bigl(\alpha,\theta_i,
\theta'_j,\zeta_{ij}\bigr)
\delta_{\tau_i,\tau'_j }
\nonumber
\\
&&{}+\sum_{i,k} g(\alpha,\theta_i,
\chi_{ik}) \delta(\tau_i,\sigma_{ik}) +\sum
_i h(\alpha,\theta_i) (
\delta_{\tau_i}\otimes\Leb_+)
\\
&&{}+\sum_{j,k} g'\bigl(\alpha,
\theta'_j,\chi'_{jk}\bigr)
\delta\bigl(\sigma'_{jk},\tau'_j
\bigr) +\sum_j h'\bigl(\alpha,
\theta'_j\bigr) (\Leb\otimes\delta_{\tau'_j}),\nonumber
\end{eqnarray}
for some measurable functions $f\ge0$ on $\mathbb R_+^4$, $g,g'\ge0$
on $\mathbb R_+^3$, and $h,h',l\ge0$ on $\mathbb R_+^2$, an array of
i.i.d. uniform random variables $(\zeta_{i,j})_{i,j\in\mathbb N}$,
some independent unit rate Poisson processes $(\tau_j,\theta_j)_j$,
$(\sigma'_{ij}, \chi'_{ij})_j$, $i\in\mathbb N$, on $\mathbb H$,
$(\tau'_j, \theta'_j)_j$, $(\sigma_{ij},\chi_{ij})_j$,
$i\in\mathbb N$ on $\mathbb R_+^2$, and $(\rho_j,\rho'_j,\eta_j)_j$
on $\mathbb H\times\mathbb R_+$, and an independent pair of random
variables $\alpha,\gamma\ge0$.
\end{theorem}

\begin{pf}
Theorem~9.23 of \cite{Kal05} gives analogous characterization for the
separately exchangeable random measures on the quadrant
$\mathbb R_+\times\mathbb R_+$. To transfer this result to
$\mathbb H$, fix an arbitrary measure preserving bijection\break
$f\dvtx \mathbb R_+\to\mathbb R$, and note that the random measure
$\mu$ on the quadrant $\mathbb R_+\times\mathbb R_+$ is separately
exchangeable iff the image measure $\mu\circ(f\otimes\Id)^{-1}$ is
separately exchangeable on $\mathbb H$. The claim of the theorem then
follows by observing that the image of a unit rate Poisson process on
$\mathbb R_+$ under $f$ is a unit rate Poisson process on $\mathbb R$.
\end{pf}

Ignoring for the moment the issue of convergence of the sum in
\eqref{e:Kaldecomp}, let
us describe in words various terms appearing there to make a link
to our result. For this discussion, we ignore the random variable
$\alpha$
and omit it from the notation (later we will justify this step).

The term $\sum_k l(\eta_k) \delta_{\rho_k,\rho'_k}$ has the same law
as the random measure\break $\sum_k z_k \delta_{x_k,y_k}$ for a
Poisson point
process $(x_k,y_k,z_k)$ on $\mathbb H\times\mathbb R_+$ with intensity
$dx \,dy\,  \Pi_l(dz)$ where the measure $\Pi_l$ is given by
%
\begin{equation}
\label{e:Pil} \Pi_l(A) = \Leb_+\bigl(l^{-1}(A)\bigr), \qquad A
\in\mathcal B(\mathbb R_+).
\end{equation}
Recalling Example~\ref{ex:FK}, this term resembles to the random measure
driving the FK process, the $z$-component of the intensity measure being
more general here.

Similarly, the terms
$\sum_{i,k} g(\theta_i,\xi_{ik}) \delta(\tau_i,\sigma_{ik})
+\sum_i h(\theta_i)(\delta_{\tau_i}\otimes\Leb_+)$
can be interpreted as the random measure $\mu_\GFIN^{\mathbb F}$
defined in
Example~\ref{ex:GFIN}: $\tau_i$'s correspond to $x_i$'s, and
$f_i = f_{h(\theta_i), \Pi_{g(\theta_i,\cdot)}}$ [recall
\eqref{e:fdPi}, \eqref{e:Pil} for the notation]. The intensity
measure $\mathbb F$ used in the definition of SSBM is thus determined
by functions $h$ and~$g$.

The terms with $g'$, $h'$ can be interpreted analogously, with the role
of $x$- and $y$-axis interchanged. Term $\gamma\Leb_{\mathbb H}$ will
correspond to the Brownian motion component of $\nu$ (recall
Example~\ref{ex:speedmeasure}). Finally, the term containing $f$ can
be viewed as a family of atoms placed on the grid
$(\tau_i)_i\times(\tau'_j)_j$; we will not need it later.

We now explain why the limiting measure $\nu$ appearing in
Theorem~\ref{t:scalinglimitclassification} is less general than
\eqref{e:Kaldecomp}. The first reason comes from the fact that the
trapping landscape is i.i.d. This implies that $\nu$ is not only
exchangeable in the $x$-direction, but also that for every disjoint
sets $A_1$,
$A_2\subset\mathbb R$ the processes $\nu\langle A_1\rangle$,
$\nu\langle A_2\rangle$ are independent. As the consequence of this
property, we see that $\alpha$ and $\gamma$ must be a.s. constant (or
$f,h,h',g,g',l$ independent of $\alpha$). We can thus omit $\alpha$
from the notation.

Further, this independence implies that $h'=g'=f\equiv0$. Indeed,
assume that it is not the case. Then it is easy to see that, for
$A_1, A_2$ disjoint, the processes $\nu\langle A_1\rangle$,
$\nu\langle A_2\rangle$ have a nonzero probability to have a jump at
the same time. On the other hand, for every $\omega$ fixed,
$\nu\langle A_1\rangle(\omega)$ and $\nu\langle A_2\rangle(\omega)$
are independent L\'evy processes (they are limits of i.i.d. sums) and,
therefore,
for every $\omega$, $\tilde\Pb$-a.s., they do not jump at the same
time, contradicting the assumption.

The previous reasoning implies that
$\nu = \nu_1 + \nu_2 + \nu_3 + \nu_4$ where $\nu_1, \ldots, \nu
_4$ are
the Brownian, FK, FIN and ``pure SSBM'' component, respectively [by pure
SSBM we understand SSBM with $\mathbb F$ supported on Laplace exponents
with $\mathtt d=0$, see \eqref{e:fdPi}, cf. also Definition~\ref{d:FIN}]
%
\begin{eqnarray}
\nu_1 &=& \gamma \Leb_{\mathbb H}, \qquad\nu_3 = \sum
_i h(\theta_i) (\delta_{\tau_i}
\otimes\Leb_+),
\nonumber
\\[-8pt]
\\[-8pt]
\nonumber
\nu_2 &=& \sum_k l(\eta_k)
\delta_{\rho_k,\rho'_k},\qquad\nu_4 = \sum_{i,k}
g(\theta_i,\xi_{ik}) \delta(\tau_i,
\sigma_{ik}).
\end{eqnarray}

Observe that the functions $l$, $g$ and $h$ are not determined uniquely
by the law of~$\nu$. In particular, for any measure preserving
transformation $f$ of $\mathbb R_+$, $l$ and $l\circ f^{-1}$ give rise
to the
same law of $\nu$, and similarly for $h$ and $g(\theta,\cdot)$.
Hence, we
may assume that $l$, $h$ are nonincreasing, and $g$ is nonincreasing
in the second coordinate.

The final restriction on $\nu$ comes from its scaling invariance
\eqref{e:nuscaling} and the local finiteness. To complete the proof, we
should thus explore scaling properties of various components of $\nu$.

The Brownian component $\nu_1$ is trivial. It is scale-invariant with
$\alpha=2$.
To find the conditions under which the FK component $\nu_2$ is
scale-invariant, we set
$A=[0,x]\times[0,y]$ and compute the Laplace transform of $\nu_2 A$.
To this end, we use the formula
%
\begin{equation}
\mathbb E\bigl[e^{-\pi f}\bigr] = \exp \bigl\{-\lambda
\bigl(1-e^{-f}\bigr) \bigr\},
\end{equation}
which holds for any Poisson point process $\pi$ on a measurable space
$E$ with
intensity measure $\lambda\in M(E)$ and $f\dvtx E\to\mathbb R$ measurable.
Using this formula with $\pi= (\rho_i,\rho'_i, \eta_i)$ and
$f(\rho,\rho',\eta)= \1_{A}(\rho,\rho') \lambda l (\eta)$, we
obtain that
%
\begin{equation}
\label{e:PPPLaplace} \mathbb E\bigl[e^{-\lambda \nu_2 A}\bigr]=\exp \biggl\{-xy\int
_0^\infty \bigl(1-e^{-\lambda l(\eta)}\bigr)\,d\eta
\biggr\}.
\end{equation}
The scaling invariance \eqref{e:nuscaling} then yields
%
\begin{equation}
a^2 \int_0^\infty
\bigl(1-e^{-\lambda l(\eta)}\bigr)\,d\eta = \int_0^\infty
\bigl(1-e^{-\lambda a^\alpha l(\eta)}\bigr)\,d\eta\qquad \forall\lambda,a >0,
\end{equation}
implying (together with the fact that $l$ is nonincreasing) that
$l(\eta)=c' \eta^{-\alpha/2}$, for a $c'\ge0$, $\alpha>0$. By
\cite{Kal05}, Theorem~9.25, $\nu_2$ is locally finite iff
$\int_0^\infty(1\wedge l(\eta))\,d\eta<\infty$, yielding $\alpha>2$.
Finally, using the observation from the discussion around
\eqref{e:Pil}, we see that $\nu_2= c \mu_\FK^{2/\alpha}$.

The component $\nu_3$ can be treated analogously. Using formula
\eqref{e:PPPLaplace} with $\pi= (\tau_i,\theta_i)$ and
$f=\lambda y h(\theta) \1_{[0,x]}(\tau)$, we obtain
%
\begin{equation}
\mathbb E\bigl[e^{-\lambda\nu_3 A }\bigr] = \exp \biggl\{-x\int_0^\infty
\bigl(1-e^{-\lambda y h(v)}\bigr) \,dv \biggr\}.
\end{equation}
The scaling invariance and the fact that $h$ is nonincreasing then
yields $h(\theta) = c \theta^{1-\alpha}$, for $c\ge0$, $\alpha\ge1$.
Using \cite{Kal05}, Theorem~9.25, again, $\nu_3$ is locally finite iff
$\int_0^\infty(1\wedge h(\theta))\,d\theta<\infty$, implying
$\alpha>2$.

The component $\nu_4$ is slightly more difficult as we need to deal
with many Poisson point processes. Using formula \eqref{e:PPPLaplace}
for the processes $(\sigma_{ij})_j$ and $(\chi_{ij})_j$ we get
%
\begin{equation}\qquad
\mathbb E\bigl[e^{-\lambda\nu_4 A}|(\theta_i),(\tau_i)
\bigr] = \exp \biggl\{- \sum_{i}
\1_{[0,x]}(x_i) y \int_0^\infty
\bigl(1-e^{-\lambda g(\theta_i,\chi) }\bigr)\,d\chi \biggr\}.
\end{equation}
Applying \eqref{e:PPPLaplace} again, this time for processes $(\tau_i)$,
$(\theta_i)$, then yields
%
\begin{equation}
\mathbb E\bigl[e^{-\lambda\nu_4 A}\bigr] = \exp \biggl\{-x \int
_0^\infty \bigl(1- e^{
- y
\int_0^\infty(1-e^{-\lambda g(\theta,\chi) })\,d\chi} \bigr)\,d\theta
\biggr\}.
\end{equation}
Hence, by scaling invariance and trivial substitutions, $g$ should
satisfy
%
\begin{eqnarray}
\label{e:scalingnufour} &&\int_0^\infty \bigl(1-
e^{ - y \int_0^\infty
(1-e^{-\lambda g(\theta,\chi) })
\,d\chi} \bigr) \,d\theta
\nonumber
\\[-8pt]
\\[-8pt]
\nonumber
&&\qquad= \int_0^\infty \bigl(1- e^{ - y \int_0^\infty
(1-e^{-\lambda a^{-\alpha} g(\theta/a,\chi/a) })
\,d\chi} \bigr)
\,d\theta
\end{eqnarray}
for every $a,y,\lambda>0$.

By \cite{Kal05}, Theorem~9.25, once more, $\nu_4$ is locally finite iff
%
\begin{equation}
\label{e:gfinite} \int \biggl\{1\wedge\int\bigl(1\wedge g(\theta,\chi)\bigr)\,d \chi
\biggr\} \,d\theta<\infty.
\end{equation}
We use this condition to show that for $\nu_4$ the scaling exponent
must satisfy $\alpha>2$. As $\alpha> 1$ is obvious, we should only
exclude $\alpha\in(1,2]$. By \eqref{e:scalingnufour} and the fact that
Laplace transform determines measures on $\mathbb R_+$,
%
\begin{eqnarray}
&&\Leb_+ \biggl\{\theta\dvtx \int\bigl(1-e^{-g(\theta, \chi)}\bigr)\,d\chi\ge u \biggr
\}
\nonumber
\\[-8pt]
\\[-8pt]
\nonumber
&&\qquad= \Leb_+ \biggl\{\theta\dvtx \int\bigl(1-e^{-a^{-\alpha}
g(\theta/a, \chi/a )}\bigr)\,d\chi\ge u \biggr
\}.
\end{eqnarray}
For some $c>1$, $c^{-1}(1\wedge x)\le1-e^{-x}\le1\wedge x$, therefore,
for $u\in(0,1)$
%
\begin{eqnarray}
\label{e:Ku} K(u) &:=& \Leb_+\biggl\{\theta\dvtx \int\bigl(1\wedge g(\theta,
\chi)\bigr)\,d \chi\ge u\biggr\}
\nonumber
\\
&\ge& \Leb_+\biggl\{\theta\dvtx \int\bigl(1-e^{-g(\theta, \chi)}\bigr)\,d\chi\ge u\biggr
\}
\nonumber
\\
&=&\Leb_+\biggl\{\theta\dvtx \int\bigl(1-e^{-a^{-\alpha}
g(\theta/a, \chi/a )}\bigr)\,d\chi\ge u\biggr\}
\nonumber
\\[-8pt]
\\[-8pt]
\nonumber
&\ge& a \Leb_+\biggl\{\theta\dvtx \int\bigl(a^\alpha\wedge g(\theta,\chi)
\bigr)\,d \chi\ge c a^{\alpha-1}u\biggr\}
\\
&\ge& u^{-1/(\alpha-1)} \Leb_+\biggl\{\theta\dvtx \int\bigl(1\wedge g(\theta,\chi)
\bigr)\,d \chi\ge c\biggr\}
\nonumber
\\
&=& u^{-1/(\alpha-1)} K(c),\nonumber
\end{eqnarray}
where for the last inequality we set $a\ge1$ so that $a^{\alpha-1 }u =1$.
Using \eqref{e:Ku}, it can be checked easily that the integral over
$\theta$ in \eqref{e:gfinite} is not finite when $\alpha\in(1, 2]$,
implying $\alpha>2$.

To complete the proof of Theorem~\ref{t:scalinglimitclassification}, it
remains to show the scaling relation~\eqref{e:Fscaling}. This is easy
to be done using the correspondence of $\nu_3+\nu_4$ and
$\mu_\GFIN^{\mathbb F}$.
Indeed, let $\mu_\GFIN^{\mathbb F}$, $\mu_{(x_i,f_i)}$ be as in
Example~\ref{ex:GFIN}. By scaling considerations,
%
\begin{equation}
a^{-\alpha}\mathfrak S_{a^{-1}}\mu_{(x_i,f_i)} \laweq
\mu_{(x_i/a,\sigma_a^\alpha f_i)},
\end{equation}
from which \eqref{e:Fscaling} follows immediately.

The fact that $\nu$ is dispersed follows from Lemma~\ref{l:borde}, as
in the both cases, (i) and~(ii), $\nu$ is a trap measure of RTBM.
Density of $\nu$ can be easily deduced from its scaling invariance and
infiniteness of $\nu$ is obvious.
\end{pf*}

\subsection{Convergence to the Brownian motion}
\label{ss:BMconvergence}

Here, we present the proof of the convergence to Brownian motion stated in
Theorem~\ref{t:BMconvergence}. For reading the proof, it is useful to
recall the notation introduced when defining RTRW in Section~\ref{ss:rtrw}.

\begin{pf*}{Proof of Theorem~\ref{t:BMconvergence}}
Let $\mu$ be the random trap measure of the RTRW $X$ under consideration.
We recall that $s_z^i$ stands for the duration of the $i$th visit of
$X=Z[\mu]$ to $z\in\Z$.

We use the multidimensional individual ergodic theorem, which we
recall for the sake of completeness in the \hyperref[app]{Appendix},
Theorem~\ref{t:ergodic}.
We apply it for $X=\R_+^{\Z\times\Z}$, $Q$ the distribution of
$(s_z^i)_{z,i\in\mathbb Z}$ under $\Pb\otimes\tilde\Pb$, and
$\mathcal G$ the cylinder field (here we extend $s_z^i$ to negative $i$'s
in the natural way). We define
$(\theta_{i,j})_{i,j\in\Z}\dvtx \R_+^{\Z\times\Z}\to\R_+^{\Z\times
\Z}$ via
$\theta_{x,j}((s_z^i)_{z,i\in\Z})=(s_{x+z}^{i+j})_{z,i\in\Z}$. It is
clear from the construction that $Q$ is stationary under $\theta_{x,j}$.
As the trapping landscape and $(s_z^i)_{i}$, $z\in\mathbb Z$, are
i.i.d., $Q$ is ergodic with respect to every $\theta_{x,j}$ with
$x\neq0$. Hence, the invariant field is trivial. The multidimensional
ergodic theorem then implies that for any two intervals
$I,J\subset\mathbb R$
%
\begin{equation}
\qquad \frac{1}{n^2}\sum_{z:{z}/n\in I}\sum
_{i:{i}/n\in J} s_z^i\xrightarrow{n\to\infty}
|I||J|(\mathbb E\otimes\tilde{\mathbb E})\bigl[s^i_z
\bigr] = |I||J|M, \qquad Q\mbox{-a.s.}
\end{equation}
Therefore,
$\varepsilon^2 \mathfrak{S}_\varepsilon(\mu)(I\times J)\to
|I||J|M$, and
thus $\varepsilon^2 \mathfrak S_\varepsilon(\mu)$ converges to
$M\times\Leb_{\mathbb H}$,
$\Pb\times\tilde\Pb$-a.s. This together with
Theorem~\ref{t:delayedrwconv} completes the proof.
\end{pf*}

\subsection{Convergence to the FK process} 
\label{ss:FKconvergence}
Here, we present the proof of Theorem~\ref{t:FKconv}. As usual, $\mu$
will stand for the random trap measure of the RTRW under consideration.

\begin{pf*}{Proof of Theorem~\ref{t:FKconv}}
To show the convergence in $P^{\bolds\pi}$-distribution, in
$\Pb$-probability, we will show the equivalent statement; see
\cite{Kal02}, Lemma~4.2.
%
\begin{equation}
\label{e:equivstat}
\quad\begin{tabular}{p{300pt}@{}}
For every sequence $\varepsilon_n$
there exists a subsequence $\varepsilon_{n_k}$ such that as $k\to
\infty$, $(\varepsilon X_{q_{\FK}(\varepsilon_{n_k})^{-1}t})_{t\ge0}$ converges to the FK
process with parameter $\gamma=2/\alpha$, in $P^{\pi}$-distribution, $
\Pb$-a.s.
\end{tabular}
\end{equation}

We thus fix a sequence $\varepsilon_n\to0$
and check \eqref{e:equivstat} for a subsequence
$\varepsilon_{n_k}=:\tilde\varepsilon_k$ satisfying
%
\begin{equation}
\label{e:FKsummability} \sum_{k=1}^\infty \tilde
\varepsilon_k^{-3} \mathbb E \bigl[ \bigl(1-\hat{\pi}
\bigl(q_{\FK}(\tilde\varepsilon_k)\bigr)
\bigr)^{2} \bigr]< \infty.
\end{equation}
By
Theorem~\ref{t:delayedrwconv}, it is sufficient to show that
$\mu_{\tilde\varepsilon_k}:= q_\FK(\tilde\varepsilon_k
)\mathfrak S_{\tilde\varepsilon_k} (\mu)$
converges vaguely in distribution to $\mu_\FK^\gamma$, $\Pb$-a.s.,
where $\mu_\FK^\gamma$ is the driving measure of the FK process
introduced in Example~\ref{ex:FK}, and $\mu$ is the trap measure of the
RTRW~$X$. For every given $\omega\in\Omega$,
$\mu=\mu(\omega,\tilde\omega)$ is the trap measure of a TRW. We
also know that $\mu_\FK^\gamma$ is L\'evy and has independent
increments. Therefore, we can apply
Proposition~\ref{p:kalconv}, and only check that for every
rectangle $A=[x_1,x_2]\times[y_1,y_2]$ with rational coordinates, $\Pb$-a.s.,
$\mu_\epsk(A)\xrightarrow{k\to\infty}\mu_\FK^\gamma(A)$ (it is
easy to
see that such rectangles form a DC semiring and are in
$\mathcal T_{\mu^\gamma_\FK}$). $\mu_\FK^\gamma(A)$ has a $\gamma$-stable
distribution with scaling parameter proportional to
$\Leb_{\mathbb H}(A)$, and thus its Laplace exponent is
$(x_2-x_1)(y_2-y_1)\lambda^\gamma$. The Laplace transform of
$\mu_\varepsilon(A)$ given $\omega$ [and thus given the trapping
landscape $(\pi_z)_{z\in\mathbb Z}$] is easy to compute. By the
independence of $s_z^i$'s,
%
\begin{equation}
\tilde{\mathbb E} \bigl[e^{-\lambda\mu_\varepsilon(A)}\bigr] = \prod
_{z=x_1\varepsilon^{-1}}^{x_2\varepsilon^{-1}} \hat\pi_z\bigl(\lambda
q_\FK(\varepsilon)\bigr)^{\varepsilon^{-1}(y_2-y_1)}.
\end{equation}
Hence, taking the $-\log$ to obtain the Laplace exponent, we shall show
that $\Pb$-a.s., for every $x_1<x_2$, $y_1,y_2\in\mathbb Q$,
$0\le\lambda\in\mathbb Q$,
%
\begin{equation}\qquad
\label{e:FKLLN} \tilde\varepsilon_k^{-1}
(y_2-y_1) \sum_{z=x_1\tilde\varepsilon_k^{-1}}^{x_2\tilde\varepsilon_k^{-1}}
\bigl(-\log\hat\pi_z \bigl(\lambda q_\FK(\tilde
\varepsilon_k)\bigr) \bigr) \xrightarrow{k\to\infty}
(y_2-y_1) (x_2-x_1)
\lambda^\gamma.
\end{equation}
As $\mathbb Q$ is countable, it is sufficient to show this for fixed $x$'s,
$y$'s and $\lambda$. This will follow by a standard
law-of-large-numbers argument as $\pi_z$'s are i.i.d. under $\mathbb P$.
To simplify the notation, we set $x_1=0$, $x_2=1$; $y$'s can be omitted
trivially.

We first consider $\lambda\le1$ and truncate. Using the monotonicity of
$\hat\pi$, $\lambda\le1$, and the Chebyshev inequality
%
\begin{equation}
\label{e:FKsupremum} \mathbb P\Bigl[\sup_{0\le z \le\epsk^{-1}} \bigl(1- \hat
\pi_z \bigl(q_\FK(\lambda\epsk)\bigr)\bigr)\ge\epsk\Bigr]
\le \epsk^{-3} \mathbb E\bigl[\bigl(1- \hat\pi\bigl(q_\FK(
\epsk)\bigr)\bigr)^2\bigr].
\end{equation}
Equation~\eqref{e:FKsummability} then implies that the above supremum is smaller
than $\epsk$ for all $k$ large enough, $\Pb$-a.s. Hence, for all $k$
large,
%
\begin{eqnarray}
\label{e:FKLLNtrunc} &&\tilde\varepsilon_k^{-1} \sum
_{z=0}^{\epsk^{-1}} \bigl(-\log\hat\pi_z \bigl(
\lambda q_\FK(\tilde\varepsilon_k)\bigr) \bigr)
\nonumber
\\[-8pt]
\\[-8pt]
\nonumber
&&\qquad=\epsk^{-1} \sum_{z=0}^{\epsk^{-1}}
\bigl(-\log \bigl( (1-\epsk)\vee\hat\pi_z \bigl(\lambda
q_\FK(\tilde \varepsilon_k)\bigr) \bigr) \bigr).
\end{eqnarray}

For any $\delta>0$, there is $\varepsilon$ small so that
%
\begin{equation}\quad
\label{eq:boundsforlog} (1-x)\le-\log x \le(1-x)+\bigl(\tfrac{1}2 +\delta\bigr)
(1-x)^2,\qquad x\in(1-\varepsilon,1].
\end{equation}
The expectation of the right-hand side of \eqref{e:FKLLNtrunc} is
bounded from above by
%
\begin{eqnarray}\qquad
\label{e:FKexp}&& \epsk^{-2}\mathbb E\bigl[\epsk\wedge\bigl(1-\hat
\pi_z \bigl(\lambda q_\FK (\tilde \varepsilon_k)
\bigr)\bigr)\bigr] + c \epsk^{-2}\mathbb E \bigl[ \bigl(\epsk\wedge
\bigl(1-\hat\pi_z \bigl(\lambda q_\FK(\tilde
\varepsilon_k)\bigr)\bigr) \bigr)^2 \bigr]
\nonumber
\\[-8pt]
\\[-8pt]
\nonumber
&&\qquad\le\epsk^{-2}\mathbb E\bigl[1-\hat\pi_z \bigl(\lambda
q_\FK (\tilde \varepsilon_k)\bigr)\bigr] + o(1),
\end{eqnarray}
as $k\to\infty$, by \eqref{e:FKsecondmoment}.
On the other hand, using \eqref{eq:boundsforlog},
%
\begin{equation}
\E \Biggl(\tilde\varepsilon_k^{-1} \sum
_{z=0}^{\epsk^{-1}} \bigl(-\log\hat\pi_z \bigl(
\lambda q_\FK(\tilde\varepsilon_k)\bigr) \bigr) \Biggr)
\geq\epsk^{-2}\mathbb E\bigl[1-\hat\pi_z \bigl(\lambda
q_\FK(\tilde \varepsilon_k)\bigr)\bigr].
\end{equation}

Moreover,
%
\begin{equation}
\epsk^{-2}\mathbb E\bigl[1-\hat\pi_z \bigl(\lambda
q_\FK(\tilde \varepsilon_k)\bigr)\bigr] =
\frac{\Gamma(\lambda q_{\FK}(\epsk))}{
\Gamma(q_{\FK}(\epsk))} \xrightarrow{k\to\infty} \lambda^{\gamma},
\end{equation}
by the fact that $\Gamma$ is regularly varying. Therefore, the
expectation of \eqref{e:FKLLNtrunc} converges to
$\lambda^\gamma$.

To compute the variance of the right-hand side of \eqref{e:FKLLNtrunc},
we observe that the second moment of one term is, for $k$ large,
bounded by
%
\begin{equation}\qquad
2 \mathbb E \bigl[ \bigl(\epsk\wedge\bigl(1-\hat\pi_z \bigl(\lambda
q_\FK (\tilde \varepsilon_k)\bigr)\bigr)
\bigr)^2 \bigr] \le 2 \mathbb E \bigl[ \bigl(1-\hat\pi_z
\bigl(q_\FK(\tilde \varepsilon_k)\bigr)
\bigr)^2 \bigr] = o\bigl(\epsk^3\bigr),
\end{equation}
as $k\to\infty$, by \eqref{e:FKsecondmoment}. Since the first moment
of one term is
$O(\epsk^2)$, by the previous computation, we see that
the variance of the right-hand side of \eqref{e:FKLLNtrunc} is bounded
by
%
\begin{equation}
C \epsk^{-3} \mathbb E \bigl[ \bigl(1-\hat\pi_z
\bigl(q_\FK(\tilde \varepsilon_k)\bigr)
\bigr)^2 \bigr],
\end{equation}
which is summable over $k$, by \eqref{e:FKsummability}. This implies
the strong law of large numbers for \eqref{e:FKLLNtrunc}, and thus
\eqref{e:FKLLN} for $\lambda\le1$. For $\lambda\ge1$,
\eqref{e:FKLLN} follows from the analyticity of Laplace transform. This
proves \eqref{e:FKLLN}, and thus the first claim of the theorem.

To prove the second claim of the theorem, it is sufficient to repeat
the previous argument with $\epsk= k^{-1+\delta/2}$. From the
assumption of the theorem then follows that
$\varepsilon^{-4-\delta} \mathbb E  [ (1-\hat{\pi}(q_{\FK
}(\varepsilon)) )^{2} ]=o(1)$,
and thus
%
\begin{equation}
\tilde\varepsilon_k^{-3} \mathbb E \bigl[ \bigl(1-\hat{
\pi}\bigl(q_{\FK}(\tilde\varepsilon_k)\bigr)
\bigr)^{2} \bigr] =o\bigl(\epsk^{1+\delta}\bigr) = o
\bigl(k^{(1+\delta) (1-\delta/2)} \bigr),
\end{equation}
and hence \eqref{e:FKsummability} holds. Therefore, $\mathbb P$-a.s. holds
along $\epsk$. To pass from the convergence along $\epsk$ to the
convergence as $\varepsilon\to0$, it is
sufficient to observe that, since
$\tilde\varepsilon_{k+1}^{-1} -\epsk^{-1}\xrightarrow{k\to\infty}
0$, for any
rectangle $A$ and $\varepsilon$ small enough there is $k$ such that
$\mathfrak S_\varepsilon(\mu)(A) = \mathfrak S_\epsk(\mu)(A)$.
\end{pf*}

\subsection{Convergence to the SSBM process}
\label{ss:GFINconvergence}
Next, we prove Theorem~\ref{t:RSBMconv}. Again, $\mu$ stands for
the random trap measure of the RTRW $X$ under consideration.

\begin{pf*}{Proof of Theorem~\ref{t:RSBMconv}}
The proof is based on the following lemma.

\begin{lemma}
\label{l:GFINcoupling}
There exists a probability space
$(\bar\Omega,\bar{\mathcal F},\bar\Pb)$ and a family of trap measures
$(\bar{\mu}^\varepsilon_{\bar\omega})_{\varepsilon\ge0,\bar
\omega\in\bar\Omega}$
on another probability space
$(\tilde\Omega,\tilde{\mathbb F},\tilde\Pb)$ indexed by
$\bar\omega\in\bar\Omega$, such that, when
$\bar\mu^\varepsilon$, $\varepsilon\ge0$, denotes the mixture of
$\bar\mu^\varepsilon_{\bar\omega}$ w.r.t. $\bar\Pb$, the following
conditions hold:
\begin{longlist}[(a)]
\item[(a)] For every $\bar\omega\in\bar\Omega$ and $\varepsilon>0$,
$\bar\mu_{\bar\omega}^\varepsilon$ is a trap measure of a TRW,
and $\bar\mu_{\bar\omega}^0$ is a trap measure of a TBM.

\item[(b)] For every $\varepsilon>0$, $\bar{\mu}^{\varepsilon}$ is
distributed as $\mu$.

\item[(c)] $\bar\mu^0$ is distributed as $\mu_\GFIN^{\mathbb F}$.

\item[(d)]
$q(\varepsilon)\mathfrak{S}_{\varepsilon}(\bar{\mu}^\varepsilon
_{\bar\omega})$
converges vaguely in $\tilde\Pb$-distribution to $\bar{\mu}^0_{\bar
\omega}$
as $\varepsilon\to0$, for $\bar\Pb$-a.e. $\bar\omega$.
\end{longlist}
\end{lemma}

We first complete the proof of Theorem~\ref{t:RSBMconv} using the
previous lemma. As, by~(a), $\bar\mu^0_{\bar\omega}$ is a L\'evy
trap measure for every $\bar\omega$, it is dispersed trap measure for
every $\bar\omega$, by Lemma~\ref{l:borde}. By Assumption (L), $\mu$
and $\mu_\GFIN^{\mathbb F}$ are $\mathbb P\otimes\tilde\Pb$-infinite.
Hence, due to (a)--(c) of the last lemma,
$(\bar{\mu}^\varepsilon_{\bar\omega})_{\varepsilon\ge0}$ are
infinite measures, $\bar\Pb$-a.s. From the scaling
relation \eqref{e:Fv}, one further deduces that $\mathbb F$ is not a
finite measure, so $\bar\mu^0$ is $\bar\Pb$-a.s. dense. Thus, we can
apply Theorem~\ref{t:delayedrwconv} and deduce from (d) the $\bar\Pb
$-a.s. convergence
in $\tilde\Pb$-distribution of
$(\varepsilon Z[\bar\mu_{\bar\omega}^\varepsilon]_{q(\varepsilon
)^{-1}t})_{t\geq0}$
to $(B[\bar\mu^0_{\bar\omega}]_t)_{t\geq0}$. By (b), (c) of the last
lemma, for every $\varepsilon>0$, $Z[\mu]$ is distributed as
$Z[\bar\mu^{\varepsilon}]$, and $B[\mu_\GFIN^{\mathbb F} ]$ is
distributed as $B[\bar\mu^0]$, this implies the claim of the theorem.
\end{pf*}

\begin{pf*}{Proof of Lemma~\ref{l:GFINcoupling}}
The proof of Lemma~\ref{l:GFINcoupling} is split to two parts. In the
first, we construct the coupling that satisfies (a)--(c) of the
lemma. In the second part, we prove that this coupling satisfies the
convergence claim (d).

\emph{Construction of the coupling.}
We consider a probability space $(\Omega_1,\mathcal F_1,
\mathbb P_1)$
on which we construct a Poisson point process
$(x_i,v_i)_{i\in\mathbb N}$ on $\mathbb R\times(0,\infty)$
with intensity $\gamma v^{-\gamma-1} \,dx \,dv$. For
$\omega\in\Omega_1$, we define
$\rho(\omega)=\sum_{i>0}v_i \delta_{x_i}$, and
$V(\omega)\in D(\mathbb R)$ by $V_0(\omega)=0$ and
$V_b(\omega)-V_a(\omega)= \rho((a,b]) (\omega)$, $a<b$, so that
$V$ is a two-sided $\gamma$-stable subordinator.

On the same probability space, we construct for every $\varepsilon>0$
a families of nonnegative random variables
$(m_z^\varepsilon)_{z\in\mathbb Z}$, such that
$(m_z^\varepsilon)_{z\in\mathbb Z}$
has the same distribution as $(m(\pi_z))_{z\in\mathbb Z}$. Similarly as
in \eqref{e:VV}, we define
$V^\varepsilon\in D(\mathbb R)$ by
%
\begin{equation}
V^\varepsilon_x= \cases{ \displaystyle\sum_{i=1}^{\lfloor x \rfloor}
m^\varepsilon_i, &\quad $x\ge 1$,\vspace*{2pt}
\cr
0,&\quad $x\in[0,1)$,
\vspace*{2pt}
\cr
\displaystyle\sum_{\lfloor x \rfloor+1}^0
m^\varepsilon_i, &\quad $x<0$.}
\end{equation}
By Assumption (HT), using Remark~\ref{rk:assumptionpp},
$d(\varepsilon)^{-1} V^\varepsilon_{\varepsilon^{-1}x}$ converges in
distribution on $(D(\mathbb R),J_1)$ to $V$. By the Skorokhod representation
theorem, we may choose $(m^\varepsilon_z)_{z,\varepsilon}$ so that this
convergence holds $\mathbb P_1$-a.s., and we do so.

For $\omega\in\Omega_1$ for which
$\varepsilon^{1/\gamma} V^\varepsilon_{\varepsilon^{-1}\cdot
}(\omega)\to V(\omega)$,
we fix an injective mapping
$I_\omega^\varepsilon(z)\dvtx \mathbb Z\to\mathbb N$ which satisfies
%
\begin{equation}
\label{e:Icond} \varepsilon z_i^\varepsilon\xrightarrow{
\varepsilon\to0} x_i, \qquad d(\varepsilon)^{-1}
m^\varepsilon_{z^\varepsilon_i} \xrightarrow{\varepsilon\to0} v_i\qquad
\mbox{for every $i\in\mathbb N$,}
\end{equation}
with $z_i^\varepsilon:= (I_\omega^\varepsilon)^{-1}(i)$,
$i\in\mathbb N$, $\varepsilon>0$. This is possible by the matching of
jumps property of the $J_1$-topology (see, e.g., \cite{Whi02}, Section~3.3). Remark that, as $I_\omega^\varepsilon$ is not
necessarily surjective, $z_i^\varepsilon$ is not defined for all $i$
and $\varepsilon$. On the other hand, \eqref{e:Icond} implicitly
requires that, for every $i\in\mathbb N$, $z_i^\varepsilon$ is
defined for all $\varepsilon$ small enough.

To proceed with the construction, we need a simple lemma.

\begin{lemma}
\label{l:onetrap}
Let $(v_\varepsilon)_{\varepsilon>0}$ be such that
$v_\varepsilon\to v$ as $\varepsilon\to0$.
Then
%
\begin{equation}
\Psi_\varepsilon\bigl(\pi^{d(\varepsilon)v_\varepsilon}\bigr) \xrightarrow{\varepsilon
\to0}\mathbb F_v.
\end{equation}
\end{lemma}

\begin{pf}
Let $t(\varepsilon)$ be defined by
$d(t(\varepsilon))=d(\varepsilon)v_\varepsilon$ or equivalently
$t(\varepsilon):=d^{-1}(d(\varepsilon)v_\varepsilon)$
(recall that $d$ is strictly decreasing and continuous). Then,
using the function $\sigma_a^\alpha$ introduced in \eqref{e:sigma},
%
\begin{eqnarray}
\Psi_{\varepsilon}\bigl(\pi^{d(\varepsilon)v_\varepsilon}\bigr) &=&\varepsilon^{-1}
\biggl(1-\int_{\mathbb{R}_+} e^{-q(\varepsilon)\lambda u} \pi^{d(\varepsilon)v_\varepsilon}(du)
\biggr)
\nonumber
\\
&=&\varepsilon^{-1} \biggl(1-\int_{\mathbb{R}_+}
e^{-\varepsilon d(\varepsilon)^{-1}\lambda u} \pi^{d(t(\varepsilon))}(du) \biggr)
\nonumber
\\
& =& \varepsilon^{-1} \biggl(1-\int_{\mathbb{R}_+}
e^{-\varepsilon v_\varepsilon t(\varepsilon)^{-1}
q(t(\varepsilon))\lambda u} \pi^{d(t(\varepsilon))}(du) \biggr)
\\
&= &\frac{t(\varepsilon)}{\varepsilon} \biggl(\frac{\varepsilon v_\varepsilon}{t(\varepsilon)} \biggr) ^{\gamma/{(1+\gamma)} }
\sigma^{1+1/\gamma}_{ ({\varepsilon v_{\varepsilon}}/
{t(\varepsilon)} )^{-{\gamma}/{(\gamma+1)}}} \bigl(\Psi_{t(\varepsilon)}\bigl(
\pi^{d(t(\varepsilon))}\bigr)\bigr)
\nonumber
\\
&= &\biggl(\frac{t(\varepsilon)}{\varepsilon} \biggr)^{{1}/{(\gamma+1)}} v_\varepsilon^{{\gamma}/{(\gamma+1)}}
\sigma^{1+1/\gamma}_{v^\gamma ({\varepsilon v_\varepsilon}/
{t(\varepsilon)} )^{-{\gamma}/{(\gamma+1)}}} \bigl(\sigma^{1+1/\gamma}_{v^{-\gamma}}
\bigl(\Psi_{t_\varepsilon}\bigl(\pi^{d(t(\varepsilon))}\bigr)\bigr) \bigr).\nonumber
\end{eqnarray}
As $d(\varepsilon)$, and thus $d^{-1}(\varepsilon)$ are regularly varying,
%
\begin{equation}
\frac{t(\varepsilon)}{\varepsilon} =\frac{d^{-1}(v_\varepsilon d(\varepsilon))}{d^{-1}(d(\varepsilon))} \xrightarrow{\varepsilon\to0}
v^{-\gamma}.
\end{equation}
Hence,
%
\begin{equation}
\biggl(\frac{t(\varepsilon)}{\varepsilon} \biggr)^{{1}/{(\gamma+1)}} v_\varepsilon^{{\gamma}/{(\gamma+1)}}
\xrightarrow{\varepsilon \to0}1,
\end{equation}
and similarly
%
\begin{equation}
v_\varepsilon^{\gamma} \biggl(\frac{\varepsilon v_\varepsilon}{t(\varepsilon)} \biggr)
^{-{\gamma}/{(\gamma+1)}} \xrightarrow{\varepsilon\to0}1,
\end{equation}
and thus
$\sigma^{1+1/\gamma}_{v_\varepsilon^{\gamma} (
{\varepsilon v_\varepsilon}/{t(\varepsilon)} )^{-{\gamma
}/{(\gamma+1)}}}$
converges to the identity. Assumption (L) together with
$t(\varepsilon) \to0$ and \eqref{e:Fv} then implies the lemma.
\end{pf}

The space $C(\mathbb{R}_+)$, and thus
$\mathfrak F^\ast\subset C(\mathbb R_+)$, endowed with the topology of
uniform convergence over compact sets is separable. It is a known fact
that in the space $\mathfrak F^\ast$ the pointwise convergence and the
uniform convergence over compact sets coincide. (Recall
$\mathfrak F^\ast$ is the space of Laplace exponents. When the Laplace
exponents converge pointwise to an element of $\mathfrak F^\ast$, the
corresponding probability measures converge weakly, which in turns
gives the uniform convergence over compacts.) We deduce that
$\mathfrak F^\ast$ with the topology of pointwise convergence is also
separable.

We further consider a measurable space $(\Omega_2, \mathcal F_2)$
and construct a probability kernel $\mathbb P_2^\cdot$ from $\Omega
_1$ to
$\Omega_2$, and $\mathfrak F^\ast$-valued random variables
$(\psi_z^\varepsilon)_{z\in\mathbb Z, \varepsilon>0}$,
$(f_i)_{i\in\mathbb N}$ on $\Omega_2$
such that under $\mathbb P_2^\omega$ the random variables
$(\psi_z^\varepsilon)_{z\in\mathbb Z} $ are
independent for every $\varepsilon>0$, $\psi_z^\varepsilon$
has the same distribution as
$\Psi_\varepsilon(\pi_z^{m^\varepsilon_z(\omega)})$, and $f_i$,
$i\in\mathbb N$, are i.i.d. with marginal $\mathbb F_1$. As
$v_i^\varepsilon:=d(\varepsilon)^{-1 } m_{z_i^\varepsilon
}^\varepsilon \to v_i$,
by Lemma~\ref{l:onetrap},
%
\begin{equation}
\label{e:psisconv} \psi^\varepsilon_{z_i^\varepsilon} \xrightarrow{\varepsilon
\to0} \sigma^{1+1/\gamma}_{v_i^{-\gamma}} f_i\qquad \mbox{for all $i\in
\mathbb N$,}
\end{equation}
in distribution on $\mathfrak F^\ast$. Using the separability of
$\mathfrak F^\ast$ and thus of $(\mathfrak F^\ast)^{\mathbb Z}$, by
Skorokhod representation theorem, we may require that
$\psi^\varepsilon_z$'s are such that this convergence holds
$\mathbb P_2^\omega$-a.s.

We take $\bar\Omega= \Omega_1\times\Omega_2$,
$\bar{\mathcal F}=\mathcal F_1\otimes\mathcal F_2$ and we define
$\bar\Pb$ to be a semidirect product
%
\begin{equation}
\bar\Pb[A] = \int_{\Omega_1} \mathbb P^{\omega_1}_2
\bigl[\bigl\{\omega_2\dvtx (\omega_1,\omega_2)
\in A\bigr\}\bigr] \mathbb P_1(d \omega_1).
\end{equation}
For $\bar\omega=  (\omega_1,\omega_2) \in\bar\Omega$, we
define sequences
of probability measures $(\pi^\varepsilon_z(\bar\omega))_{z\in
\mathbb Z}$,
$\varepsilon>0$, by requiring that
%
\begin{equation}
\label{e:GFINpsis} \Psi_\varepsilon\bigl(\pi^\varepsilon_z(
\bar\omega)\bigr) =\psi^\varepsilon_z(\bar\omega).
\end{equation}
This determines $\pi^\varepsilon_z(\bar\omega)$ uniquely, because
$\Psi_\varepsilon$
is an affine transformation of the Laplace transform. Since
$(m_z^\varepsilon)_z $ has the same distribution as $(m (\pi_z))_z$ and
$(\psi_z^\varepsilon)_z $ has the same distribution as
$(\Psi(\pi_z^{m_z^\varepsilon}))_z$, it follows that for every
$\varepsilon>0$,
$(\pi_z^\varepsilon)_z$ has the same distribution as $(\pi_z)_z$.

Finally, we set $\bar\mu^\varepsilon_{\bar\omega} $ to be the trap
measure on $(\tilde\Omega,\tilde{\mathcal F},\tilde\Pb)$ with the
trapping landscape $(\pi_z^\varepsilon(\bar\omega) )_z$, and define
$\bar\mu^\varepsilon$ to be mixture of
$\bar\mu^\varepsilon_{\bar\omega} $
w.r.t. $\bar\Pb$. From the previous discussion, it is obvious that
$\bar\mu^\varepsilon$ satisfy (a), (b) of the
Lemma~\ref{l:GFINcoupling}. We define
%
\begin{equation}
\bar\mu^0_{\bar\omega} =\mu_{(x_i(\omega_1),\sigma^{1+1/\gamma}_{v_i(\omega_1)^{-\gamma}
}f_i(\omega_1))}
\end{equation}
(see Example~\ref{ex:GFIN} for the notation)
and set $\bar\mu^0$ to be mixture of $\bar\mu^0_{\bar\omega}$
w.r.t. $\bar\Pb$. The measure $\bar\mu^0$ clearly satisfies (a), (c) of
the lemma.

\emph{$\bar\Pb$-a.s. convergence of
$\bar\mu^\varepsilon_{\bar\omega}$.} We need to show that
$\bar\Pb$-a.s., the trap measures
$\rho(\varepsilon)\mathfrak S_\varepsilon(\bar\mu^\varepsilon
_{\bar\omega})$
converge to the L\'evy trap measure $\bar\mu^0_{\bar\omega}$ vaguely
in distribution. Using Proposition~\ref{p:Levymeasureconv}, it is
sufficient to check that
%
\begin{equation}
\label{} \rho(\varepsilon)\mathfrak S_\varepsilon\bigl(\bar
\mu^\varepsilon_{\bar\omega}\bigr) \bigl(I\times[0,1]\bigr) \xrightarrow{
\varepsilon\to0} \bar\mu^0_{\bar\omega}\bigl(I\times[0,1]\bigr)
\end{equation}
$\bar\Pb$-a.s., in distribution, for every interval $I=[a,b]$ whose boundary
points are not in the set $\{x_i\dvtx i\in\mathbb N\}$. Computing Laplace
transforms, and taking $-\log$, the last display is equivalent to
%
\begin{equation}
\label{e:GFINLap} -\sum_{z:z\varepsilon\in I} \varepsilon^{-1}
\log \bigl(\hat \pi^\varepsilon_z(\bar\omega) \bigl(\rho(
\varepsilon)\lambda\bigr) \bigr) \xrightarrow{\varepsilon\to0} \sum
_{i:x_i\in I} \sigma^{1+1/\gamma}_{v_i(\bar\omega)^{-\gamma}
} f_i(
\bar \omega) (\lambda),
\end{equation}
for all $\lambda\ge0, \bar\Pb\mbox{-a.s.}$

We fix $\delta, \delta'>0$ (depending on $\bar\omega$) such that
%
\begin{equation}
\label{e:malasuma} \sum_{i:x_i\in I}v_i\1\bigl
\{v_i\le\delta'\bigr\} \le\delta.
\end{equation}
This is always possible as $V$ is an increasing pure jump process and
$V(b)-V(a)$ is $\bar\Pb$-a.s. finite. We define a finite set
$J:=\{i\dvtx x_i\in I, v_i>\delta'\}$. We consider $\varepsilon$ small
enough so that $z_i^\varepsilon$ is defined for all $i\in J$, and set
$J^\varepsilon=\{z_i^\varepsilon\dvtx  i\in J\}$. We consider separately
the sum over $J$ and its complement.

We start with the sum over $J$. Observe that as the boundary points of
$I$ are not in $\{x_i\}_i$, $\varepsilon z^\varepsilon_i\in I$ for all
$\varepsilon$ small enough. By the coupling construction, more
precisely by \eqref{e:psisconv} and \eqref{e:GFINpsis}, using that $J$
is finite and some elementary analysis, we see that for $\delta$,
$\delta'$ fixed, $\bar\Pb$-a.s.,
%
\begin{eqnarray}
-\sum_{z\in J^\varepsilon} \varepsilon^{-1}\log \bigl(
\hat \pi^\varepsilon_z(\bar\omega) \bigl(\rho(\varepsilon)\lambda
\bigr) \bigr) \xrightarrow{\varepsilon\to0} \sum_{i\in J}
\sigma^{1+1/\gamma}_{v_i(\bar\omega)^\gamma} f_i(\bar \omega) (\lambda)
\nonumber
\\[-8pt]
\\[-8pt]
\eqntext{\forall\lambda\ge0, \bar\Pb\mbox{-a.s.}}
\end{eqnarray}

The contribution of $i\notin J$ might be neglected on the right-hand
side of~\eqref{e:GFINLap}. Indeed, by Remark~\ref{r:fboudedness} and
\eqref{e:malasuma},
%
\begin{eqnarray}
0&\le&\sum_{i\notin J:x_i\in I} \sigma^{1+1/\gamma}_{v_i(\bar\omega)^{-\gamma} }
f_i(\bar\omega) (\lambda) = \sum_{i\notin J:x_i\in I}
v_i^{-\gamma} f_i\bigl(v_i^{\gamma+1}
\lambda\bigr)
\nonumber
\\[-8pt]
\\[-8pt]
\nonumber
&\le& \lambda \sum_{i\notin J:x_i\in I} v_i
\le\lambda\delta.
\end{eqnarray}

Finally, the contribution of the sum over $z\notin J^\varepsilon$ on
the left-hand side of \eqref{e:GFINLap} is asymptotically negligible.
Indeed, as $J$ is finite,
$\varepsilon^{1/\gamma} m_{z_i^\varepsilon}^\varepsilon\to v_i$ for
every $i\in J$, and
$\varepsilon^{1/\gamma}V^\varepsilon_{\varepsilon^{-1}\cdot} $
converges to $V$, it follows that for $\varepsilon$ small enough
%
\begin{equation}
d(\varepsilon)^{-1 } \sum_{z\in\varepsilon^{-1}I\setminus J^\varepsilon}
m_z^\varepsilon\le2 \delta.
\end{equation}
It follows that $m_z^\varepsilon\le2\delta d(\varepsilon)$,
and thus
$m^\varepsilon_z \rho(\varepsilon)\xrightarrow{\varepsilon\to0} 0$,
for every $z\notin J^\varepsilon$. From
$m(\pi^\varepsilon_z)=m^\varepsilon_z$, it follows that
$\hat\pi^\varepsilon_z(\rho(\varepsilon) \lambda)\ge
1-m^\varepsilon_z \rho(\varepsilon)\lambda$.
Using the inequality $-\log x \le2(1-x)$ which holds in some interval
$(c,1]$, we obtain
%
\begin{eqnarray}
0&\le& -\sum_{z\in\varepsilon^{-1}I\setminus J^\varepsilon} \varepsilon^{-1}\log
\bigl(\hat \pi^\varepsilon_z(\bar\omega) \bigl(\rho(\varepsilon)
\lambda\bigr) \bigr) \le2 \sum_{z\in\varepsilon^{-1}I\setminus J^\varepsilon}
\varepsilon^{-1} m_z^\varepsilon \rho(\varepsilon)
\lambda
\nonumber
\\[-8pt]
\\[-8pt]
\nonumber
&=&2\lambda \varepsilon^{1/\gamma} \sum_{z\in\varepsilon^{-1}I\setminus J^\varepsilon}
m_z^\varepsilon\le4\lambda\delta,
\end{eqnarray}
by \eqref{e:malasuma} again. This completes the proof.
\end{pf*}

\subsection{Convergence to FIN}\label{ss:FINconvergence} 
Since the FIN diffusion is a special case of the SSBM (see
Definition~\ref{d:FIN}), we can specialize Theorem~\ref{t:RSBMconv}
to obtain criteria for the convergence of
a rescaled RTRW with i.i.d. trapping landscape is the FIN diffusion.
Here, we present the proof of such convergence as stated in Theorem~\ref{t:FINconv}.
We recall that $\mu$ is a trapping measure of a RTRW
$X=Z[\mu]$ with an i.i.d. random trapping landscape $\bolds\pi$
whose marginal is $P$.

\begin{pf*}{Proof of Theorem~\ref{t:FINconv}}
Due to Definition~\ref{d:FIN} and the scaling property~\eqref{e:GFINF},
we only need to verify Assumption (L) with
$\mathbb F_1=\delta_{\lambda\mapsto\lambda}$.
For all positive~$x$, it holds that
$x-\frac{x^2}{2}\leq1-e^{-x}\leq x$.
Inserting this inequality in the definition of $\Psi_\varepsilon$, we
obtain
%
\begin{eqnarray}
\varepsilon^{-1} \bigl(\lambda q(\varepsilon) m\bigl(
\pi^{d(\varepsilon)}\bigr) -\tfrac{1}2 q(\varepsilon)^2
\lambda^2 m_2\bigl(\pi^{d(\varepsilon)}\bigr) \bigr) &\leq&
\Psi_\varepsilon\bigl(\pi^{d(\varepsilon)} \bigr) (\lambda)
\nonumber
\\[-8pt]
\\[-8pt]
\nonumber
&\leq&\varepsilon^{-1}\lambda q(\varepsilon) m\bigl(\pi ^{d(\varepsilon)}
\bigr).
\end{eqnarray}
Taking the limit $\varepsilon\to0$ in this inequality, recalling
$q(\varepsilon) = \varepsilon d(\varepsilon)^{-1}$, we obtain using
the assumptions of the theorem
%
\begin{equation}
\lim_{\varepsilon\to0} \Psi_\varepsilon\bigl(\pi^{d(\varepsilon)}
\bigr) (\lambda) = \lambda
\end{equation}
in distribution. This completes the proof.
\end{pf*}

\section{Applications} 
\label{s:applications}

In this section, we make use of the previously developed theory to
prove Theorems \ref{t:phasediagramtransparenttraps} and \ref
{t:phasediagramcombmodel}.

\subsection{The simplest case of a phase transition} 
\label{ss:faketrap}
Recall from Definition~\ref{d:faketrap}, that the trap model with
transparent traps is defined using two positive parameters $\alpha$,
$\beta$, a family $(\tau_x)_{x\in\mathbb{Z}}$
of i.i.d. random variables satisfying $\tau_x>1$ and
%
\begin{equation}
\label{e:newtail} \lim_{u\to\infty}u^{\alpha}\mathbb{P}(
\tau_0>u)=c\in (0,\infty),
\end{equation}
and its i.i.d. trapping
landscape $\bolds\pi= (\pi_x)_{x\in\Z}$, where
%
\begin{equation}
\pi_x(\omega):=\bigl(1-\tau_x(\omega)^{-\beta}
\bigr)\delta_1 +\tau_x(\omega)^{-\beta}
\delta_{\tau_x(\omega)}.
\end{equation}
In words, given $\tau_x$'s, at site $x$ the walk is trapped for time
$\tau_x$ with
probability $\tau_x^{-\beta}$, otherwise it spends just a unit time
at $x$.
Here, we present the proof of Theorem~\ref{t:phasediagramtransparenttraps}.
%
\begin{remark}
For the sake of simplicity, during the computations we will replace the
traps $\pi_x:=(1-\tau_x^{-\beta})\delta_1
+\tau_x^{-\beta}\delta_{\tau_x}$ by $(1-\tau_x^{-\beta})\delta_0
+\tau_x^{-\beta}\delta_{\tau_x}$. It should be clear that the
asymptotics should be the same in both cases.
\end{remark}

\begin{pf*}{Proof of Theorem~\ref{t:phasediagramtransparenttraps}}
By the definition of the model,
$m(\pi_z(\omega)) = \tau_z(\omega)^{1-\beta}$,
and thus
%
\begin{equation}
\label{e:transtail} \lim_{x\to\infty} x^{\alpha/{(1-\beta)}} \Pb\bigl[ m(
\pi_z)\ge x\bigr]=1.
\end{equation}
When $\alpha+\beta>1$, $m(\pi_z)$ has finite expectation, and
Theorem~\ref{t:BMconvergence} yields claim (i).

For claims (ii) and (iv), Condition (HT) is verified due to
\eqref{e:transtail}. The function $d(\varepsilon)$ introduced in
Remark~\ref{rk:assumptionpp} may be chosen to be
$d(\varepsilon) = \varepsilon^{-1/\gamma}$. Conditioning on
$m(\pi_0)=d(\varepsilon)$ is equivalent to conditioning on
$\tau_0^{1-\beta}=\varepsilon^{-1/\gamma}$, which, in turn, is
equivalent to $\tau_0=\varepsilon^{-1/\alpha}$. Hence, conditionally on
$m(\pi_0)=d(\varepsilon)$, $\pi_0$ is deterministic probability
measure
$\pi^{d(\varepsilon)}_z=(1-\varepsilon^{\beta/\alpha})\delta
_0+\varepsilon^{\beta/\alpha}\delta_{\varepsilon^{-1/\alpha}}$,
and
%
\begin{equation}
\hat{\pi}^{d(\varepsilon)}(\lambda) =1-\varepsilon^{\beta/\alpha} +
\varepsilon^{\beta/\alpha}\exp\bigl(-\lambda\varepsilon^{-1/\alpha}\bigr).
\end{equation}
Therefore, $\Psi_\varepsilon(\hat{\pi}^{d(\varepsilon)})$ is
deterministic,
%
\begin{equation}
\Psi_\varepsilon\bigl(\hat{\pi}^{d(\varepsilon)}\bigr) (\lambda) =
\varepsilon^{(\beta-\alpha)/\alpha} \bigl(1-\exp\bigl(-\lambda\varepsilon^{(\alpha-\beta)/\alpha}
\bigr)\bigr).
\end{equation}

When $\alpha+ \beta<1$ and $\alpha>\beta$, this implies
$\lim_{\varepsilon\to0}\Psi_\varepsilon(\hat{\pi}^{d(\varepsilon
)})(\lambda)=\lambda$.
Hence, Condition (L) is verified,
and Theorem~\ref{t:RSBMconv} together with Definition~\ref{d:FIN} yields
claim (ii).

Similarly, when $\alpha+ \beta<1$ and $\alpha= \beta$,
$\lim_{\varepsilon\to0}\Psi_\varepsilon(\hat{\pi}^{d(\varepsilon
)})(\lambda)=1-\exp(-\lambda)$,
which implies (iv). Observe that in this case, the traps are
``Poissonian'' in the sense that $\mathbb{F}_1$ is concentrated on
$\lambda\mapsto1-\exp(-\lambda)$, which is the Laplace exponent of a
Poisson process.

When $\alpha+\beta<1$ and $\alpha<\beta$,
$\Psi_\varepsilon(\hat{\pi}^{d(\varepsilon)})$ converges to $0$,
indicating that Theorem~\ref{t:FKconv} should be used instead of
Theorem~\ref{t:RSBMconv}. Recall that
$\Gamma(\varepsilon)=\E(1-\hat{\pi}(\varepsilon))$. We will first show
that $\Gamma(\varepsilon)$ is regularly varying of index $\kappa$ at
$\varepsilon=0$. Let $\nu$ be the distribution of $\tau_0$. Then
%
\begin{equation}
\E\bigl(1-\hat{\pi}(\varepsilon)\bigr)= \int_0^{\infty}
t^{-\beta}\bigl(1-\exp(-\varepsilon t)\bigr)\nu(dt).
\end{equation}
Changing variables, we obtain
%
\begin{equation}
\E\bigl(1-\hat{\pi}(\varepsilon)\bigr) =\varepsilon^\beta\int
_0^{\infty}t^{-\beta}\bigl(1-\exp(-t)\bigr)\nu
\bigl(\varepsilon^{-1} \,dt\bigr).
\end{equation}
By \eqref{e:newtail},
$\varepsilon^{-\alpha}\nu(\varepsilon^{-1} \,dt)$ converges weakly to
$c \alpha t^{-1-\alpha}\,dt$. Hence, as $\varepsilon\to0$,
%
\begin{equation}
\E\bigl(1-\hat{\pi}(\varepsilon)\bigr) =c \alpha\varepsilon^{\alpha+\beta}\int
_0^{\infty}t^{-1-\alpha
-\beta}\bigl(1-\exp(-u)\bigr) \,du
\bigl(1+o(1)\bigr).
\end{equation}
The integral on the right-hand side is finite, so the condition
\eqref{e:qFKcond} [cf. also \eqref{e:qFKcondeq}] of
Theorem~\ref{t:FKconv} is verified with
$q_{\FK}(\varepsilon) = \varepsilon^{2/\kappa}$. Similarly, as
$\varepsilon\to0$,
%
\begin{eqnarray}
\E\bigl(\bigl(1-\hat{\nu}(\varepsilon)\bigr)^{2}\bigr) &=&\alpha\int
_0^\infty t^{-2\beta}\bigl(1-\exp(-
\varepsilon t)\bigr)^2 \nu(dt)
\nonumber
\\[-8pt]
\\[-8pt]
\nonumber
&=&\alpha\varepsilon^{2\beta+\alpha} \int_0^\infty
u^{-2\beta-1-\alpha}\bigl(1-\exp(-u)\bigr)^2 \,du \bigl(1+o(1)\bigr),
\end{eqnarray}
leading to
%
\begin{equation}
\varepsilon^{-3}\Pb\bigl(\bigl(1-\hat{\pi}\bigl(q_{\FK}(
\varepsilon)\bigr)\bigr)^2\bigr) \xrightarrow{\varepsilon\to0} 0.
\end{equation}
This verifies the assumptions of Theorem~\ref{t:FKconv} and proves
claim (iii).
\end{pf*}

\subsection{The comb model} 
\label{ss:combmodel}
In this section we give the proof of Theorem~\ref{t:phasediagramcombmodel}.

\begin{pf*}{Proof of Theorem~\ref{t:phasediagramcombmodel}}
To prove the theorem, we first need to control the distribution of the time
that the random walk $Y^\comb$ spends in the teeth of the comb.
Therefore, for $N\ge1$, we let $V^N = (V^N_k)_{k\ge0}$ be a random
walk on
$\{0,\ldots,N\}$ with drift $g(N)$, reflected at $N$, started from
$V^N_0=1$. Let $\tau^N = \inf\{n\ge0, V^N_n=0\}$ be the hitting time of
$0$ by $V^N$, and let $\theta^N$ be the law of $\tau^N$.

It is easy to see that the distribution $\pi_z$ of the time that
$X^\comb$ spends on one visit to $z$ coincides with the law of
$1+\sum_{i=1}^G (1+ \xi^z_i)$, where $\xi^z_i$ are i.i.d. with distribution
$\theta^{N_z}$, and $G$ is a geometric random variable with parameter
$\frac{2}3$, $\Pb[G=k]=\frac{2}3(\frac{1}3)^k$, $k\ge0$ (for $G=0$ the
above sum is zero, by definition).
In particular,
%
\begin{eqnarray}
\label{e:firstmomentthetattopi} m(\pi_z)&= &\frac{3+m(\theta^{N_z})}2,
\\
\label{e:secondmomomentthetatopi} m_2(\pi_z)&=&\frac{1}2
\bigl(m_2\bigl(\theta^{N_z}\bigr)+6m\bigl(
\theta^{N_z}\bigr)+m\bigl(\theta^{N_z}\bigr)^2+6
\bigr),
\\
\hat\pi_z(\lambda)&= &\frac{2 }{3e^{\lambda}-\hat\theta^{N_z}
(\lambda)}.
\end{eqnarray}
As a consequence,
%
\begin{equation}
\label{e:hatthetatopi} 1-\hat\pi_z(\lambda) = \tfrac{1}2 \bigl(3
\lambda + \bigl(1-\hat\theta^{N_z}(\lambda)\bigr) \bigr) \bigl(1+o(1)
\bigr) \qquad\mbox{as $\lambda\to0$}. 
\end{equation}

The distribution $\theta^N$ is characterized by the following lemma.

\begin{lemma}
Let $p=(1+g(N))/2$, $\xi=(1-p)/p$, and
%
\begin{equation}
\label{e:chi} \chi(s) = \chi(s,p) = \frac{1+\sqrt{1-4s^2p(1-p)}}{2sp},\qquad s\in(0,1].
\end{equation}
Then the generating function of $\theta^N$ is given by
%
\begin{eqnarray}
\label{generatingfunction} \hat\theta^N(-\log s) &=& \mathbb E
\bigl[s^{\tau^N}\bigr]
\nonumber
\\[-8pt]
\\[-8pt]
\nonumber
&=&\frac{\xi\chi(s)^{2N-2}(\chi(s)-s)+\xi^{N-1}\chi(s)(s\chi
(s)-\xi)}{
\chi(s)^{2N-1}(\chi(s)-s)+\xi^{N-1}(s\chi(s)-\xi)}.
\end{eqnarray}
\end{lemma}

\begin{pf}
The proof is a standard one-dimensional random walk computation.
Writing $f_x(s) = \mathbb E[s^{\tau^N}|V_0=x]$ for the generating
function of $\tau^N$ for the random walk starting at $x$ [i.e., $\hat
\theta^N(-\log s) = f_1(s)$], we
have the equation
%
\begin{equation}
\label{recurrence} f_x(s)=spf_{x+1}(s)+s(1-p)f_{x-1}(s)\qquad
\mbox{for }1\leq x\leq N-1,
\end{equation}
with the boundary conditions $f_0(s)=1$, and $f_N(s)=sf_{N-1}(s)$.
Solving this system, we obtain
%
\begin{equation}
f_x(s)=A_+(s)\lambda_+(s)^x+A_-(s)\lambda_-(s)^x,
\end{equation}
with $\lambda_+(s)=\chi(s) $, $\lambda_-(s)=\xi/\chi(s) $ and
%
\begin{eqnarray}
A_+(s)&=&\frac{-\lambda_-(s)^{N-1} (\lambda_-(s)-s)}{
\lambda_+(s)^{N-1}(\lambda_+(s)-s)-\lambda_-(s)^{N-1}(\lambda_-(s)-s)},
\\
A_-(s)&=&\frac{\lambda_+(s)^{N-1}(\lambda_+(s)-s)}{
\lambda_+(s)^{N-1}(\lambda_+(s)-s)-\lambda_-(s)^{N-1}(\lambda_-(s)-s)}.
\end{eqnarray}
A simple rearrangement yields the claim.
\end{pf}

Knowing the generating function, the moments of $\theta^N$ can be
obtained easily. We collect the asymptotic behavior of the first
and second moments in the following lemma.
Its proof is an easy asymptotic analysis of the derivatives of the
generating function of $\theta^N$ and is omitted.

\begin{lemma}
\label{l:moments}
When $\beta>0$, as $N\to\infty$,
the first and second moment of $\theta^N$ satisfy
%
\begin{equation}
\label{e:moments} m\bigl(\theta^N\bigr) \sim\frac{N^{2\beta+1}}{\beta\log(N)},\qquad
m_2\bigl(\theta^N\bigr)\sim\frac{N^{3+4\beta}}{\beta^3\log^3(N)},
\end{equation}
where $f\sim g$ as $N\to\infty$ means
$\lim_{N\to\infty}f/g=1$. Moreover, when $\beta=0$, then
$m(\theta^N)\sim2 N$.
\end{lemma}

We further need to control the behavior of
$1-\hat\theta^N(\varepsilon)$ as $\varepsilon\to0$ and for $N$
possibly diverging with $\varepsilon$. This is the content of the next
two lemmas. Both these lemmas contain an asymptotic statement where the
dependence of $N$ on $\varepsilon$ is explicitly given (which later
will be used to control the dominant contributions of various
convergence conditions), and an upper bound that holds uniformly over
$N$ for all $\varepsilon$ small enough (which will be used to
bound the error terms).

\begin{lemma}
\label{l:tanh}
Let $\beta=0$. Then, for every $y>0$,
%
\begin{equation}
\frac{(1-\hat\theta^{\lfloor y /\sqrt{2\varepsilon} \rfloor
}(\varepsilon))}{
\sqrt{2\varepsilon}}\xrightarrow{\varepsilon\to0} \tanh(y),
\end{equation}
and there exists $c>0$ such that for all $N\ge1$ and
$\varepsilon\in(0,1/2)$
%
\begin{equation}
1-\hat\theta^N(\varepsilon)\le c \sqrt{\varepsilon}.
\end{equation}
\end{lemma}

\begin{pf}
From \eqref{generatingfunction}, we obtain
%
\begin{eqnarray}
\label{e:oneminustheta} &&1 -\hat\theta^N (-\log s)
\nonumber
\\[-8pt]
\\[-8pt]
\nonumber
&&\qquad= \frac{(\chi(s) - \xi)\chi
(s)^{2N-2}(\chi(s) - s) + \xi^{N-1}(1 - \chi(s))(s\chi(s) -
\xi)}{\chi(s)^{2N-1}(\chi(s) - s)+\xi^{N-1}(s\chi(s) - \xi)}.
\end{eqnarray}
When $\beta=0$, then $\xi=1$, and $\chi(s) =(1+\sqrt{1-s^2})/s$
(which is independent of~$N$).
Therefore, setting $s=e^{-\varepsilon} \sim1- \varepsilon$, we find as
$\varepsilon\to0$,
%
\begin{equation}
\label{e:chichi} \chi(s) - 1 \sim\sqrt{2\varepsilon}.
\end{equation}
For $N=\lfloor y/\sqrt{2\varepsilon}\rfloor$, this
and \eqref{e:oneminustheta} imply that
%
\begin{equation}
1-\hat\theta^N(\varepsilon)\sim\sqrt{2\varepsilon}
\frac{(1+\sqrt{2\varepsilon})^{2N}-1}{
(1+\sqrt{2\varepsilon})^{2N}+1} \sim\sqrt{2\varepsilon} \tanh(y),
\end{equation}
proving the first claim of the lemma. Further, for
$\varepsilon\in(0,1/2)$ there is $c\in(0,1)$ such that
$c \sqrt{2\varepsilon} \le\chi(s) -1 \le c^{-1} \sqrt{2\varepsilon}$.
Inserting this into \eqref{e:oneminustheta} implies the second claim.
\end{pf}

\begin{lemma}
\label{l:betagezero}
Let $\beta>0$ and set
%
\begin{eqnarray}
u(\varepsilon) &= &\varepsilon^{-1/(2+2\beta)} \log^{1/(1+\beta)}\bigl(
\varepsilon^{-1}\bigr),
\\
v(N,\varepsilon)&=& \frac{2\beta N^{1+2\beta}\log N}{
N^{2+2\beta}+2\beta^2 \varepsilon^{-1} \log^2 N}.
\end{eqnarray}
Then
%
\begin{equation}
\label{e:astheta} \sup_{u(\varepsilon)^{1/2}\le N \le u(\varepsilon)^{1+\beta/2}} \biggl|\frac{1-\hat\theta^N(\varepsilon)}{v(N,\varepsilon)} -1 \biggr|
\xrightarrow{\varepsilon\to0} 0,
\end{equation}
and there is a constant $c<\infty$ such that for all
$\varepsilon\in(0,1/2)$ and all $N$ in given regimes
%
\begin{equation}
\label{e:ubtheta} 1-\hat\theta^N (\varepsilon) \le \cases{
\varepsilon m\bigl(\theta^N\bigr), &\quad $N < u(\varepsilon)^{1/2}$,
\vspace *{2pt}
\cr
c g, & \quad$u(\varepsilon)^{1+
\beta/2}<N<u(\varepsilon)^{1+\beta}$,
\vspace*{2pt}
\cr
c \sqrt\varepsilon, &\quad  $N > u(\varepsilon)^{1+ \beta}$.}
\end{equation}
\end{lemma}

\begin{pf}
The first line of \eqref{e:ubtheta} follows from the fact that
$1-\hat\nu(\lambda)\le\lambda m(\nu)$ for every
probability distribution $\nu$ supported on $[0,\infty)$.

For the remaining parts of \eqref{e:ubtheta}, observe that
%
\begin{equation}
\label{e:thetaxichi} \hat\theta^N (-\log s) \ge\xi/\chi(s).
\end{equation}
To see that this inequality holds, it is
sufficient to replace $\hat\theta^N(-\log s)$ by the right-hand side
of \eqref{generatingfunction}, multiply the resulting inequality by the
denominator (which is always positive) and observe that
$\chi(s) \ge1 \ge\xi$. Using \eqref{e:thetaxichi},
%
\begin{equation}
\label{e:meziaa} 1-\hat\theta^N(-\log s) \le\bigl(\chi(s) -\xi
\bigr)/\chi(s) \le\chi(s) -\xi.
\end{equation}
From the definition \eqref{e:chi} of $\chi(s)$, it follows that
%
\begin{equation}
\label{e:chias} \chi(s)-1 = \frac{(1-s)(1+g) -g + \sqrt{2(1-s) -(1-s)^2 + s^2 g^2}}{
s(1+g)}.
\end{equation}
Writing $s = e^{-\varepsilon}\sim1-\varepsilon$ as $\varepsilon\to
0$, and
observing the fact that $1-\xi\sim2g$ as $N\to\infty$ (or equivalently
as $g \to0$), and inserting those into \eqref{e:chias}, we obtain that
%
\begin{equation}
\label{e:chibound} \chi(s) -1 \le c \biggl(\sqrt{\varepsilon} + \frac\varepsilon g
\biggr)
\end{equation}
for some sufficiently large $c$. [To see that \eqref{e:chibound}
holds, it is useful to observe that
$ \chi(s) -1 \sim\sqrt{2\varepsilon}$ when $ g^2\ll\varepsilon$,
and $\chi(s) -1\sim\frac\varepsilon g$ when
$ 1\gg g^2\gg\varepsilon$.] Going back to \eqref{e:meziaa}, this
implies that
%
\begin{equation}
1-\hat\theta^N(-\log s) \le c\biggl(\sqrt\varepsilon + \frac
\varepsilon g + g\biggr).
\end{equation}
Observing further that when $N=u(\varepsilon)^{1+\beta}$, then $g^2$
is comparable with $\varepsilon$, the remaining parts of
\eqref{e:ubtheta} follow.

It remains to show \eqref{e:astheta}. A simple analysis of
formula \eqref{e:chias} implies that
$\chi(s) -1 \sim\varepsilon/g$ uniformly over $N$ in
the considered regime [i.e., in the same sense as in
\eqref{e:astheta}]. In addition, $1-\xi\sim2g$, and
thus $\xi^{N-1}\sim N^{-2\beta}$, and $\chi(s)^{2N-1}\sim1$ since
$\varepsilon/g \ll N^{-1}$. Inserting these observations into
\eqref{e:oneminustheta} proves \eqref{e:astheta}.
\end{pf}

We can now proceed with the proof of
Theorem~\ref{t:phasediagramcombmodel}. From
\eqref{e:firstmomentthetattopi} and Lemma~\ref{l:moments}, it follows
that for $\beta\ge0$,
%
\begin{equation}
\label{eq:heavytailscombmodel} \mathbb P\bigl[m(\pi_0)\ge x\bigr]=x^{-\gamma}L(x),
\end{equation}
for $\gamma= \alpha/(1+2\beta)$ and a slowly varying function $L$.
This implies that $\mathbb E[m(\pi_0)]$ is finite for
$\alpha> 1+2\beta$, and claim (i) follows by applying Theorem~\ref
{t:BMconvergence}.

To show claim (ii), we observe that when $\alpha>1$, then Lemma~\ref
{l:moments} implies that
$m(\theta^N)^{2+\gamma}\gg m_2(\theta^N)$ as $N\to\infty$, which is
sufficient to check the assumptions of Theorem~\ref{t:FINconv}.

The line $\alpha=1,\beta>0$ requires a sharper analysis. Let
$\mathcal N(x)$
be defined by the relation
%
\begin{equation}
\mathcal N(x)=\inf \bigl\{N\dvtx \tfrac{1}2\bigl(3+m\bigl(
\theta^N\bigr)\bigr)\ge x \bigr\}.
\end{equation}
Then, using \eqref{e:firstmomentthetattopi},
for a constant $c>0$,
%
\begin{equation}
\mathbb P\bigl[m(\pi_0)\ge x\bigr] =\sum
_{N=\mathcal N(x)}^{\infty} \mathcal{Z}^{-1}N^{-2}
\sim\frac{c} {\mathcal N(x)} \qquad\mbox{as $x\to\infty$}.
\end{equation}
Therefore, the slowly varying function $L$ in
\eqref{eq:heavytailscombmodel} satisfies
%
\begin{equation}
\label{e:asymptoticsofl} L(x)\sim c\mathcal N(x)^{-1} x^{\gamma}\qquad
\mbox{as $x\to\infty$}.
\end{equation}
From the classical theory of convergence to stable laws (see
\cite{Whi02}, Theorem~4.5.1) it follows that the function
$d(\varepsilon)$ (defined in Remark~\ref{rk:assumptionpp}) satisfies
%
\begin{equation}
\label{e:asymptoticsofla} \frac{L(d(\varepsilon))}{\varepsilon d(\varepsilon)^{\gamma}} \xrightarrow{\varepsilon\to0}
C_{\gamma},
\end{equation}
where $C_{\gamma}$ is a positive constant.
Combining \eqref{e:asymptoticsofl} and \eqref{e:asymptoticsofla} implies
%
\begin{equation}
\label{e:asymptoticsofm-1depsilon} \mathcal N\bigl(d(\varepsilon)\bigr)\sim c \varepsilon^{-1}\qquad
\mbox{as $\varepsilon\to0$}.
\end{equation}
Therefore, recalling the definition of $\mathcal N$
and Lemma~\ref{l:moments}, we find that
%
\begin{equation}
d(\varepsilon)\sim \frac{3+m(\theta^{\lfloor c \varepsilon^{-1}\rfloor})}{2} \sim\frac{c'\varepsilon^{-2\beta-1}}{ \log(\varepsilon^{-1})}
\end{equation}
for some $c'>0$, and thus, using \eqref{e:secondmomomentthetatopi} and
Lemma~\ref{l:moments},
%
\begin{equation}
m_2\bigl(d(\varepsilon)\bigr)\sim \frac{1}{2}m_2
\bigl(\theta^{\lfloor c\varepsilon^{-1}\rfloor}\bigr) \sim\frac{c \varepsilon^{-3-4\beta}}{\log^3(\varepsilon^{-1})}.
\end{equation}
We are ready to check the condition of Theorem~\ref{t:FINconv}.
It follows from the above computations that
%
\begin{equation}\qquad
\label{e:criticallineonthecomb} \varepsilon d(\varepsilon)^{-2}m_2\bigl(d(
\varepsilon)\bigr) \sim c\varepsilon\cdot {\varepsilon^{4\beta+2}} { \log
\bigl(\varepsilon^{-1}\bigr)}\cdot \frac{\varepsilon^{-3-4\beta}}{\log^3(\varepsilon^{-1})} \sim c
\log^{-1}\bigl(\varepsilon^{-1}\bigr).
\end{equation}
The right-hand side of the last display converges to $0$
as $\varepsilon\to0$, which verifies the condition of Theorem~\ref{t:FINconv}, and the second part of claim (ii) follows.

For claim (iii), we need to check the assumptions of
Theorem~\ref{t:FKconv}. Using~\eqref{e:hatthetatopi} and dominated
convergence,
%
\begin{equation}
\label{e:poa} \Gamma(\varepsilon)=\E\bigl(1-\hat\pi_0(\varepsilon)
\bigr) \sim\frac{3\varepsilon}{2} + \frac{1}{2 \mathcal Z} \sum
_{N=1}^{\infty}N^{-1-\alpha}\bigl(1-\hat
\theta^N(\varepsilon)\bigr).
\end{equation}
We now discuss separately the cases $\beta=0 $ and
$\beta>0$.

When $\beta=0$, choosing $\delta>0$ small,
using the first claim of Lemma~\ref{l:tanh},
and the change of variables $y=\sqrt{2\varepsilon}N$ we obtain for
$\varepsilon\to0$
%
\begin{eqnarray}
\label{e:zxcvv}&& \sum_{N={\delta}/ {\sqrt{2\varepsilon}} }^{{1}/{(\delta\sqrt
{2\varepsilon})} }
N^{-1-\alpha}\bigl(1-\hat\theta^N(\varepsilon)\bigr)
\nonumber
\\
&&\qquad\sim \int_{\delta}^{\delta^{-1}} (2\varepsilon)^{{(1+\alpha)}/{2}}
y^{-1-\alpha}\bigl(1-\hat\theta^{\lfloor y/\sqrt{2\varepsilon}\rfloor
}(\varepsilon)\bigr) (2
\varepsilon)^{-1/2} \,dy
\\
&&\qquad\sim (2\varepsilon)^{{(1+\alpha)}/2} \int_{\delta}^{\delta^{-1}}
y^{-1-\alpha}\tanh(y)\,dy.\nonumber
\end{eqnarray}
The second claim of Lemma~\ref{l:tanh} can be then used to justify that
%
\begin{equation}
\sum_{N=1 /(\delta\sqrt{2\varepsilon}) }^\infty N^{-1-\alpha}\bigl(1-
\hat\theta^N(\varepsilon)\bigr) \le c \delta^\alpha
\varepsilon^{{(1+\alpha)}/2 }.
\end{equation}
Using elementary\vspace*{1pt} properties of Laplace transform and
Lemma~\ref{l:moments} it follows that
$1-\hat{\theta}^N(\varepsilon)\leq\varepsilon m(\theta
^N)=2N\varepsilon$.
Therefore,
%
\begin{equation}
\label{e:pob} \sum_{N=1 }^{\delta/\sqrt{2\varepsilon} }
N^{-1-\alpha}\bigl(1-\hat\theta^N(\varepsilon)\bigr) \le c
\delta^{(1-\alpha) /2}\varepsilon^{{(1+\alpha)}/2 }.
\end{equation}
As the integral on the
right-hand side of \eqref{e:zxcvv} converges,
\eqref{e:poa}--\eqref{e:pob} imply that
$\Gamma(\varepsilon)$ is regularly varying with index
$\kappa=(1+\alpha)/2$, that is, \eqref{e:qFKcondeq} and thus
\eqref{e:qFKcond} holds for $q(\varepsilon) = \varepsilon^{2/\kappa}$.

To check \eqref{e:FKsecondmoment}, we estimate
$\mathbb E[(1-\hat\pi_0(\varepsilon))^2]$ first. Splitting the sum in
the same way as for \eqref{e:poa}, using the first claim of
Lemma~\ref{l:tanh}, we obtain
%
\begin{eqnarray}
\label{e:opi}&& \sum_{N=\delta/\sqrt{2\varepsilon} }^{1/(\delta\sqrt
{2\varepsilon}) }
N^{-1-\alpha}\bigl(1-\hat\theta^N(\varepsilon)
\bigr)^2
\nonumber
\\[-8pt]
\\[-8pt]
\nonumber
&&\qquad \sim (2\varepsilon)^{{(2+\alpha)}/2} \int
_{\delta}^{\delta^{-1}} y^{-1-\alpha}\tanh^2(y)\,dy.
\end{eqnarray}
Using the second claim of Lemma~\ref{l:tanh} and
$1-\hat\theta^N (\varepsilon)\le2\varepsilon N$ again then implies that
%
\begin{eqnarray}
\label{e:opii}&& \sum_{N=1 /(\delta\sqrt{2\varepsilon}) }^\infty
N^{-1-\alpha}\bigl(1-\hat\theta^N(\varepsilon)
\bigr)^2+ \sum_{N=1 }^{\delta
/\sqrt{2\varepsilon} }
N^{-1-\alpha}\bigl(1-\hat\theta^N(\varepsilon)
\bigr)^2
\nonumber
\\[-8pt]
\\[-8pt]
\nonumber
&&\qquad \le c(\delta)\varepsilon^{{(2+\alpha)}/{2}}
\end{eqnarray}
with
$c(\delta)\to0$ as $\delta\to0$. Replacing $\varepsilon$ by
$q(\varepsilon) = \varepsilon^{2/\kappa}$ in \eqref{e:opi},
\eqref{e:opii}, we obtain
%
\begin{equation}
\varepsilon^{-3} \E\bigl(\bigl(1-\hat{\pi}\bigl(q(\varepsilon)\bigr)
\bigr)^2\bigr)\sim c\varepsilon^{(1-\kappa)/\kappa}\xrightarrow{
\varepsilon\to0 } 0.
\end{equation}
Hence, the second assumption of Theorem~\ref{t:FKconv} is verified and
claim (iii) is proved for $\beta=0$.

We follow similar steps in the case $\beta>0$, using the estimates from
Lemma~\ref{l:betagezero}. We first get using the first part of
\eqref{e:ubtheta} and Lemma~\ref{l:moments}
%
\begin{eqnarray}
\sum_{N=1}^{u(\varepsilon)^{1/2}}N^{-1-\alpha}\bigl(1-
\hat\theta ^N(\varepsilon)\bigr) &\le& c \sum
_{N=1}^{u(\varepsilon)^{1/2}}N^{-1-\alpha} \varepsilon
\frac{N^{2\beta+1}} {\beta\log N}
\nonumber
\\[-8pt]
\\[-8pt]
\nonumber
&\le&\varepsilon^{(2\beta+\alpha+3)/(2\beta
+2)}L(\varepsilon) \ll\varepsilon^{\kappa},
\end{eqnarray}
where $L$ is a slowly varying function and
$\kappa=\frac{1+\alpha}{2\beta+2}$, as in the theorem. Further, by
the second part of \eqref{e:ubtheta},
%
\begin{eqnarray}
\sum_{N=u(\varepsilon)^{1+\beta/2}}^{u(\varepsilon)^{1+\beta}} N^{-1-\alpha}\bigl(1-
\hat\theta^N(\varepsilon)\bigr) &\le& \sum
_{N=u(\varepsilon)^{1+\beta/2}}^{u(\varepsilon)^{1+\beta}} c N^{-2-\alpha} \beta \log N
\nonumber
\\[-8pt]
\\[-8pt]
\nonumber
&\le&\bigl(\varepsilon^\kappa\bigr)^{{(2+\beta)}/2}L (\varepsilon)\ll
\varepsilon^\kappa,
\end{eqnarray}
and by the third part of \eqref{e:ubtheta},
%
\begin{eqnarray}
\sum_{N=u(\varepsilon)^{1+ \beta}}^{\infty} N^{-1-\alpha}\bigl(1-
\hat\theta^N(\varepsilon)\bigr) &\le& c \sqrt\varepsilon \sum
_{N=u(\varepsilon)^{1+ \beta}}^{\infty} N^{-1-\alpha}
\nonumber
\\[-8pt]
\\[-8pt]
\nonumber
&\le&\bigl(\varepsilon^\kappa\bigr)^{{1}/{(1+\beta)} }L (\varepsilon)\ll
\varepsilon^\kappa.
\end{eqnarray}
Using \eqref{e:astheta}, we then get for the remaining part of the sum
%
\begin{equation}
\sum_{N= u(\varepsilon)^{1/2}}^{u(\varepsilon)^{1+\beta/2}} N^{-1-\alpha}\bigl(1-
\hat\theta^N(\varepsilon)\bigr) \sim \sum
_{N= u(\varepsilon)^{1/2}}^{u(\varepsilon)^{1+\beta/2}} \frac{2\beta N^{2\beta-\alpha }\log N}{
N^{2+2\beta}+2\beta^2 \varepsilon^{-1} \log^2 N}.
\end{equation}
Substituting $N= u(\varepsilon)y$, an easy analysis yields
%
\begin{equation}
\sim \int_{u(\varepsilon)^{-1/2}}^{u(\varepsilon)^{\beta/2} } \frac{2\beta u(\varepsilon)^{2\beta-\alpha+1} y^{2\beta-\alpha}
\log(u(\varepsilon) y)}{
u(\varepsilon)^{2(1+\beta)} y^{2(1+\beta)} + 2\beta^2
\varepsilon^{-1}\log^2 (u(\varepsilon) y)} \sim
\varepsilon^\kappa L(\varepsilon).
\end{equation}
Combining all the parts of the sum yields
$\Gamma(\varepsilon)= \varepsilon^\kappa L(\varepsilon)$, that is the
first assumption of Theorem~\ref{t:FKconv} is satisfied with
$q(\varepsilon) = \varepsilon^{2/\kappa} \tilde L(\varepsilon)$.

Analogously, it can be shown that
%
\begin{eqnarray}
&&\sum_{N= u(\varepsilon)^{1/2}}^{u(\varepsilon)^{1+\beta/2}} N^{-1-\alpha}\bigl(1-
\hat\theta^N(\varepsilon)\bigr)^2
\nonumber
\\
&&\qquad\sim \int_{u(\varepsilon)^{-1/2}}^{u(\varepsilon)^{\beta/2} } \frac{4\beta^2 u(\varepsilon)^{4\beta-\alpha+2} y^{4\beta-\alpha+1}
\log^2(u(\varepsilon) y)}{
(u(\varepsilon)^{2(1+\beta)} y^{2(1+\beta)} + 2\beta^2
\varepsilon^{-1}\log^2 (u(\varepsilon) y))^2}
\\
&&\qquad\sim \varepsilon^{{(2+\alpha)}/ {(2\beta+2)}} L(\varepsilon),\nonumber
\end{eqnarray}
where $L$ is a slowly varying function at $\varepsilon=0$.
Hence,
%
\begin{eqnarray}
\sum_{N=1}^{u(\varepsilon)^{1/2}} N^{-1-\alpha}\bigl(1-
\hat\theta^N(\varepsilon)\bigr)^2&\le& c \sum
_{N=1}^{u(\varepsilon)^{1/2}}N^{-1-\alpha} \varepsilon
\frac{N^{4\beta+2}} {\beta^2 \log^2 N}
\nonumber
\\[-8pt]
\\[-8pt]
\nonumber
& =&\varepsilon^{{(4\beta+6+\alpha)}/{(4\beta+4)}}L(\varepsilon)\ll \varepsilon^{{(2+
\alpha)}/ {(2\beta+2)}},
\end{eqnarray}
where $L$ is slowly varying.
Similarly,
%
\begin{eqnarray}
&&\sum_{N=u(\varepsilon)^{1+\beta/2}}^{u(\varepsilon)^{1+\beta}} N^{-1-\alpha}\bigl(1-
\hat\theta^N(\varepsilon)\bigr)^2 \nonumber\\
&&\qquad\le \sum
_{N=u(\varepsilon)^{1+\beta/2}}^{u(\varepsilon)^{1+\beta}} c N^{-3-\alpha} \beta^2
\log^2 N
\\
&&\qquad\le\varepsilon^{{((1+\beta/2) (2+\alpha))}/ {(2+2\beta)}}L (
\varepsilon)\ll \varepsilon^{{(2+\alpha)}/ {(2+2\beta)}},\nonumber
\end{eqnarray}
where $L$ is slowly varying.
Finally,
%
\begin{eqnarray}
\sum_{N=u(\varepsilon)^{1+ \beta}}^{\infty} N^{-1-\alpha}\bigl(1-
\hat\theta^N(\varepsilon)\bigr)^2 &\le &c \varepsilon \sum
_{N=u(\varepsilon)^{1+ \beta}}^{\infty} N^{-1-\alpha}
\nonumber
\\[-8pt]
\\[-8pt]
\nonumber
&\le &c \varepsilon^{{(2\alpha+1)}/{(2\alpha)}} \ll \varepsilon^{{(2+\alpha)}/ {(2
\beta+2)}}.
\end{eqnarray}
Therefore, for some $L$ slowly varying function at $\varepsilon=0$.
%
\begin{equation}
\varepsilon^{-3}\E\bigl(\bigl(1-\hat{\pi}\bigl(q(\varepsilon)\bigr)
\bigr)^2\bigr)=\varepsilon ^{{(1-\alpha)}/{(1+\alpha)}}L(\varepsilon)
\xrightarrow{\varepsilon\to0}0.
\end{equation}
Hence, \eqref{e:FKsecondmoment} holds
claim (iii) for $\beta>0$ then follows from Theorem~\ref{t:FKconv}. This
completes the proof.
\end{pf*}

\begin{appendix}\label{app}
\section*{Appendix: Random measures} 

In this appendix, we collect frequently used notation and recall few
known theorems from the theory of random measures.

For any Polish topological space $E$, $\mathcal B(E)$ stands for the Borel
$\sigma$-field of $E$. We write $M(E)$ for the set of positive Radon
measures on $E$, that is, for the set of positive Borel measures on $E$
that are finite over compact sets. We will endow $M(E)$ with the topology
of vague convergence. $M_1(E)$ stands for the space of probability
measures over $E$ endowed with the weak convergence.

It is a known fact (\cite{rmeasures}, Lemmas 1.4 and 4.1), that the
$\sigma$-field
$\mathcal{B}(M(E))$ coincides with the field generated by the functions
$\{\mu\mapsto\mu(A)\dvtx A\in{\mathcal B}(E)$ bounded$\}$, as
well as
with the with the $\sigma$-field generated by the functions
$\{\mu\mapsto\int_{E} f \,d\mu\dvtx f \in C_0(E) \}$.

For every measure $\nu\in M((0,\infty))$, we define its Laplace
transform $\hat\nu\in C(\mathbb{R}_+)$ as
%
\setcounter{equation}{0}
\begin{equation}
\label{e:Laplace} \hat{\nu}(\lambda):=\int_{\mathbb{R}_+}\exp(-\lambda
t)\nu(dt).
\end{equation}

We recall that $\mu$ is a \emph{random measure on $\mathbb H$} defined on
a probability space
$(\tilde{\Omega},\tilde{\mathcal{F}},\tilde{\mathbb{P}})$ iff
$\mu\dvtx \tilde{\Omega}\to M(\mathbb H)$ is a measurable function from the
measurable space $(\tilde{\Omega},\tilde{\mathcal F})$ to the measurable
space $(M(\mathbb H),{\mathcal B}(M(\mathbb H)))$ (see \cite{rmeasures}).
Equivalently, $\mu$ is a random measure iff
$\mu(A)\dvtx \tilde{\Omega} \to\bar{\mathbb{R}}_+$ is a measurable function
for every $A\in\mathcal{B}(\mathbb H)$. The law induced by $\mu$ on
$M(\mathbb H)$ will be denoted $P_\mu$,
%
\begin{equation}
\label{e:Pmu} P_\mu= \tilde{\mathbb P} \circ\mu^{-1}.
\end{equation}

Let $\mu$ be a random measure on $\mathbb H$ defined on a
probability space $(\tilde{\Omega},\tilde{\mathcal{F}},\tilde{\Pb})$
and $f\dvtx \mathbb H\to\R_+$ be a measurable function. We define Laplace
transforms
%
\begin{equation}
L_\mu(f)= \tilde{\mathbb E} \biggl[\exp \biggl\{-\int
_{\mathbb H}f(t)\mu(dt) \biggr\} \biggr].
\end{equation}
The following proposition is well known (see Lemma~1.7 of
\cite{rmeasures}).

\begin{propositionn}\label{p:existence}
Let $(\Omega,\mathcal{F},\mathbb{P})$ be a probability space and let
$(\mu_\omega)_{\omega\in\Omega}$ be a family of random measures on
$(\tilde\Omega,\tilde{\mathcal F}, \tilde\Pb)$ indexed by
$\omega\in\Omega$. Then there
exists a probability measure $\mathcal{P}$ on $M(\mathbb H)$ given by
[recall \eqref{e:Pmu} for the notation]
%
\begin{equation}
\mathcal{P}(A)=\int_{\Omega}P_{\mu_\omega}(A)\mathbb{P}(d
\omega)\qquad \mbox{for each }A\in\mathcal{B}\bigl(M(\mathbb H)\bigr)
\end{equation}
if and only if the mapping $\omega\mapsto L_{\mu_{\omega}}(f)$ is
$\mathcal{F}$-measurable for each $f\in C_0(\mathbb H)$. The random
measure $\mu\dvtx \Omega\times\tilde\Omega\to M(\mathbb H)$ given by
$\mu(\omega,\tilde\omega)= \mu_\omega(\tilde\omega)$ whose
distribution is $\mathcal{P}$ is called the \emph{mixture of
$(\mu_{\omega})_{\omega\in\Omega}$ with respect to $\Pb$}.
\end{propositionn}

Let $\mu$ be a random measure. Denote
%
\begin{equation}
\label{e:Tmu} \mathcal{T}_\mu:=\bigl\{A\in\mathcal{B}(\mathbb H)
\dvtx \mu(\partial A)=0\ \tilde\Pb\mbox{-a.s.}\bigr\}.
\end{equation}
By a \emph{DC semiring} we shall mean a semiring
$\mathcal{U}\subset\mathcal{B}(\mathbb H)$ with the property that, for
any given $B\in\mathcal{B}(\mathbb H)$ bounded and any $\varepsilon>0$,
there exist some finite cover of $B$ by $\mathcal{U}$-sets of diameter
less than $\varepsilon$. The following is a known fact.

\begin{propositionn}[(Theorem~4.2 of \cite{rmeasures})]
\label{p:kalconv}
Let $\mu$ be a random measure and suppose that $\mathcal A$ is a DC
semiring contained in $\mathcal{T}_\mu$. To prove vague convergence in
distribution of random measures $\mu^\varepsilon$ to $\mu$ as
$\varepsilon\to0$, it suffices to prove convergence in distribution of
$(\mu^{\varepsilon}(A_i))_{i\leq k}$ to $(\mu(A_i))_{i\leq k}$ as
$\varepsilon\to0$ for every finite family $(A_i)_{i\leq k}$ of bounded,
pairwise disjoints sets in $\mathcal{A}$.
\end{propositionn}

Finally, we recall here the multidimensional individual ergodic theorem.
For its proof for square domains, see, for example, \cite{georgii}, Theorem~14.A5.
The proof can be easily adapted to rectangles.

\begin{theoremm}[(Multidimensional ergodic theorem)]
\label{t:ergodic}
Let $(X,\mathcal{G},Q)$ be a probability space and
$\Theta=(\theta_{i,j})_{(i,j)\in\Z^2}$ be a group of $Q$ preserving
transformations on $X$ such that
$\theta_{(i_1,j_1)}\circ\theta_{(i_2,j_2)}=\theta_{(i_1+i_2,j_1+j_2)}$.
Let $\mathcal I$ be the field of $\Theta$-invariant sets,
$a\leq0 < b$ and $c\leq0<d$ be real numbers, and
$\Delta_n=[\lfloor an\rfloor,\lfloor bn\rfloor]\times[\lfloor
cn\rfloor,\lfloor dn\rfloor]$.
Then, for any $Q$-measurable $f$ with $Q(|f|)<\infty$
%
\begin{equation}
\lim_{n\to\infty}\frac{1}{|\Delta_n|} \sum
_{i\in\Delta_n}f\circ\theta_i=Q(f|\mathcal{I}), \qquad Q
\mbox{-a.s.}
\end{equation}
\end{theoremm}
\end{appendix}
%
%





\printaddresses
\end{document}